\documentclass[10pt]{amsart}
\usepackage{geometry}
\usepackage{amssymb,amsopn}
\usepackage{graphicx,float}
\usepackage{enumerate,array,multirow,bm}
\usepackage{aliascnt,stmaryrd}

\newcommand{\R}{\mathbb{R}}
\newcommand{\C}{\mathbb{C}}
\newcommand{\Z}{\mathbb{Z}}

\newcommand{\rmd}{\mathrm{d}}
\newcommand{\vep}{\varepsilon}
\newcommand{\vphi}{\varphi}
\newcommand{\fl}[1]{\lfloor #1 \rfloor}
\newcommand{\ka}{\kappa}
\newcommand{\md}[1]{\llbracket #1 \rrbracket}

\usepackage[colorlinks=true, allcolors=blue]{hyperref}
\usepackage[capitalize,nameinlink,noabbrev]{cleveref}
\crefname{equation}{}{}
\usepackage{xcolor}

\theoremstyle{plain}
\newtheorem{theorem}{Theorem}[section]

\newaliascnt{lemma}{theorem}
\newtheorem{lemma}[lemma]{Lemma}
\aliascntresetthe{lemma}

\newaliascnt{proposition}{theorem}
\newtheorem{proposition}[proposition]{Proposition}
\aliascntresetthe{proposition}

\newaliascnt{corollary}{theorem}
\newtheorem{corollary}[corollary]{Corollary}
\aliascntresetthe{corollary}

\theoremstyle{remark}
\newaliascnt{remark}{theorem}
\newtheorem{remark}[remark]{Remark}
\aliascntresetthe{remark}

\numberwithin{equation}{section}
\numberwithin{figure}{section}
\numberwithin{table}{section}

\begin{document}
	
	\title[Exponential integrator for dispersive-dissipative equations]{Optimal error bounds on the exponential integrator for dispersive equations with highly concentrated potential}
	
	\author[G. Bal]{Guillaume Bal} \address{Departments of Statistics and Mathematics and Committee on Computational and Applied Mathematics, University of Chicago, Chicago, IL 60637} \email{guillaumebal@uchicago.edu}
	\author[C. Wang]{Chushan Wang$^\ast$} \thanks{$^\ast$Corresponding author.} \address{Committee on Computational and Applied Mathematics, University of Chicago, Chicago, IL 60637} \email{chushanwang@uchicago.edu}
	
	\begin{abstract}
		We study a one-dimensional linear dispersive equation of differential order $\ka \geq 2$ with concentrated potential of extension $\vep$ with $0 < \vep \ll 1$, featuring a competition between weak dispersion of strength $\vep^\alpha \  (0 \leq \alpha \leq \ka)$ and localization induced by the concentrated potential. 
        %%% The second-order case $\kappa=2$ is motivated from equations satisfied by statistical moments of solutions to an It\^o-Schr\"odinger model of wave propagation in random media.  %%% Not sure we want this here for foundational result
        We first obtain precise regularity estimates of the exact solution in terms of $\vep$. We then apply a natural first-order exponential integrator with step size $\tau$ to discretize the equation, and establish an optimal  error bound of the form $O_{L^\infty}(\tau \vep^\beta)$ (up to logarithmic factors in $\tau$ and $\vep$).
        %where $\beta = \min\{1 + (\ka-1)\alpha/\ka, 2 (1- \alpha/\ka) \} \geq 0$
        Salient features of the result are: (i) error bounds are not only uniform in $\vep$ but improve as $\vep \rightarrow 0$; and (ii) no restriction on $\tau$ in terms of $\vep$.
        %; hence, the error bound is valid even when $\tau \gg \vep$. 
        The analysis combines iterated Duhamel's expansions %normal form
        and a transformation that exploits cancellations in oscillatory phases that cannot be obtained directly from regularity estimates of the exact solution. We also show that other classical numerical schemes, such as Lie or centered splitting schemes and low regularity integrators, fail to display optimal rates of convergence. Extensive numerical results are presented and confirm the theoretical error estimates.% and suggest that the rate is optimal in both $\tau$ and $\vep$. 
	\end{abstract}

    
	
	\maketitle

%\gb{This is great. Here are a first few comments. The potential is highly concentrated and dissipative (not oscillatory which is reserved to periodic-type oscillations). 
%\\
%On the justification of the model in the appendix: one has to be careful that some second-order moments can be calculated explicitly, namely those with $\Delta_x-\Delta_y$. It is really for the other models that the reduction is interesting; but they do not come from second-order wave moments. We may want to say a few words about the fourth moment problem, which was one of the main motivations for the work, and say that the analysis is too complicated without preliminary analyzes, which we carry out in this paper. 
%\\
%Do we use regularity results of the first part in the analysis of the numerical scheme? It seems not as we need to compute errors manually without ever relying on regularity estimates. This is fine but I believe we need to find a way to stress that point.
%\\
%I also believe it would be great to have a comparison with other schemes such as time splitting, and show both theoretically (poor integration capabilities) and numerically while they fail to be asymptotically preserving.
%}
    
	\section{Introduction}
	We consider the following linear dispersive equation of order $\kappa \geq 2 \ (\kappa \in \mathbb{N})$ with potential:
	\begin{equation}\label{eq:mu}
		\left\{
		\begin{aligned}
			&\partial_z \mu = i \vep^\alpha D_\kappa \mu + R(\frac{x}{\vep}) \mu, && z \geq 0, \quad x \in \R, \\
			&\mu(0, x) = \mu_0(x), && x \in \R, 
		\end{aligned}
		\right.
	\end{equation}
	where $\mu = \mu(z, x)$ is a complex valued function, $0 < \vep \leq 1$, $0 \leq \alpha \leq \ka$, $R$ is a real-valued potential, and
	\begin{equation}\label{eq:P_def}
		D_\ka := -P(- i \partial_x), \qquad P(y) := \sum_{0 \leq 2j < \ka} d_{\ka - 2j} y^{\ka - 2j}, \quad y \in \R, \quad d_{\ka - 2j} \in \mathbb{R}, \quad d_\ka = 1.
	\end{equation}
%    \gb{Do we need all such coefficients non-negative?}
    Here, $P$ is a polynomial of order $\ka$, which consists of either even order terms or odd order terms with real-valued coefficients. In particular, when $\ka = 2$, then $D_\ka = \partial_{xx}$ and \cref{eq:mu} reduces to the following Schr\"odinger-type equation: 
	\begin{equation}\label{eq:schrodinger}
		\partial_z \mu = i \vep^\alpha \partial_{xx} \mu + R(\frac{x}{\vep}) \mu, \qquad \mu(0, x) = \mu_0(x), \qquad z \geq 0, \quad x \in \R. 
	\end{equation}
	Note, however, that the real-valued potential introduces dissipation. For simplicity of notation, we denote $R_\vep = R_\vep(x) = R(x/\vep)$ in the rest of the paper.
	
	Equation \cref{eq:schrodinger} arises in the analysis of statistical moments of an It\^o-Schr\"odinger equation in the scintillation regime \cite{bal2025,bal2025long,gar2014,gar2016}, which finds applications in the modeling of long distance beam propagation through a highly oscillatory random medium. 
%	The scintillation scaling is characterized by a wide initial beam, a long propagation distance, and weak medium fluctuations. 
	In this context, the potential $R$ is the correlation function of the random medium in the transverse direction. More precisely, \cref{eq:schrodinger} arises from the equation of the fourth moment of the It\^o-Schr\"odinger equation (see Appendix \ref{append:A1} for more details), which is a fundamental quantity in wave propagation through complex media. While the fourth moment equation is the primary object of practical interest, models of the form \cref{eq:mu} inherit many of its technical difficulties. In particular, they display terms that prevent many standard numerical methods such as time-splitting methods, from achieving optimal accuracy when the step size $\tau$ is large compared to $\vep$ (or some power of $\vep$). 
    %Hence, we take \cref{eq:mu} as the first step to study the equation of the fourth moment, while the results can also be applied to other moments such as the second- and third-order moments. 
    Moreover, \cref{eq:mu} in its general form can also be derived from higher-order dispersive counterparts of the moment equations of the It\^o-Schr\"odinger model (see Appendix \ref{append:A2} for a detailed derivation). 
	
	In addition to the aforementioned applications, \cref{eq:mu} itself is of independent interest as it features a competition between the dispersion and the concentration resulting from the potential. In particular, when $\alpha = 1$ in \cref{eq:schrodinger}, it is shown that the solution $\mu$ is $O_{L^\infty}(\vep^\frac{1}{2})$ away from the free solution (i.e. the solution with $R = 0$) \cite{bal2025}. However, no such estimates hold for the derivatives of $\mu$, nor are there corresponding results for the general case \cref{eq:mu}. Moreover, understanding how the competition between dispersion and potential-induced concentration effects influences the dynamics remains challenging.
    
    We plot in \cref{fig:soln_diff_alpha} the magnitude of the solution $\mu$ to \cref{eq:schrodinger} at $z=8$ with different $\alpha$. When $\alpha = 3/4$, dispersion dominates and the solution remains flat for very small $\vep$, while when $\alpha = 4/3$, concentration effects dominate and the solution is highly oscillatory. When $\alpha = 1$, these two effects (almost) balance each other. 
	
	\begin{figure}[htbp]
		\centering
		{\includegraphics[width=0.3\textwidth]{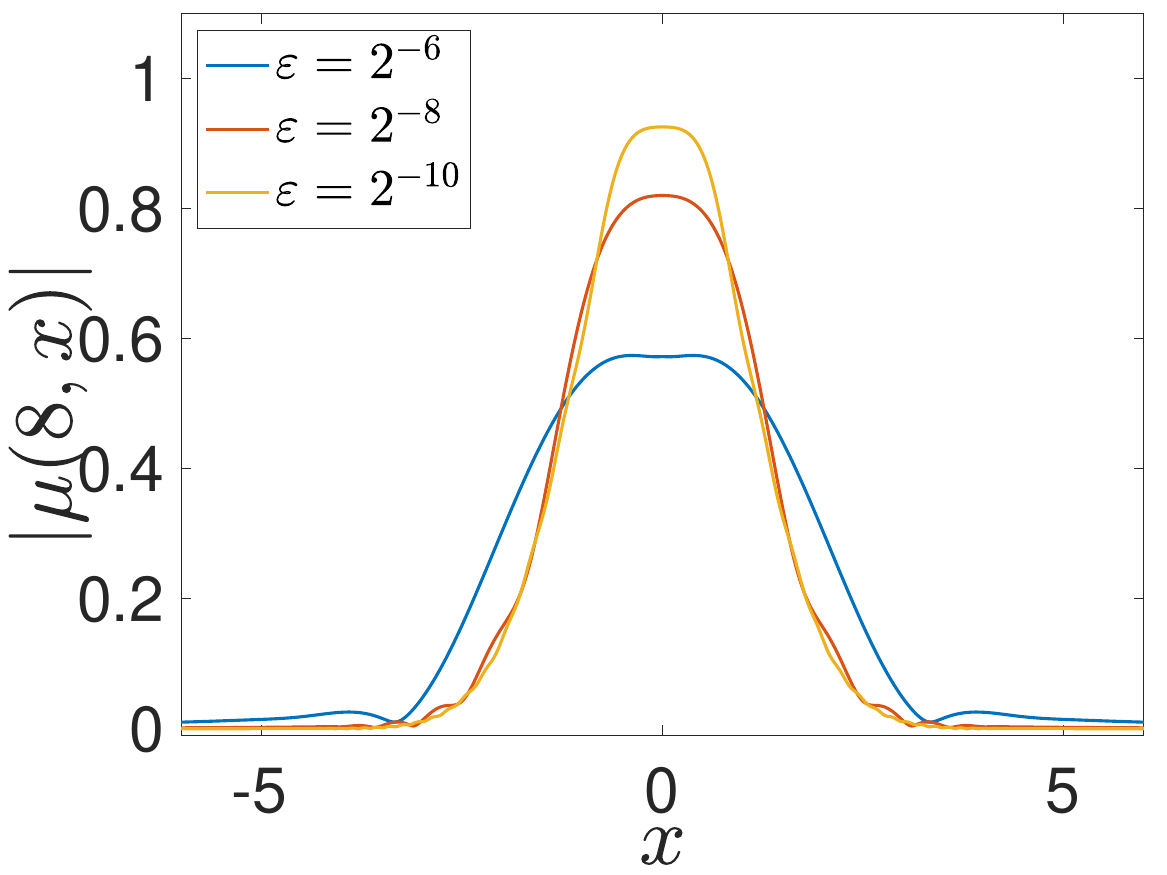}}\hspace{1em}
		{\includegraphics[width=0.3\textwidth]{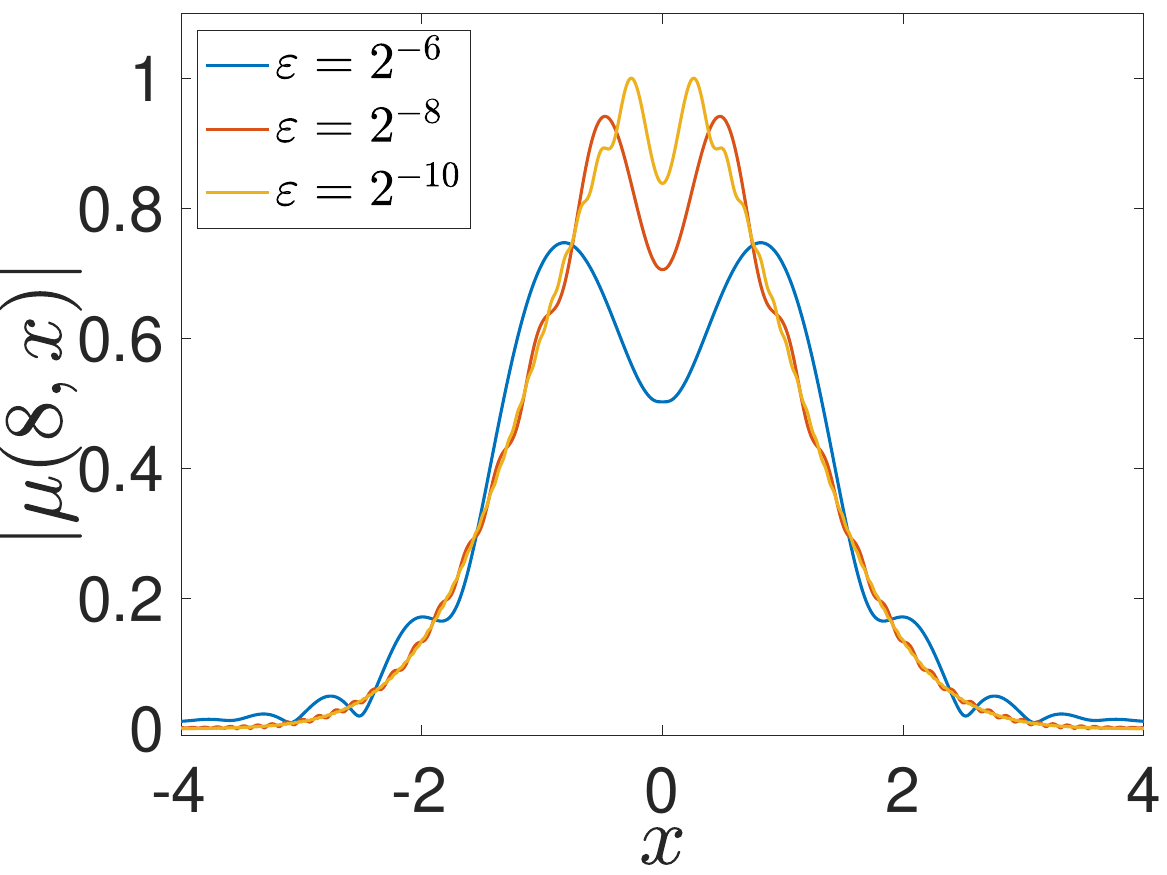}}\hspace{1em}
		{\includegraphics[width=0.3\textwidth]{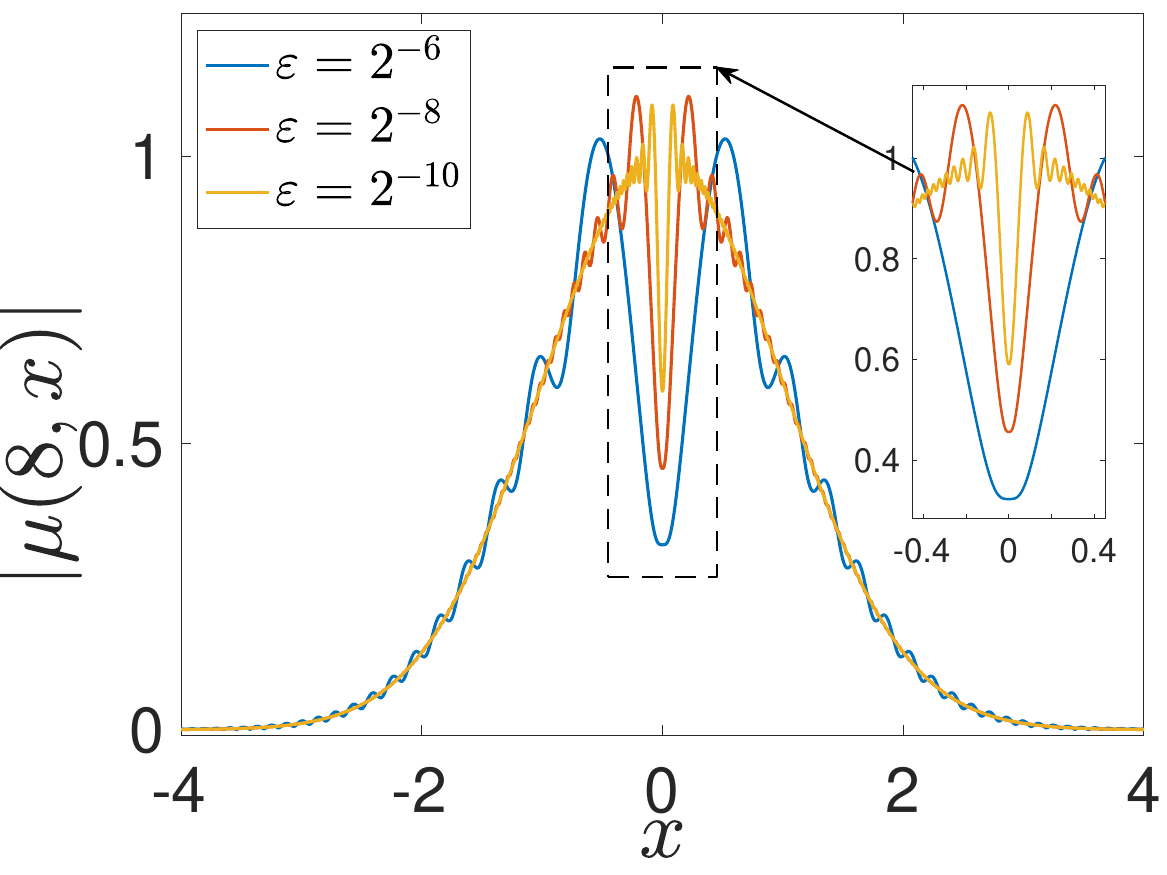}}
		\caption{Plots of the solution $\mu$ at $z=8$ of various $\vep$ for \cref{eq:mu} with $D_\ka = D_2 = \partial_{xx}$ and $\alpha = 3/4$ (left), $\alpha = 1$ (middle), and $\alpha = 4/3$ (right)}
		\label{fig:soln_diff_alpha}
	\end{figure} 

    Capturing the oscillations generated by \cref{eq:mu} numerically is an even more challenging problem. When $\vep = 1$, most standard numerical methods for linear dispersive equations can be directly applied; for instance time-splitting methods \cite{lubich2000,BBD,bao2002,bal2004,zhao2020,blanes2024,liu2013,bal2025splitting}, exponential integrators \cite{expint2010,symmetric2008,book_lubich,ETD2002}, and low regularity integrators (LRIs) \cite{bronsard2022,LRI_1,LRI_general}. The precise form of these methods applied to \cref{eq:mu} are given in \cref{eq:EWI_mu} below for the exponential integrator (which displays optimal rates of convergence) and in Appendix \ref{append:A3} for the other models (which all turn out to display non-optimal convergence rates). The main objective of this work is to analyze the performance of these numerical methods and understand why they differ significantly when $\vep \ll1$. 
%	In the following,  (the precise formulation of these methods applied to the model problem \cref{eq:schrodinger} can be found in \cref{eq:EWI_mu} and \cref{eq:LT,eq:Strang,eq:LRI}). 

	As a direct illustration, we apply the three discretizations to the model problem \cref{eq:schrodinger} with $\alpha = 1$ and $\vep = 2^{-10}$, and plot the obtained numerical solutions with different step sizes in \cref{fig:soln_diff_method}. In particular, time-splitting methods work reasonably only when the step size $\tau \lesssim \vep$ (or $\vep^{\ka - \alpha}$ in the general case \cref{eq:mu}), while severe (spurious) oscillations are observed in the numerical solution when $\tau \gg \vep$. In comparison, the exponential integrator and the low regularity integrator perform significantly better, producing accurate results even when $\tau \gg \vep$, with the exponential integrator being more accurate than the low regularity integrator. In fact, our theory will show that errors of the exponential integrator are not only uniform in $\vep$ but even improve as $\vep \rightarrow 0$. While numerical results also indicate improved errors for the LRI as $\vep \rightarrow 0$, the rate is no longer optimal. 
    %\gb{How about the low regularity integrator?} 
    We remark here that the splitting methods work perfectly for \cref{eq:mu} when $\alpha \geq \ka$, which can be seen using standard commutator estimates \cite{lubich2000,lubich2008}. A comprehensive comparison of these methods is carried out in \cref{sec:comp}.
	
	\begin{figure}[htbp]
		\centering
		{\includegraphics[width=0.3\textwidth]{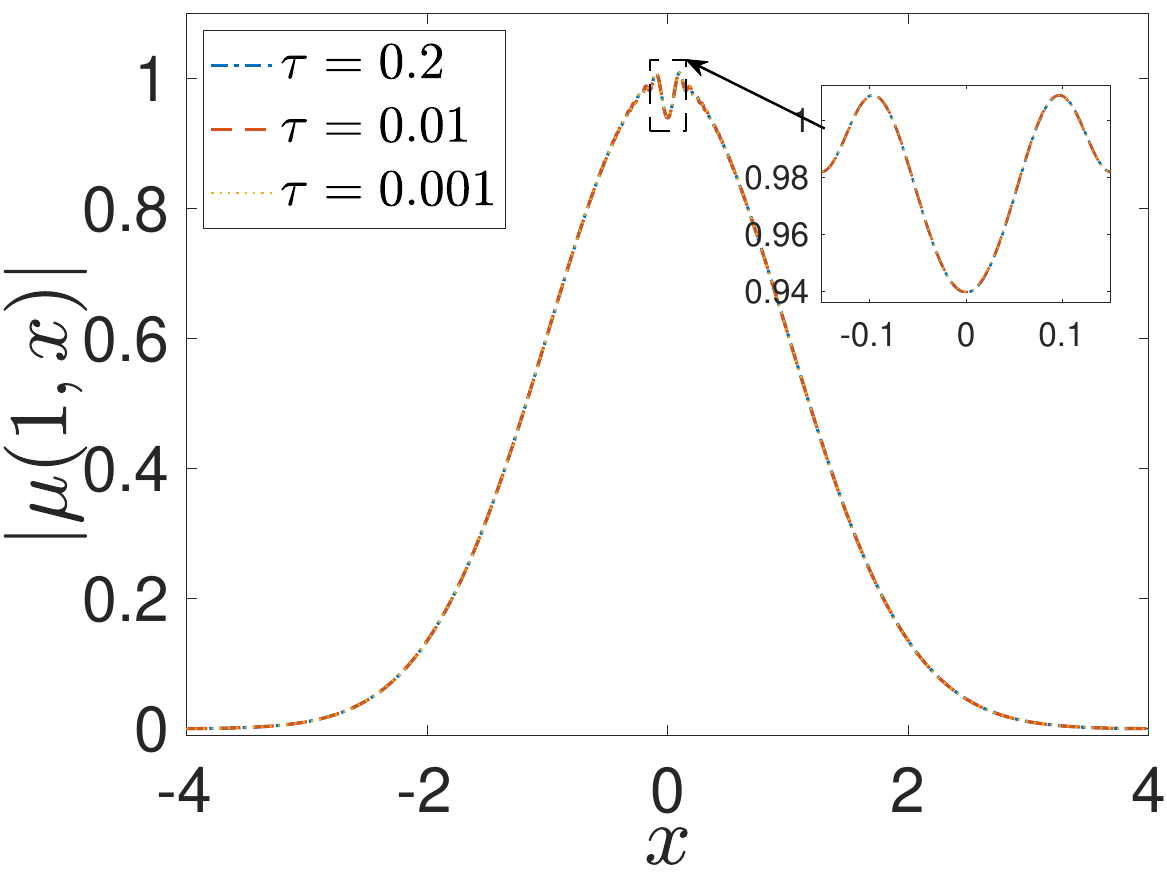}}\hspace{1em}
		{\includegraphics[width=0.3\textwidth]{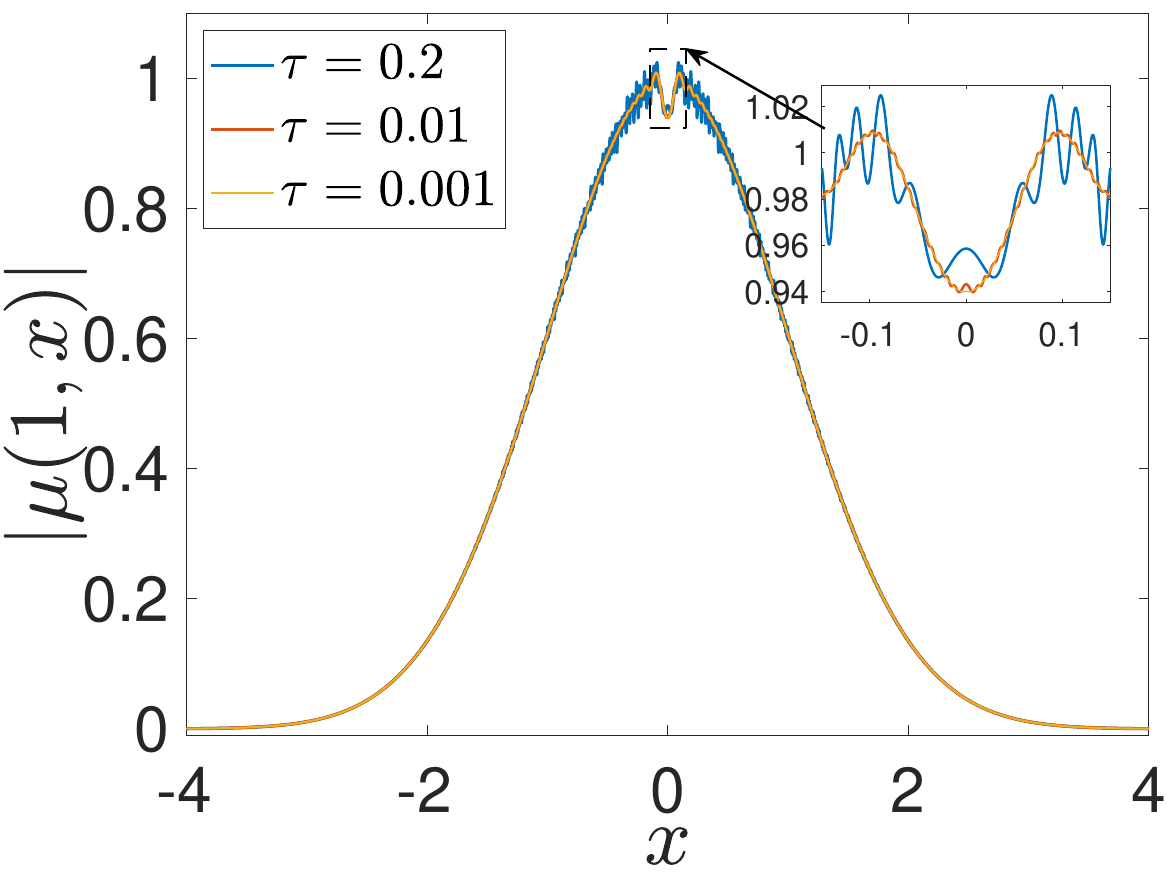}}\hspace{1em}
		{\includegraphics[width=0.3\textwidth]{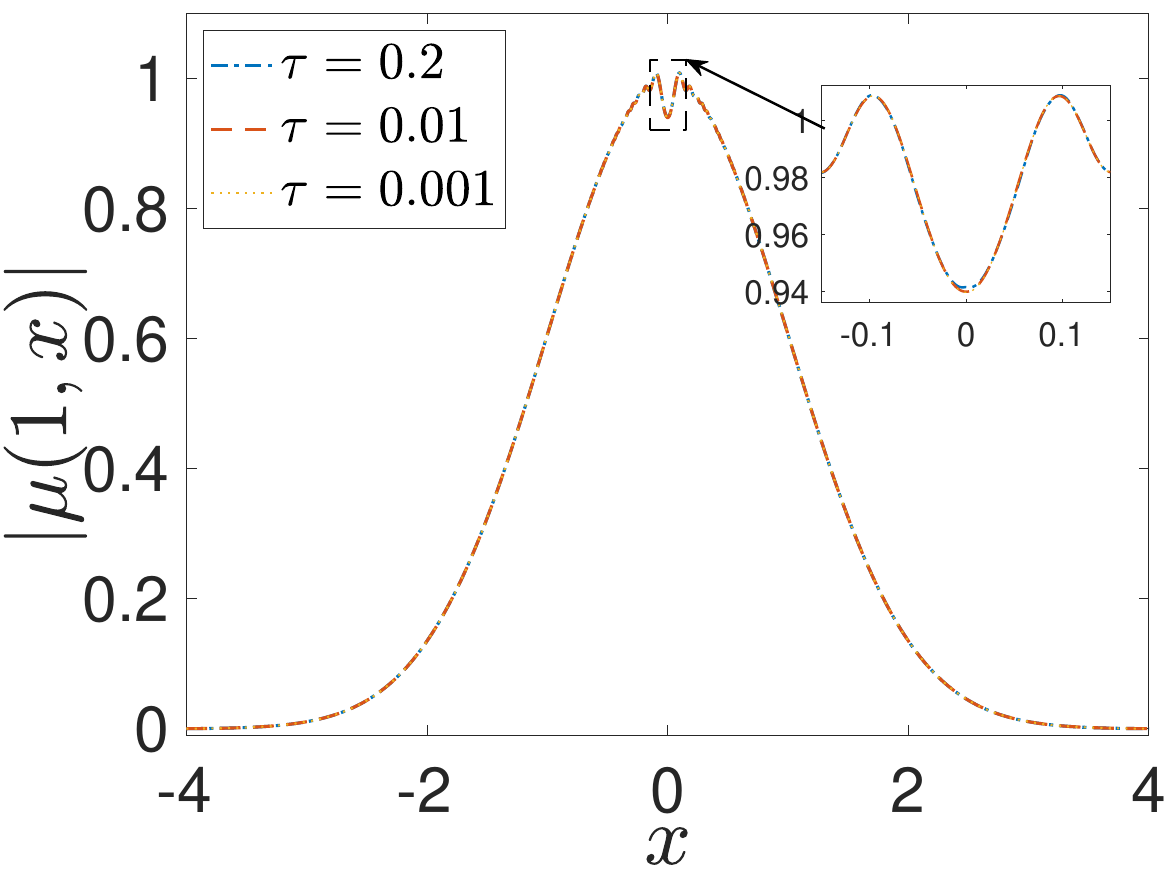}}
		\caption{Plots of the numerical solution at $z=1$ obtained by the exponential integrator (left), the splitting method (middle), and the low regularity integrator (right) for \cref{eq:schrodinger} with $\alpha = 1$ and $\vep = 2^{-10}$}
		\label{fig:soln_diff_method}
	\end{figure}

    Our theoretical analysis of solutions to \cref{eq:mu} is summarized in \cref{thm:exact} below. In particular, we show that the exact solution $\mu$ is $O_{L^\infty}(\vep^{1-\alpha/\ka})$ away from the solution to the free equation. 
    
	%In this work, we address both the analytical and numerical questions mentioned above. On the analytical side, we derive estimates of the exact solution and its derivatives with explicit $\vep$ dependence (see \cref{thm:exact}). In particular, we show that the exact solution $\mu$ is $O_{L^\infty}(\vep^{1-\alpha/\ka})$ away from the solution to the free equation. 

    To discretize \cref{eq:mu}, we use a standard first-order exponential integrator, which exhibits excellent performance as we illustrated above. More precisely, let $\mu^n$ be the numerical approximation of $\mu(z_n)$ with $z_n = n\tau$ for $n=0, 1, \cdots$ and $\tau>0$ the step size defined iteratively as
	\begin{equation}\label{eq:EWI_mu}
		\begin{aligned}
			&\mu^{n+1} = e^{i\tau\vep^\alpha D_\kappa}\mu^n + \tau\vphi_1(i \tau \vep^\alpha D_\kappa) (R_\vep \mu^n), \quad n \geq 0, \\
			&\mu^0 = \mu_0, 
		\end{aligned}
	\end{equation}
	where $\vphi_1(z) :=(e^z - 1)/z$ for $z \in \C$, and $\vphi_1(i \tau \vep^\alpha D_\ka)$ is the Fourier multiplier operator with symbol $\vphi_1(-i \tau \vep^\alpha P(\xi))$, where $\xi \in \R$ is the dual (Fourier) variable to $x$. We establish an optimal error estimate for the exponential integrator \cref{eq:EWI_mu} in \cref{thm:main} below.
    %
    %For the exponential integrator \cref{eq:EWI_mu}, we establish a rigorous and comprehensive error estimate  in \cref{thm:main}, 
    It shows that, without any coupling conditions on $\tau$ and $\vep$, the exponential integrator \cref{eq:EWI_mu} converges in the $L^\infty$-norm at $O(\tau \vep^\beta)$ (up to logarithmic factors in $\tau$ and $\vep$), where $\beta = \min\{1 + (\ka-1)\alpha/\ka, 2 (1- \alpha/\ka) \}$. Notably, we have $\beta \geq 1 - \alpha/\ka$ when $0 \leq \alpha \leq \ka$. Since the exact solution itself differs from the free solution with an error of $O_{L^\infty}(\vep^{1-\alpha/\ka})$, it might appear surprising at first that the error is uniform in $\tau$ and gains an additional order in $\vep$ with a larger rate $\beta$. 

\medskip
	
    We now briefly present the main ideas of the proof of these results (see \cref{sec:main_results} for a more detailed explanation). In \cref{eq:mu}, the highly concentrated potential $R_\vep$ introduces a rapidly oscillating phase $e^{i z \Phi}$ in the Fourier domain where $\Phi$ depends on $\vep$. Both the regularity estimates and the error analysis rely crucially on the use of this oscillatory phase to gain appropriate powers of $\vep$ using a normal form transformation (an integration by parts in $z$). To perform the normal form transformation, we need to identify resonant frequencies. This is relatively easy at the continuous level where the resonance occurs only when $\Phi = 0$. However, for the exponential integrator and its discretization in the $z$-direction, the oscillatory phase takes the form $e^{iz_k\Phi}$ with $z_k = k \tau$. This introduces a much larger `numerical' resonant set $\Phi = 2 \pi \mathbb{Z}/\tau$ whose detailed analysis is necessary in order to avoid introducing any $\vep$-dependent step size restriction that are not observed in the numerical experiments.
    %, we must control the error in this ``numerical resonance" regime. 
    Another crucial technique of the proof lies in the use of an iterated Duhamel expansion for both the exact solution and the numerical solution. This allows us to capture all phase interactions and cancellations without appealing to regularity results for the exact solution. We observe that directly using regularity estimates inevitably introduces loss in phase cancellations, leading to loss of optimality in power in $\vep$. Our analysis shows that capturing phase cancellations in both first-order and higher-order terms in the Duhamel expansion is necessary.

	The rest of the paper is organized as follows. In \cref{sec:main_results}, we present the main results and the ideas of the proof. \cref{sec:2,sec:3} are devoted to the regularity estimates of the exact solution and the error estimate of the exponential integrator, respectively. Numerical results are displayed in \cref{sec:numer_resutls}, where we also compare the exponential integrator with other popular numerical schemes. Finally, some conclusions are drawn in \cref{sec:conclusion}. We conclude this section by introducing some notation. 
	
	\subsection*{Notation.} For a function $\phi = \phi(x)$, we denote by $\widehat{\phi}$ the standard Fourier transform with respect to the variable $x \in \R$ with dual variable denoted by $\xi \in \R$. Define the space $X^p$ for $p \geq 0$ as 
	\begin{equation}
		X^p: = \{\phi \in \mathcal{S}': (1 + |\xi|^p)\widehat{\phi}(\xi) \in L^1(\R)\}. 
	\end{equation}
	For simplicity, we denote $X^0$ by $X$, and equip it with the norm $\| \cdot \|$ defined as the $L^1$-norm of the Fourier transform, i.e. $\| \phi \| : = \int_{\R} |\widehat{\phi}(\xi)| \rmd \xi$ for $\phi \in X$. We recall some basic properties of the space $X$: 
	\begin{enumerate}[(i)]
		\item $X$ equipped with $\| \cdot \|$ is an algebra with
		\begin{equation*}
			\| \phi_ 1 \phi_2 \| \leq \| \phi_1 \| \| \phi_2 \|, \qquad \phi_1, \phi_2 \in X. 
		\end{equation*}
		\item The linear group $e^{i z \vep^\alpha D_\ka} \ (z \in \R)$ is an isometry on $X$:
		\begin{equation*}
			\| e^{i z \vep^\alpha D_\ka} \phi \| = \| \phi \|, \quad \phi \in X. 
		\end{equation*}
		\item The norm $\| \cdot \|$ is invariant under rescaling: for any $\lambda \neq 0$, 
		\begin{equation*}
			\| \phi(\lambda \cdot) \| = \| \phi \|, \quad \phi \in X. 
		\end{equation*}
		\item $X$ is continuously embedded into $L^\infty(\R)$. 
	\end{enumerate}
	We denote by $C$ a generic positive constant independent of the step size $\tau$ and the small parameter $\vep$. The notation $A \lesssim B$ is used to mean the existence of a generic constant $C>0$ such that $|A| \leq CB$. For a function $\phi = \phi(\xi, \xi_1, \cdots, \xi_m)$ with $m \geq 2$ and $\xi$ fixed, we define the integral
	\begin{equation*}
		\int_{\xi_1 + \cdots + \xi_m = \xi} \phi(\xi, \xi_1, \cdots, \xi_m) \rmd \xi_1 \cdots \rmd \xi_{m-1} := \int_{\R^{m-1}} \phi(\xi, \xi_1, \cdots, \xi_{m-1}, \xi - (\xi_1 + \cdots + \xi_{m-1})) \rmd \xi_1 \cdots \rmd \xi_{m-1}. 
	\end{equation*}

	\section{Main results and ideas of the proof}\label{sec:main_results}
	In this section, we present the two main results of this paper and explain the main ideas of the proof. Throughout this paper, we adopt the following assumptions on the potential $R$ and the initial datum $\mu_0$: 
	\begin{equation}\label{eq:assumption}
		\widehat{R} \in L^1(\R) \cap L^\infty(\R), \qquad \mu_0 \in X^\ka. 
	\end{equation}
	
	\subsection{Main results}
	The first result concerns the regularity of the exact solution $\mu$ in \cref{eq:mu}. Let $T>0$ be fixed in the rest of the paper. 
	\begin{theorem}\label{thm:exact}
		Under assumptions \cref{eq:assumption}, we have, for $0 \leq z \leq T$, 
		\begin{align}
			&\| \partial_z \mu(z) \| + \vep^\alpha\| \partial_x^\ka \mu(z) \| \lesssim 1,  \label{eq:exact_1}\\
			&\| \partial_x^j(\mu(z) - e^{iz\vep^\alpha D_\ka}\mu_0) \| \lesssim 
			\left\{
			\begin{aligned}
				&\vep^{1 - \frac{1+j}{\kappa}\alpha}, && 0 \leq j \leq \kappa - 2, \\
				&\vep^{1 - \alpha} |\ln \vep|, && j = \kappa - 1.
			\end{aligned}
			\right. \label{eq:exact_2}
		\end{align}
	\end{theorem}
	
	\begin{remark}\label{rem:balance}
		From \cref{thm:exact}, we see that when $\alpha = \ka / 2$, the first-order derivative $ \partial_x \mu(z) $ is uniformly bounded in $\vep$, except for the case $\ka = 2$ where $ \| \partial_x \mu(z) \|$ grows as $\ln \vep^{-1}$. This suggests that the dispersion from $i \vep^\frac{\ka}{2} D_\ka$ and the concentration from $R(x/\vep)$ balance each other (see \cref{fig:soln_diff_alpha} for an example of $\ka = 2$). 
	\end{remark}
	
	\begin{remark}
		If the initial data $\mu_0$ and the potential $R$ are sufficiently smooth, we can obtain estimates on higher order derivatives in $z$ and $x$ similarly by using \cref{eq:exact_1}. In particular, we have
		\begin{equation*}
			\| \partial_z^{1+j} \mu(z) \| \lesssim \vep^{- j (\ka-\alpha)}, \qquad \| \partial_x^{\ka + j} \mu(z) \| \lesssim \vep^{-(\alpha + j)}, \qquad j \geq 0. 
		\end{equation*}
		Roughly speaking, each additional $x$-derivative yields $\vep^{-1}$ from $\partial_x R_\vep$, and each additional $z$-derivative yields $\vep^{-(\ka-\alpha)}$ from $\vep^\alpha D_\ka R_\vep$.  
	\end{remark}
	
	We now consider the error estimate of the exponential integrator \cref{eq:EWI_mu} applied to \cref{eq:mu}. Concerning the decay of $\widehat{R}$ and $\widehat{\mu_0}$, we further assume that there exists $\sigma > 0$ such that for any sufficiently large $L$, we have 
	\begin{equation}\label{eq:assumption_2}
		\int_{|\xi| \geq L} |\widehat{R}(\xi)| \rmd \xi \leq \frac{1}{L^\sigma}, \qquad \int_{|\xi| \geq L} |\widehat{(\partial_{x}^\ka\mu_0)}(\xi)| \rmd \xi \leq \frac{1}{L^\sigma}, \qquad \widehat{(\partial^\ka_x \mu_0)} \in L^\infty(\R). 
	\end{equation} 
	When $0 \leq \alpha < \frac{2\ka}{\ka + 2}$, to obtain optimal error bounds, we need an additional decay assumption on $\widehat{\mu_0}$: 
	\begin{equation}\label{eq:assumption_3}
		\int_{\R} |\xi|^{2\ka-1}|\widehat{\mu_0}(\xi)| \rmd \xi \lesssim 1. 
	\end{equation}
	These give the following optimal error estimates.
	\begin{theorem}\label{thm:main}
		Assume \cref{eq:assumption,eq:assumption_2}, and assume \cref{eq:assumption_3} if $0 \leq \alpha < 2\ka/(\ka + 2) $. For all $0<\tau<\tau_0$ and $0<\vep<\vep_0$ with $\tau_0, \vep_0 > 0$ sufficiently small, we have
		\begin{equation*}
			\| \mu(z_n) - \mu^n \| \lesssim \tau \vep^{1 + \frac{\ka - 1}{\ka}\alpha}(1 + \tau^{1-\frac{1}{\ka}} |\ln \vep|)+ \tau \vep^{2- \frac{2}{\kappa}\alpha} |\ln \vep| \times \left\{
			\begin{aligned}
				&1, && \ka \geq 3, \\
				&(|\ln \vep| + |\ln \tau|)^2, && \ka = 2, 
			\end{aligned}
			\right.
			, \quad 0 \leq n \leq \frac{T}{\tau}. 
		\end{equation*}
	\end{theorem}
	
    \begin{remark}
        When $\vep \geq \vep_0$, following standard arguments with the regularity estimate \cref{eq:exact_1}, one can easily obtain an error bound of order $O(\tau)$ uniform in $ \vep_0 \leq \vep \leq 1$ under assumptions \cref{eq:assumption}. 
    \end{remark}
	
	Heuristically, \cref{thm:main} implies that the convergence order of the exponential integrator is $O(\tau \vep^\beta)$ for
	\begin{equation*}
		\beta = \min\{1 + \frac{\ka - 1}{\ka} \alpha, \, 2 - \frac{2}{\ka} \alpha\}, 
	\end{equation*}
	where the different values in the above minimum correspond to different phase cancellation structures. 
	We stress again that there is no requirement on the relationship between $\tau$ and $\vep$ in the error bound. The error estimate remains valid when $\tau \gg \vep$. In fact, this is one of the main difficulties in the proof of \cref{thm:main}. 
	
	 \begin{remark}\label{rem:non-optimal}
		If removing the additional decay assumption \cref{eq:assumption_3} in \cref{thm:main}, the result remains the same when $\frac{2\ka}{\ka+2} \leq \alpha \leq \ka$. When $0\leq\alpha<\frac{2\ka}{\ka+2}$, we would lose a small power of $\vep$ (depending on $\alpha$) in the first bound, and the loss of $\vep$ is compensated by a gain of $\tau$:
		\begin{equation}\label{eq:non-optimal}
			\| \mu(z_n) - \mu^n \| \lesssim \tau \vep^{1 + \frac{\ka - 1}{\ka}\alpha}(1 + (\frac{\tau}{\vep^\frac{\alpha}{\ka-1}})^{(\frac{\ka-1}{\ka})^2} |\ln \vep|) + \tau \vep^{2- \frac{2}{\kappa}\alpha} |\ln \vep| \times \left\{
			\begin{aligned}
				&1, && \ka \geq 3, \\
				&(|\ln \vep| + |\ln \tau|)^2, && \ka = 2. 
			\end{aligned}
			\right.
		\end{equation}
		%		which follows directly from \cref{eq:prop_est1_disc_2} below.  
	\end{remark}
	
	\subsection{Ideas and difficulties of the proofs of \cref{thm:exact,thm:main}}
	%In this subsection, we explain the ideas and difficulties of the proof of  with more details.
	%
	We first introduce the twisted (phase conjugated) function  
	\begin{equation}\label{eq:psi_def}
		\psi(z) := e^{-i z \vep^\alpha D_\kappa} \mu(z), \qquad z \geq 0.
	\end{equation}
	%In most of the paper, instead of directly working with $\mu$ in \cref{eq:mu}, we will work with $\psi$ defined above. 
    Plugging \cref{eq:psi_def} into \cref{eq:mu}, $\psi$ satisfies
	\begin{equation}\label{eq:psi}
		\partial_z \psi = e^{-i z \vep^\alpha D_\kappa} R_\vep e^{i z \vep^\alpha D_\kappa} \psi, \quad \psi(0) = \mu_0, \qquad z \geq 0. 
	\end{equation}
	Integrating \cref{eq:psi} in $z$, we get
	\begin{equation}\label{eq:psi_int}
		\psi(z) = \psi(0) + \int_0^z e^{- i s \vep^\alpha D_\ka} R_\vep e^{i s \vep^\alpha D_\ka} \psi(s) \rmd s. 
	\end{equation}
	Taking the Fourier transform of \cref{eq:psi_int}, noting that $\widehat{R_\vep}(\xi) = \vep\widehat{R}(\vep\xi)$, we obtain by the change of variable $\vep \xi_1 \rightarrow \xi_1$, 
	\begin{align}\label{eq:psi_hat_int}
		\widehat{\psi}(z, \xi) 
		%		&= \int_{\xi_1 + \xi_2 = \xi} \int_0^z e^{i s \vep^\alpha (P(\xi) - P(\xi_2))} \widehat{R_\vep}(\xi_1) \widehat{\psi}(s, \xi_2) \rmd s \rmd \xi_1 \notag \\
		&= \int_{\frac{\xi_1}{\vep} + \xi_2 = \xi} \int_0^z e^{i s \vep^\alpha (P(\xi) - P(\xi_2))} \widehat{R}(\xi_1) \widehat{\psi}(s, \xi_2) \rmd s \rmd \xi_1, 
	\end{align}
	which yields a highly oscillatory phase $e^{is\Phi}$ where
	\begin{equation}\label{eq:phase}
		\Phi(\xi_1, \xi_2) := \vep^\alpha (P(\frac{\xi_1}{\vep} + \xi_2) - P(\xi_2)) = O(\frac{1}{\vep^{\ka - \alpha}}), \qquad \xi_1, \xi_2 \in \R, 
	\end{equation}
	satisfying a crucial algebraic structure described in \cref{lem:phase} in the appendix.

	To prove \cref{thm:exact}, it suffices to capture the first-order phase interaction in \cref{eq:psi_int} by applying a normal form transformation (see \cref{prop:est1_continuous}), which yields a cancellation of order $O(\vep^{1 - \frac{\alpha}{\ka}})$ in the norm $\| \cdot \|$. For the error estimate of the exponential integrator \cref{eq:EWI_mu} in \cref{thm:main}, the propagation of numerical approximation errors not only improves this first-order phase interaction, but also introduces higher-order phase interactions with stronger cancellation effects of order $O(\vep^{2(1-\frac{\alpha}{\ka})})$. To capture such effects, which would be lost by merely using regularity estimates of the exact solution, we need to use an iterated Duhamel expansion expressing the objects of interest directly as a functional of the initial data.
    
 %   However, both of the first-order and higher-order phase interactions could be lost in the error estimate when using the regularity estimates of the exact solution. Instead, we need to use a iterated Duhamel's expansion in terms of the initial data to capture all these phase interactions. 
	
	To explain why the Duhamel expansion is necessary, we present an informal error analysis, which also shows why more standard arguments using regularity estimates (as in, e.g., the recently developed regularity compensation oscillation technique \cite{bao2023_longtime,bao2022_longtime,bao2023_improved,bao2024sEWI} or \cite{LRI_strichartz,bao2025}) would fail in our setting. The local truncation error of the exponential integrator \cref{eq:EWI_mu} at the step $z_k$ is given by
	\begin{equation}
		\mathcal{L}^k := \int_0^\tau e^{i(\tau - s)\vep^\alpha D_\ka} R_\vep (\mu(z_k + s) - \mu(z_k)) \rmd s, \qquad k \geq 0, 
	\end{equation}
	which accumulate up to the final step $z_n$ via 
	\begin{align}\label{eq:sum_local}
		\sum_{k=0}^{n-1} e^{i(n - 1- k)\tau \vep^\alpha D_\ka} \mathcal{L}^k 
		&= e^{i z_n \vep^\alpha D_\ka} \sum_{k=0}^{n-1} e^{- i z_k \vep^\alpha D_\ka} \int_0^\tau e^{- i s \vep^\alpha D_\ka} R_\vep (\mu(z_k + s) - \mu(z_k)) \rmd s \notag \\
		&\approx \tau e^{i z_n \vep^\alpha D_\ka} \int_0^{z_n} e^{- i s \vep^\alpha D_\ka} R_\vep \partial_z \mu(s) \rmd s.   
		%		\notag \\
		%		&= e^{i z_n \vep^\alpha \partial_{xx}} \sum_{k=0}^{n-1} e^{- i z_k \vep^\alpha \partial_{xx}} \int_0^\tau e^{- i s \vep^\alpha \partial_{xx}} R_\vep \int_0^s \partial_z \mu(z_k + s_1) \rmd s_1 \rmd s. 
	\end{align}
	Here, we use $\mu(z_k + s) - \mu(z_k) \approx \tau \partial_z \mu(z_k)$ and $e^{-is\vep^\alpha D_\ka} \approx 1$ for $0 \leq s \leq \tau$, and approximate the discrete sum in $k$ by a continuous integral. Hence, the error estimate reduces to the estimate of the (oscillatory) integral in \cref{eq:sum_local} in terms of $\vep$, i.e.
	\begin{equation}\label{eq:int_intro}
		\int_0^{z_n} e^{- i s \vep^\alpha D_\ka} R_\vep \partial_z \mu(s) \rmd s =: I_\text{osc}. 
	\end{equation}
	According to the sharp regularity estimates in \cref{thm:exact}, $\| \partial_z \mu \|$ is of order $O(1)$ and thus a direct estimate can yield an $O(1)$ bound of $I_\text{osc}$ without any $\vep$ dependence. Even though a highly oscillatory phase $e^{-is\vep^\alpha D_\ka}$ is present in $I_\text{osc}$, we cannot make use of it as $\partial_z \mu$ has exhausted all the regularity properties of $\mu$ and we can no longer perform the normal form transformation (or integration by parts). The situation would not change if we used the equation to substitute $\partial_z \mu$ by $i \vep^\alpha D_\ka \mu + R_\vep \mu$, or substitute the twisted function $\mu(z) = e^{-i z \vep^\alpha D_\ka} \psi(z)$. 
    The reason is that any such techniques performed on $\mu(z)$ would ultimately rely on the regularity estimates of $\mu(z)$ or $\psi(z)$ or their derivatives at $z > 0$ and inevitably overlook phase cancellations.
	
	%To the best of our knowledge, the only way to capture all the phases in $\partial_z \mu$ in \cref{eq:int_intro} is to use an iterated Duhamel expansion in terms of the initial data $\mu_0$ as follows: 
	To capture all the phases in $\partial_z \mu$ in \cref{eq:int_intro}, we use an iterated Duhamel expansion in terms of the initial data $\mu_0$ as follows:
    \begin{equation}\label{eq:duhamel_expansion_intro}
		\mu(z) = e^{i z \vep^\alpha D_\ka} \mu_0 + \int_0^z e^{i(z - s)\vep^\alpha D_\ka} R_\vep e^{is \vep^\alpha D_\ka} \mu_0 \rmd s + \cdots, 
	\end{equation}
	where we only present the expansion to the second term for simplicity of the presentation. From \cref{eq:duhamel_expansion_intro}, we define a function $\mathfrak{u}(z, \cdot) := e^{iz\vep^\alpha D_\ka} \mu_0$ and an integral operator $\mathcal{I}:L^\infty([0, T]; X) \rightarrow L^\infty([0, T]; X)$
	\begin{equation}\label{eq:I_def}
		(\mathcal{I}\phi)(z) := e^{iz\vep^\alpha D_\kappa} \int_0^z e^{-is\vep^\alpha D_\kappa} R_\vep \phi(s) \rmd s, \quad z \geq 0, \quad \phi \in L^\infty([0, T]; X),  
	\end{equation} 
	such that the second term in the Duhamel expansion \cref{eq:duhamel_expansion_intro} simply reads $(\mathcal{I} \mathfrak{u})(z)$. Taking the $z$ derivative in \cref{eq:duhamel_expansion_intro}, we obtain
	\begin{equation}\label{eq:dz_mu_intro}
		\partial_z \mu(z) = i \vep^\alpha e^{i z \vep^\alpha D_\ka} D_\ka \mu_0 +  e^{i z \vep^\alpha D_\kappa} R_\vep \mu_0 + i \vep^\alpha e^{iz\vep^\alpha D_\kappa} \int_0^z e^{-is\vep^\alpha D_\kappa} R_\vep D_\ka \mu_0 \rmd s + \cdots,
	\end{equation}
	where we use the following crucial identity (see \cref{eq:chain} below): 
	\begin{equation}\label{eq:chain_intro}
		\partial_z (\mathcal{I\mathfrak{u}})(z) = e^{i z \vep^\alpha D_\kappa} R_\vep \mu_0 + i \vep^\alpha e^{iz\vep^\alpha D_\kappa} \int_0^z e^{-is\vep^\alpha D_\kappa} R_\vep D_\ka \mu_0 \rmd s. 
	\end{equation}
	Plugging the first two terms in the RHS of \cref{eq:dz_mu_intro} into \cref{eq:int_intro}, we obtain two typical phase interaction structures: 
	\begin{equation}\label{eq:int2_intro}
		I_{\text{osc}, 1} := i \vep^\alpha \int_0^z e^{-is\vep^\alpha D_\ka} R_\vep e^{i s \vep^\alpha D_\ka} D_\ka \mu_0 \rmd s, \qquad I_{\text{osc}, 2}: = \int_0^z e^{-is\vep^\alpha D_\ka} R_\vep e^{i s \vep^\alpha D_\ka} R_\vep  \mu_0 \rmd s. 
	\end{equation}
	For $I_{\text{osc}, 1}$, we first note that $\| D_\ka \mu_0 \| = O(1)$ compared to $\| D_\ka \mu(z) \| = O(\vep^{-\alpha})$. Moreover, $I_{\text{osc}, 1}$ improves the first-order phase interaction in \cref{eq:psi_int} by $\vep^\alpha$ due to the approximations made in the exponential integrator:
	\begin{equation}
		\| I_{\text{osc}, 1} \| = O(\vep^{\alpha} \vep^{1 - \frac{\alpha}{\ka}}) = O(\vep^{1+\frac{\ka-1}{\ka}\alpha}), 
	\end{equation}
	which corresponds to the first term in the error bound \cref{thm:main} and is estimated in \cref{prop:est1_continuous}. Perhaps more surprisingly, the numerical approximations by the exponential integrator \cref{eq:EWI_mu} also introduces a higher-order phase interaction $I_{\text{osc}, 2}$ which yields a stronger cancellation compared to that of the first-order phase interaction in $I_{\text{osc}, 1}$ and \cref{eq:psi_int}:
	\begin{equation}
		\| I_{\text{osc}, 2} \| = O(\vep^{2 - \frac{2\alpha}{\ka}}), 
	\end{equation}
	which contributes to the second term in the error bound \cref{thm:main}, and is estimated in \cref{prop:est2_continuous}. Finally, the higher-order terms in Duhamel's expansion in \cref{eq:dz_mu_intro} can be systematically reduced to either $I_{\text{osc}, 1}$ or $I_{\text{osc}, 2}$, which is done in \cref{prop:error_equation_2}. 
    This shows that while using a full Duhamel expansion in the proof is convenient, it would be possible to stop the expansion at a sufficiently high order and appeal to (standard) regularity estimates to control all higher-order terms.
    %\cw{We remark that instead of using the full Duhamel expansion as we do in the proof, it could be possible to stop the expansion at sufficiently high-order terms and appealing to the regularity estimate for the rest.}

    Besides the central use of iterated Duhamel expansions, we obtain optimal rates of convergence with no constraints on the relationship between $\tau$ and $\vep$.
    %
	%In addition to the use of the iterated Duhamel expansion, another advance of our error estimate, as already mentioned before, lies in getting rid of any $\vep$ dependent step size restrictions. 
    Although we can formally approximate the discrete summations by continuous integrals in \cref{eq:sum_local}, such an approximation is valid only when the phase in $e^{- i z_k \vep^\alpha D_\ka}$ (or, equivalently, $e^{i z_k \Phi(\xi_1, \xi_2)}$ recalling \cref{eq:phase}) is not far away from $0$, which in turn requires $\tau \lesssim \vep^{\ka - \alpha}$. This restriction arises in the partial summation of the oscillatory phase used to exploit phase cancellation: 
    \begin{equation}\label{eq:summationbyparts}
        \tau\sum_{k=0}^{n-1} e^{i z_k \Phi} = \frac{e^{iz_n\Phi} - 1}{(e^{i \tau \Phi} - 1)/\tau}, 
    \end{equation}
    which resembles the continuous (partial) integration if $(e^{i \tau \Phi} - 1)/\tau \approx i\Phi$ or, equivalently, $|\tau \Phi|$ is small. 
    To avoid such step size restrictions not observed in the numerical results, we need to deal with large phases (or large $\tau \Phi$) in the summation of discrete linear groups, which have no continuous counterparts. The main difficulty thus comes from a much larger resonant set in \cref{eq:summationbyparts}: the resonance appearing only when $\Phi = 0$ in the continuous integral now appears when $e^{i \tau \Phi} - 1 =0$ or $\Phi = 2 \pi \mathbb{Z}/\tau$ in the discrete summation. In \cref{prop:est1_discrete,prop:est2_discrete}, we carefully analyze the numerical errors when the phase $\tau \Phi$ is away from $0$ using a co-area formula. A crucial observation is that, when $|\tau \Phi|$ is large, in addition to the (global) phase cancellation \cref{eq:summationbyparts} across steps, there exists a local phase cancellation within each step $z_k$ (given by \cref{eq:inner_int_est_1,eq:inner_int_est_2}). 
    %together with the essential phase structure \cref{lem:phase}
    Notably, our analysis indicates that the numerical errors in such ``large phase" regime are in fact comparable to those in the ``small phase" regime, and thus do require careful analysis. \cref{prop:est1_continuous,prop:est1_discrete} below present the required estimates of continuous and discrete versions of $I_{\text{osc}, 1}$, while \cref{prop:est2_continuous,prop:est2_discrete} present the corresponding estimates of $I_{\text{osc}, 2}$.

	\section{Regularity estimates for the exact solution}\label{sec:2}
	In this section, we establish several regularity estimates for the solution $\mu$ and its derivatives. 

	\subsection{Some auxiliary results}
	To prove \cref{thm:exact}, we need several auxiliary results also used in the next section. 	

    For the phase $\Phi(\xi_1, \xi_2)$ in \cref{eq:phase}, in addition to viewing it as a polynomial in $\xi_1$ and $\xi_2$, we can also regard it as a polynomial in $\xi_1$ and $\xi_1 + 2\vep \xi_2$, as shown in \cref{lem:phase}. Moreover, we have the following estimate.  Let $\varsigma(\ka) = 1$ when $\ka$ is even and $\varsigma(\ka) = 0$ when $\ka$ is odd. 
    \begin{lemma}\label{cor:phase}
		If $|\xi_1| \geq C_0 \vep$ or $|\xi_1+2\vep\xi_2| \geq C_0 \vep$ with $C_0$ sufficiently large depending on $\{d_{\ka-2j}\}$ and $\ka$,
		\begin{equation*}
			|\vep^{\ka - \alpha} \Phi(\xi_1, \xi_2)| \gtrsim |\xi_1(\xi_1 + 2 \vep \xi_2)^{\varsigma(\ka)}|(\xi_1^{\ka - 1 - \varsigma(\ka)} + (\xi_1 + 2 \vep \xi_2)^{\ka - 1 - \varsigma(\ka)}). 
		\end{equation*}
	\end{lemma}
    The proof of \cref{cor:phase} is put in Appendix \ref{appen:phase}. We deduce from \cref{cor:phase} the following estimates on the phase $\Phi$:
    \begin{enumerate}[(i)]
        \item When $|\xi_1| \geq C_0 \vep$ or $|\xi_1+2\vep\xi_2| \geq C_0 \vep$, we have
        \begin{equation}\label{eq:Phi_est}
			\frac{1}{|\Phi(\xi_1, \xi_2)|} \lesssim \vep^{\kappa - \alpha} \left(\frac{1}{|\xi_1|^\kappa} + \frac{1}{|\xi_1 + 2 \vep \xi_2|^\kappa} \right). 
		\end{equation}

        \item When $|\xi_1| \geq C_0 \vep$, we have
		\begin{equation}\label{eq:phase_est_xi1}
			\frac{1}{|\Phi(\xi_1, \xi_2)|} \lesssim \frac{\vep^{\ka - \alpha}}{|\xi_1|^\ka}.  
		\end{equation}

        \item When $|\xi_1 + 2 \vep \xi_2| < \delta$ for some $\delta>0$, we have (see Appendix \ref{append:phase_more} for the proof)
        \begin{equation}\label{phase_diff_j_est}
			|\xi^j - \xi_2^j| \lesssim \frac{\delta^j}{\vep^j} + |\xi_2|^j, \qquad 1 \leq j \leq \ka-1. 
		\end{equation}

        \item When $|\xi_1 + 2 \vep \xi_2| \geq C_0 \vep$, we have (see Appendix \ref{append:phase_more} for the proof)
        \begin{equation}\label{phase_diff_j}
            \left|\frac{\xi^j - \xi_2^j}{\Phi(\xi_1, \xi_2)}\right| \lesssim \frac{\vep^{\kappa - j - \alpha}}{|\xi_1+2\vep\xi_2|^{\kappa - j}}, \qquad 0 \leq j \leq \ka. 
        \end{equation}
    \end{enumerate}
    When considering higher-order phase cancellations, $\Phi(\xi_1, \xi_2)$ is replaced with $\Phi(\xi_1, \frac{\xi_2}{\vep} + \xi_3)$, and \cref{cor:phase} then implies: when $|\xi_1| \geq C_0 \vep$ or $|\xi_1+2\xi_2+2\vep\xi_3| \geq C_0 \vep$, we have
    \begin{equation}\label{eq:phase_est_3_1}
		|\vep^{\ka - \alpha}\Phi(\xi_1, \frac{\xi_2}{\vep}+\xi_3)| \gtrsim |\xi_1(\xi_1 + 2 \xi_2 + 2\vep \xi_3)^{\varsigma(\ka)}|(\xi_1^{\ka - 1 - \varsigma(\ka)} + (\xi_1 + 2 \xi_2 + 2 \vep \xi_3)^{\ka - 1 - \varsigma(\ka)}). 
	\end{equation}
    % Moreover, we define another polynomial $f = f(\eta_1, \eta_2)$ such that 
    % \begin{equation}
    %     f(\xi_1, \xi_1 + 2\vep\xi_2) = \vep^{\ka - \alpha}\Phi(\xi_1, \xi_2) \Longleftrightarrow f(\xi_1, \xi_1 + 2 \xi_2 + 2 \vep \xi_3) = \vep^{\ka - \alpha}\Phi(\xi_1, \frac{\xi_2}{\vep} + \xi_3).
    % \end{equation}
    % We have the following estimates for $f$:
    % \begin{enumerate}[(i)]
    %     \item 
    % \end{enumerate}
	
	In addition, we have the following identity by integration by parts in $z$. 
	\begin{lemma}\label{lem:int_by_part}
		For any $\phi \in C([0, \infty); X^\ka) \cap C^1([0, \infty); X)$, we have
		\begin{equation*}
			i \vep^\alpha D_\ka \int_0^z e^{- i s \vep^\alpha D_\ka} R_\vep \phi(s) \rmd s =  - e^{- i z \vep^\alpha D_\ka} R_\vep \phi(z) + R_\vep \phi(0) + \int_0^z e^{- i s \vep^\alpha D_\ka} R_\vep \partial_s \phi(s) \rmd s. 
		\end{equation*}
	\end{lemma}
	
	\begin{proof}
		Denote by $I = I(z, x)$ the term on the LHS of the equation and define $w(z, x):=R_\vep(x) \phi(z, x)$. Taking the Fourier transform, we obtain by integrations by parts, 
		\begin{align}
			\widehat{I}(z, \xi) 
			&= - i \vep^\alpha P(\xi) \int_0^z e^{i s \vep^\alpha P(\xi)} \widehat{w}(s, \xi) \rmd s = - \int_0^z \partial_s (e^{i s \vep^\alpha P(\xi)}) \widehat{w}(s, \xi) \rmd s \notag \\
			&= - e^{i z \vep^\alpha P(\xi)} \widehat{w}(z, \xi) + \widehat{w}(0, \xi) + \int_0^z  e^{i s \vep^\alpha P(\xi)} \partial_s \widehat{w}(s, \xi) \rmd s, 
		\end{align}
		which yields the desired result by taking the inverse Fourier transform in $\xi$ and noting that $\partial_s w = R_\vep \partial_s \phi(s)$. 
	\end{proof}

	\subsection{Proof of \cref{thm:exact}}
	%In this subsection, we shall prove \cref{thm:exact}.
    First, we have the following estimate by using the oscillatory phase. 
	\begin{lemma}\label{prop:est1_continuous}
		For any $\phi \in X$, we have
		\begin{equation*}
			\| \int_0^z e^{-i s\vep^\alpha D_\kappa} R_\vep e^{i s \vep^\alpha D_\kappa} \phi \rmd s \| \lesssim \vep^{1-\frac{\alpha}{\kappa}} \| \phi \|, \qquad 0 \leq z \leq T. 
		\end{equation*}
	\end{lemma}
	
	\begin{proof}
		We first consider the case of even $\kappa$. Denote by $I = I(z, x)$ the integral to be estimated. Taking the Fourier transform and recalling \cref{eq:psi_hat_int,eq:phase}, we have
		\begin{equation}
			\widehat{I}(z, \xi) = \int_{\frac{\xi_1}{\vep} + \xi_2 = \xi} \int_0^z e^{i s \Phi(\xi_1, \xi_2)} \rmd s \widehat{R}(\xi_1) \widehat{\phi}(\xi_2) \rmd \xi_1. 
		\end{equation}
		We consider the resonant and non-resonant decomposition as: let $0<\delta \ll 1$ be a parameter to be determined later and define
		\begin{equation}\label{eq:resonant_1_continuous}
			\Lambda_1^\delta := \{(\xi_1, \xi_2) \in \R^2: |\xi_1| < \delta \text{ or } |\xi_1 + 2 \vep \xi_2| < \delta \}, \quad \Lambda_2^\delta := \R^2 \setminus \Lambda_1^\delta. 
		\end{equation}
		We shall use the following convention throughout the paper: for any $\Lambda \subset \R^2$, we define $I_\Lambda$ as
		\begin{equation}\label{eq:bilinear_multiplier}
			\widehat{I_{\Lambda}}(\xi) := \int_{\frac{\xi_1}{\vep} + \xi_2 = \xi} 1_{\Lambda} \int_0^z e^{i s \Phi(\xi_1, \xi_2)} \widehat{R}(\xi_1) \widehat{\phi}(\xi_2) \rmd s \rmd \xi_1, 
		\end{equation}
		where $1_\Lambda = 1_{\Lambda}(\xi_1, \xi_2)$ is the indicator function of $\Lambda$.
        %%% \cw{the indicator function of $\Lambda$, and can be understood as a (bilinear) Fourier multiplier} \gb{? Isn't it an indicatrix function on $\Lambda$?}. %%% It's the kernel of a linear operator but is not linear in the \xi variables. How about the above to avoid any confusion?
        For $I_{\Lambda^\delta_1}$, we have
		\begin{align}\label{eq:est1_con_1}
			\| \widehat{I_{\Lambda^\delta_1}}(z, \xi) \|_{L^1_\xi} 
			&\leq z \int_{\R^2} 1_{\Lambda_1^\delta} |\widehat{R}(\xi_1)| |\widehat{\phi}(\xi_2)| \rmd \xi_1 \rmd \xi_2 \notag \\
			&\leq z \int_{\R^2} (1_{|\xi_1| < \delta} + 1_{|\xi_1 + 2 \vep \xi_2| < \delta}) |\widehat{R}(\xi_1)| |\widehat{\phi}(\xi_2)| \rmd \xi_1 \rmd \xi_2 %\notag 
            %\\
		%	&\leq z \int_{\R} 1_{|\xi_1| < \delta} |\widehat{R}(\xi_1)| \rmd \xi_1 \int_{\R} |\widehat{\phi}(\xi_2)| \rmd \xi_2 + z \int_{\R} |\widehat{\phi}(\xi_2)| \int_{\R} 1_{|\xi_1 + 2 \vep \xi_2| < \delta} |\widehat{R}(\xi_1)| \rmd \xi_1 \rmd \xi_2 \notag \\ 
			\ \lesssim \ z \,\delta\, \| \widehat{R} \|_{L^\infty} \| \widehat{\phi} \|_{L^1}. 
		\end{align}
		
        When $(\xi_1, \xi_2) \in \Lambda^\delta_2$, according to the choice of $\delta$ in \cref{eq:delta} later, we have $|\xi_1|, |\xi_1 + 2 \vep \xi_2| \geq \delta \geq C_0\vep$, and thus \cref{eq:Phi_est} holds and $\Phi(\xi_1, \xi_2) \neq 0$. 
        % Hence, by \cref{cor:phase}, we have $\Phi(\xi_1, \xi_2) \neq 0$ and moreover
		% \begin{equation}
		% 	\vep^{\kappa - \alpha}|\Phi(\xi_1, \xi_2)| \gtrsim \left\{
		% 	\begin{aligned}
		% 		&|\xi_1 + 2 \vep \xi_2|^\kappa, && |\xi_1| \geq |\xi_1 + 2 \vep \xi_2|, \\
		% 		&|\xi_1|^\kappa, && |\xi_1| < |\xi_1 + 2 \vep \xi_2|, 
		% 	\end{aligned}
		% 	\right. 
		% \end{equation}
  %       \gb{I believe all properties of $\Phi$ should be collected at one place. I would combine Lemma 2.4 and the following corollary as well as any estimate we need on $\Phi$.}
		% which implies
		%		It follows that
		%		\begin{equation}
			%			\int_0^z e^{i s \Phi(\xi_1, \xi_2)} \rmd s = \frac{e^{i z \Phi} - 1}{i \Phi}. 
			%		\end{equation}
		Then for $I_{\Lambda^\delta_2}$, we have
		\begin{align}
			\widehat{I_{\Lambda^\delta_2}}(z, \xi) 
			&= \int_{\frac{\xi_1}{\vep} + \xi_2 = \xi} 1_{\Lambda^\delta_2} \frac{1}{i\Phi(\xi_1,\xi_2)} \int_0^z \partial_s e^{i s \Phi(\xi_1, \xi_2)} \rmd s \widehat{R}(\xi_1) \widehat{\phi}(\xi_2) \rmd \xi_1 \notag \\
			&= \int_{\frac{\xi_1}{\vep} + \xi_2 = \xi} 1_{\Lambda^\delta_2} \frac{e^{i z \Phi(\xi_1, \xi_2)} - 1}{i\Phi(\xi_1,\xi_2)} \widehat{R}(\xi_1) \widehat{\phi}(\xi_2) \rmd \xi_1. 
		\end{align}
		By \cref{eq:Phi_est}, it follows that
		\begin{align}\label{eq:est1_con_2}
			\| \widehat{I_{\Lambda^\delta_2}}(z, \xi) \|_{L^1_\xi} 
			&\lesssim \vep^{\kappa-\alpha} \int_{\R^2} \frac{1_{|\xi_1|>\delta}}{|\xi_1|^\kappa} |\widehat{R}(\xi_1)| |\widehat{\phi}(\xi_2)| \rmd \xi_1 \rmd \xi_2 + \vep^{\kappa-\alpha} \int_{\R^2} \frac{1_{|\xi_1 + 2 \vep \xi_2|>\delta}}{|\xi_1 + 2 \vep \xi_2|^\kappa} |\widehat{R}(\xi_1)| |\widehat{\phi}(\xi_2)| \rmd \xi_1 \rmd \xi_2 \notag \\
			&\lesssim \vep^{\kappa-\alpha} \int_{\R} \frac{1_{|\xi_1|>\delta}}{|\xi_1|^\kappa} |\widehat{R}(\xi_1)| \rmd \xi_1 \int_{\R}|\widehat{\phi}(\xi_2)| \rmd \xi_2 + \vep^{\kappa-\alpha} \int_{\R} |\widehat{\phi}(\xi_2)| \int_{\R} \frac{1_{|\xi_1 + 2 \vep \xi_2|>\delta}}{|\xi_1 + 2 \vep \xi_2|^\kappa} |\widehat{R}(\xi_1)| \rmd \xi_1 \rmd \xi_2 \notag \\
			&\lesssim \vep^{\kappa-\alpha} (\| \widehat{R} \|_{L^1} + \frac{\| \widehat{R} \|_{L^\infty}}{\delta^{\kappa-1}}) \| \widehat{\phi} \|_{L^1}. 
		\end{align}
        %%%\gb{I believe $\delta^{\kappa-1}$ is on the other term. This requires $\kappa\geq2$, which we have.} OK
		Combining \cref{eq:est1_con_1,eq:est1_con_2}, we obtain, by choosing $\delta = \vep^{1 - \alpha/\kappa}$ when $\alpha > 0$ and $\delta = C_0\vep$ according to \cref{cor:phase} when $\alpha = 0$, 
		\begin{equation}\label{eq:delta}
			\| \widehat{I}(z, \xi) \|_{L^1_\xi} \lesssim \langle z \rangle \vep^{1 - \alpha/\kappa}, \qquad \langle z \rangle: = \sqrt{1+z^2}. 
		\end{equation}
		This completes the proof for even $\kappa$. %\gb{Above bracket notation not defined?} 
        When $\kappa$ is odd, the proof follows similarly with a simpler resonant and non-resonant decomposition: 
		\begin{equation}
			\Lambda_1^\delta := \{(\xi_1, \xi_2) \in \R^2: |\xi_1| < \delta \}, \quad \Lambda_2^\delta := \R^2 \setminus \Lambda_1^\delta, 
		\end{equation}
		and using \cref{eq:phase_est_xi1} in the non-resonant regime $(\xi_1, \xi_2) \in \Lambda^\delta_2$. 
        % \gb{I believe all properties of $\Phi$ should be collected at one place. I would combine Lemma 2.4 and the following corollary as well as any estimate we need on $\Phi$.}
		We omit the details. %Then the proof is complete. 
	\end{proof}
	
	\begin{proof}[Proof of \cref{thm:exact}]
		We start with \cref{eq:exact_1}. Taking the norm $\| \cdot \|$ on both sides of \cref{eq:mu}, we have
		\begin{equation}\label{p_z_mu_bound}
			\| \partial_z \mu(z) \| \leq \vep^\alpha \|D_\ka \mu(z)\| + \| R \| \| \mu(z) \|. 
		\end{equation}
		Therefore, it suffices to obtain estimates for $\| \mu(z) \|$ and $ \| D_\ka \mu(z) \| $, or equivalently for $\| \psi(z) \|$ and $ \| D_\ka \psi(z) \| $, respectively. From \cref{eq:psi_int}, we have
		\begin{equation}\label{psi_bound}
			\| \psi(z) \| \leq \| \psi(0) \| + \| R \| \int_0^z \| \psi(s) \| \rmd s, 
		\end{equation}
		which implies $\| \psi(z) \| \lesssim 1$ by Gronwall's inequality. Then, by \cref{eq:psi}, we have
		\begin{equation}\label{p_z_psi_bound}
			\| \partial_z \psi(z) \| = \| R \| \| \psi(z) \| \lesssim 1.
		\end{equation} 
		Applying $\vep^\alpha D_\ka$ on both sides of \cref{eq:psi_int}, by \cref{lem:int_by_part}, we have
		\begin{align}\label{Dk_psi}
			\vep^\alpha D_\ka \psi(z) 
			&=  \vep^\alpha D_\ka \psi(0) + i e^{- i z \vep^\alpha D_\ka} R_\vep e^{i z \vep^\alpha D_\ka} \psi(z) - i R_\vep \psi(0) \notag \\
			&\quad  + \int_0^z e^{- i s \vep^\alpha D_\ka} R_\vep e^{i s \vep^\alpha D_\ka} (\vep^\alpha D_\ka \psi(s)) \rmd s - i \int_0^z e^{- i s \vep^\alpha D_\ka} R_\vep e^{i s \vep^\alpha D_\ka} \partial_z \psi(s) \rmd s, 
		\end{align}
		It follows that
		\begin{equation}
			\| \vep^\alpha D_\ka \psi(z) \| \lesssim \vep^\alpha \| D_\ka \psi(0) \| + \| \psi(z) \| + \| \psi(0) \| + \int_0^z \| \vep^\alpha D_\ka \psi(s) \| \rmd s + \int_0^z \| \partial_z \psi(s) \| \rmd s,  
		\end{equation}
		which implies by Gronwall's inequality, \cref{psi_bound,p_z_psi_bound} that 
		\begin{equation}
			\vep^\alpha\| D_\ka \psi(z) \| \lesssim 1.  
		\end{equation}
		This gives the bound of $\partial_z \mu$ by \cref{p_z_mu_bound}. The bound for $\partial_x^\ka \mu$ follows from
		\begin{equation}
			\| \partial_x^\ka \mu \| = \| \partial_x^\ka \psi \| \lesssim \| \psi \| + \| D_\ka \psi \| \lesssim \vep^{-\alpha}, 
		\end{equation}
		which completes the proof of \cref{eq:exact_1}. 
		
		Then we prove \cref{eq:exact_2}. Denote by $w(z): = \psi(z) - \psi(0)$ for $z \geq 0$. From \cref{eq:psi_int}, we have
		\begin{equation}\label{eq:w}
			w(z) = \int_0^{z} e^{-i s\vep^\alpha D_\kappa} R_\vep e^{i s \vep^\alpha D_\kappa} (\psi(0) + w(s)) \rmd s. 
		\end{equation}
		By \cref{prop:est1_continuous}, we have
		\begin{equation}
			\| w(z) \| \leq \| \int_0^{z} e^{-i s\vep^\alpha D_\kappa} R_\vep e^{i s \vep^\alpha D_\kappa} \psi(0) \rmd s \| + \int_0^z \| w(s) \| \rmd s \lesssim \vep^{1-\alpha/\kappa} + \int_0^z \| w(s) \| \rmd s. 
		\end{equation}
		Then the result for $j=0$ follows from Gronwall's inequality immediately. 
		
		In the following, we consider $1 \leq j \leq \kappa-1$. From \cref{eq:w}, taking $\partial_x^j$ on both sides, we have
		\begin{align}\label{eq:A}
			\partial_x^j w(z) 
			&= \int_0^{z} e^{- i s \vep^\alpha D_\kappa} R_\vep e^{i s \vep^\alpha D_\kappa}  (\partial_x^j \psi(0) + \partial_x^j w(s)) \rmd s + \int_0^{z} e^{-is\vep^\alpha D_\kappa} [\partial_x^j, R_\vep] e^{i s \vep^\alpha D_\kappa} (\psi(0) + w(s)) \rmd s, \notag \\
			&=: \sum_{l=1}^4 A_l, 
		\end{align}
		where $[\partial_x^j, R_\vep] = \partial_x^j R_\vep - R_\vep \partial_x^j$ is the standard commutator. 
		For $A_1$, by \cref{prop:est1_continuous} and the assumption \cref{eq:assumption}, we have 
		\begin{equation}\label{A1_est}
			\| A_1 \| \lesssim \langle z \rangle \vep^{1-\alpha/\ka}. 
		\end{equation}
		For $A_2$, we have 
		\begin{equation}\label{A2_est}
			\| A_2 \| \leq \| R \| \int_0^z \| \partial_x^j w(s) \| \rmd s. 
		\end{equation}
		We consider $A_3$ next. After Fourier transform, we have
		\begin{equation}
			\widehat{A_3}(\xi) = \int_{\frac{\xi_1}{\vep} + \xi_2 = \xi} \int_0^z e^{- i s \Phi(\xi_1, \xi_2)} (\xi^j - \xi_2^j) \widehat{R}(\xi_1) \widehat{\psi}(0, \xi_2) \rmd s \rmd \xi_1. 
		\end{equation}
        We consider the resonant and non-resonant decomposition as 
		\begin{equation}\label{eq:resonant_exact}
			\Lambda_1^\delta := \{(\xi_1, \xi_2) \in \R^2: |\xi_1 + 2 \vep \xi_2| < \delta\}, \quad \Lambda_2^\delta := \R^2 \setminus \Lambda_1^\delta. 
		\end{equation}

		For $A_{3, \Lambda^\delta_1}$, using \cref{phase_diff_j_est}, we have
		\begin{align}\label{A3_1_est}
			\| \widehat{A_{3, \Lambda^\delta_1}}(\xi) \|_{L^1_\xi} 
			&\leq z \int_{\R^2} 1_{|\xi_1 + 2 \vep \xi_2|<\delta} |\xi^j - \xi_2^j| |\widehat{R}(\xi_1)| |\widehat{\psi}(0, \xi_2)| \rmd \xi_1 \rmd \xi_2 \notag \\
			&\lesssim z\int_{\R^2} 1_{|\xi_1 + 2 \vep \xi_2|<\delta} \frac{\delta^j}{\vep^j} |\widehat{R}(\xi_1)| |\widehat{\psi}(0, \xi_2)| \rmd \xi_1 \rmd \xi_2 + z\int_{\R^2} 1_{|\xi_1 + 2 \vep \xi_2|<\delta} |\widehat{R}(\xi_1)| |\widehat{(\partial_x^j \psi)}(0, \xi_2)| \rmd \xi_1 \rmd \xi_2 \notag \\
			&=z \frac{\delta^j}{\vep^j} \int_{\R} |\widehat{\psi}(0, \xi_2)| \int_{\R} 1_{|\xi_1 + 2 \vep \xi_2|<\delta} |\widehat{R}(\xi_1)| \rmd \xi_1 \rmd \xi_2 + z\int_{\R} |\widehat{(\partial_x^j \psi)}(0, \xi_2)| \int_{\R} 1_{|\xi_1 + 2 \vep \xi_2|<\delta} |\widehat{R}(\xi_1)| \rmd \xi_1 \rmd \xi_2 \notag \\
			&\lesssim z\frac{\delta^{j+1}}{\vep^j} \| \widehat{R} \|_{L^\infty} \| \widehat{\psi}(0) \|_{L^1} + z\delta \| \widehat{R} \|_{L^\infty} \| \partial_x^j \psi(0) \|.  
		\end{align}
		For $A_{3, \Lambda^\delta_2}$, we have 
		\begin{align}\label{A3}
			\widehat{A_{3, \Lambda^\delta_2}}(\xi) 
			&= \int_{\frac{\xi_1}{\vep} + \xi_2 = \xi} 1_{|\xi_1 + 2 \vep \xi_2| \geq \delta} \frac{(\xi^j - \xi_2^j)}{i \Phi(\xi_1, \xi_2)} \int_0^z \partial_s e^{i s \Phi(\xi_1, \xi_2)} \rmd s \widehat{R}(\xi_1) \widehat{\psi}(0, \xi_2) \rmd \xi_1 \notag \\
			&= \int_{\frac{\xi_1}{\vep} + \xi_2 = \xi} 1_{|\xi_1 + 2 \vep \xi_2| \geq \delta} \frac{(\xi^j - \xi_2^j)}{i\Phi(\xi_1, \xi_2)} (e^{i z \Phi(\xi_1, \xi_2)} - 1) \widehat{R}(\xi_1) \widehat{\psi}(0, \xi_2) \rmd \xi_1. 
		\end{align}
        We note that when $(\xi_1, \xi_2) \in \Lambda_2^\delta$, $|\xi_1 + 2\vep \xi_2| \geq \delta \geq C_0 \vep$ according to the choice of $\delta$ later in \cref{A3_est}. 
		%		Recalling \cref{eq:phi4}, we have, when $(\xi_1, \xi_2) \in \Lambda^\delta_2$, 
		%		\begin{equation}
			%			\left|\frac{\xi^j - \xi_2^j}{\Phi(\xi_1, \xi_2)}\right| 
			%			\lesssim \left\{
			%			\begin{aligned}
				%				&\vep^{3-\alpha} \frac{1}{|\xi_1 + 2\vep \xi_2|((\xi_1 + 2\vep \xi_2)^2 + \xi_1^2)} \lesssim \frac{\vep^{3-\alpha}}{|\xi_1 + 2\vep \xi_2|^3}, && k = 1, \\
				%				&\vep^{2-\alpha} \frac{1}{(\xi_1 + 2\vep \xi_2)^2 + \xi_1^2} \lesssim \frac{\vep^{2-\alpha}}{|\xi_1 + 2\vep \xi_2|^2}, && k = 2, \\
				%				&\vep^{1-\alpha} \frac{1}{|\xi_1 + 2\vep \xi_2|}, && k = 3. \\
				%			\end{aligned}
			%			\right.
			%		\end{equation}
        % \gb{I believe all properties of $\Phi$ should be collected at one place. I would combine Lemma 2.4 and the following corollary as well as any estimate we need on $\Phi$.}
		Then, from \cref{A3}, using \cref{phase_diff_j}, we obtain 
		\begin{align}\label{A3_2_est}
			\| \widehat{A_{3, \Lambda^\delta_2}}(\xi) \|_{L^1_\xi} 
			&\lesssim \int_{\R^2} 1_{|\xi_1 + 2 \vep \xi_2| \geq \delta} \frac{\vep^{\kappa - j - \alpha}}{|\xi_1 + 2\vep \xi_2|^{\kappa-j}}|\widehat{R}(\xi_1)| |\widehat{\psi}(0, \xi_2)| \rmd \xi_1 \rmd \xi_2 \notag \\
			&=\int_{\R} |\widehat{\psi}(0, \xi_2)| \int_{\R}1_{|\xi_1 + 2 \vep \xi_2| \geq \delta} \frac{\vep^{\kappa - j - \alpha}}{|\xi_1 + 2\vep \xi_2|^{\kappa-j}}|\widehat{R}(\xi_1)| \rmd \xi_1 \rmd \xi_2 \notag \\
			&\lesssim \vep^{\kappa-j-\alpha} (\frac{\| \widehat{R} \|_{L^\infty}}{\delta^{\kappa - j - 1}} + \| \widehat{R} \|_{L^1}) \| \widehat{\psi}(0) \|_{L^1} \lesssim \frac{\vep^{\kappa-j-\alpha}}{\delta^{\kappa - j - 1}}, 
		\end{align}
		where $\frac{1}{\delta^0}$ is understood as $\ln \frac{1}{\delta}$.  Combining \cref{A3_1_est,A3_2_est}, we obtain, by choosing $\delta = \vep^{1 - \alpha/\kappa}$ when $\alpha > 0$ and $\delta = C_0\vep$ according to \cref{cor:phase} when $\alpha = 0$, 
		\begin{equation}\label{A3_est}
			\| \widehat{A_{3}} \|_{L^1} \lesssim z \delta + z \frac{\delta^{j+1}}{\vep^j} + \frac{\vep^{\kappa-j-\alpha}}{\delta^{\kappa-j - 1}} \sim \langle z \rangle
			\left\{
			\begin{aligned}
				&\vep^{1 - \frac{j+1}{\kappa}\alpha}, && 1 \leq j \leq \kappa - 2, \\
				&\vep^{1-\alpha} |\ln \vep|, && j = \kappa-1. 
			\end{aligned}
			\right. 
		\end{equation}
		
		Finally, we consider $A_4$. Taking the Fourier transform, we have
		\begin{equation}
			\widehat{A_4}(\xi) = \int_{\frac{\xi_1}{\vep} + \xi_2 = \xi} \int_0^z e^{i s \Phi(\xi_1, \xi_2)} (\xi^j - \xi_2^j) \widehat{R}(\xi_1) \widehat{w}(s, \xi_2) \rmd s \rmd \xi_1. 
		\end{equation}
		Consider again the resonant and non-resonant decomposition \cref{eq:resonant_exact}. For $A_{4, \Lambda^\delta_1}$, similar to \cref{A3_1_est}, we have
        \begin{align*}%\label{A4_est_1}
            & \qquad \| \widehat{A_{4, \Lambda^\delta_1}(\xi)} \|_{L^1_\xi} 
			\leq \int_{\R^2} 1_{|\xi_1 + 2 \vep \xi_2|<\delta} |\xi^j - \xi_2^j| |\widehat{R}(\xi_1)| \int_0^z |\widehat{w}(s, \xi_2)| \rmd s \rmd \xi_1 \rmd \xi_2  \\
            \lesssim \ &  \frac{\delta^j}{\vep^j} \int_{\R^2} 1_{|\xi_1 + 2 \vep \xi_2|<\delta} |\widehat{R}(\xi_1)| \int_0^z |\widehat{w}(s, \xi_2)| \rmd s \rmd \xi_1 \rmd \xi_2
			 + \int_{\R^2} 1_{|\xi_1 + 2 \vep \xi_2|<\delta} |\widehat{R}(\xi_1)| \int_0^z |\widehat{(\partial_x^j w)}(s, \xi_2)| \rmd s \rmd \xi_1 \rmd \xi_2 \notag \\
             =\ & \frac{\delta^j}{\vep^j} \int_{\R} \int_0^z |\widehat{w}(s, \xi_2)| \rmd s \int_{\R} 1_{|\xi_1 + 2 \vep \xi_2|<\delta} |\widehat{R}(\xi_1)| \rmd \xi_1 \rmd \xi_2 \notag  + \int_{\R} \int_0^z |\widehat{(\partial_x^k w)}(s, \xi_2)| \rmd s \int_{\R} 1_{|\xi_1 + 2 \vep \xi_2|<\delta} |\widehat{R}(\xi_1)| \rmd \xi_1 \rmd \xi_2 \notag \\
            \lesssim\ &  \frac{\delta^{j+1}}{\vep^j} \| \widehat{R} \|_{L^\infty} \int_0^z \| w(s) \| \rmd s + \delta \| \widehat{R} \|_{L^\infty} \int_0^z \| {\partial_x^j w}(s) \| \rmd s 
			\ \lesssim \ \frac{\delta^{j+1}}{\vep^j} \vep^{1-\alpha/\ka} + \delta \int_0^z \| {\partial_x^j w}(s) \| \rmd s. 
            \notag
        \end{align*}
		% \begin{align}\label{A4_est_1}
		% 	\| \widehat{A_{4, \Lambda^\delta_1}(\xi)} \|_{L^1_\xi} 
		% 	&\leq \int_{\R^2} 1_{|\xi_1 + 2 \vep \xi_2|<\delta} |\xi^j - \xi_2^j| |\widehat{R}(\xi_1)| \int_0^z |\widehat{w}(s, \xi_2)| \rmd s \rmd \xi_1 \rmd \xi_2 \notag \\
		% 	&\lesssim \frac{\delta^j}{\vep^j} \int_{\R^2} 1_{|\xi_1 + 2 \vep \xi_2|<\delta} |\widehat{R}(\xi_1)| \int_0^z |\widehat{w}(s, \xi_2)| \rmd s \rmd \xi_1 \rmd \xi_2 \notag \\
		% 	&\quad + \int_{\R^2} 1_{|\xi_1 + 2 \vep \xi_2|<\delta} |\widehat{R}(\xi_1)| \int_0^z |\widehat{(\partial_x^j w)}(s, \xi_2)| \rmd s \rmd \xi_1 \rmd \xi_2 \notag \\
		% 	&=\frac{\delta^j}{\vep^j} \int_{\R} \int_0^z |\widehat{w}(s, \xi_2)| \rmd s \int_{\R} 1_{|\xi_1 + 2 \vep \xi_2|<\delta} |\widehat{R}(\xi_1)| \rmd \xi_1 \rmd \xi_2 \notag \\
		% 	&\quad + \int_{\R} \int_0^z |\widehat{(\partial_x^k w)}(s, \xi_2)| \rmd s \int_{\R} 1_{|\xi_1 + 2 \vep \xi_2|<\delta} |\widehat{R}(\xi_1)| \rmd \xi_1 \rmd \xi_2 \notag \\
		% 	&\lesssim \frac{\delta^{j+1}}{\vep^j} \| \widehat{R} \|_{L^\infty} \int_0^z \| w(s) \| \rmd s + \delta \| \widehat{R} \|_{L^\infty} \int_0^z \| {\partial_x^j w}(s) \| \rmd s \notag \\
		% 	&\lesssim \frac{\delta^{j+1}}{\vep^j} \vep^{1-\alpha/\ka} + \delta \int_0^z \| {\partial_x^j w}(s) \| \rmd s. 
		% \end{align}
		For $A_{4, \Lambda^\delta_2}$, we have, by integration by parts, noting that $w(0) = 0$ and $\partial_s w(s) = \partial_s \psi(s)$, 
		\begin{align*}
			\widehat{A_{4, \Lambda^\delta_2}}(\xi) 
			&= \int_{\frac{\xi_1}{\vep} + \xi_2 = \xi} 1_{|\xi_1 + 2 \vep \xi_2| \geq \delta} \frac{(\xi^j - \xi_2^j)}{i \Phi(\xi_1, \xi_2)} \widehat{R}(\xi_1) \int_0^z \partial_s(e^{i s \Phi(\xi_1, \xi_2)}) \widehat{w}(s, \xi_2) \rmd s \rmd \xi_1 \notag \\
			&= \int_{\frac{\xi_1}{\vep} + \xi_2 = \xi} 1_{|\xi_1 + 2 \vep \xi_2| \geq \delta} \frac{(\xi^j - \xi_2^j)}{\Phi(\xi_1, \xi_2)} \widehat{R}(\xi_1) (e^{i z \Phi} \widehat{w}(z, \xi_2) - \int_0^z e^{i s \Phi} \partial_s \widehat{\psi}(s, \xi_2) \rmd s ) \rmd \xi_1. 
		\end{align*}
		Then we have, using \cref{phase_diff_j}, 
		\begin{align*}%\label{A4_est_2}
			\| \widehat{A_{4, \Lambda_2^\delta}}(\xi) \|_{L^1_\xi} 
			&\lesssim \int_{\R^2} 1_{|\xi_1 + 2 \vep \xi_2| \geq \delta} \frac{\vep^{\kappa - j - \alpha}}{|\xi_1+2\vep\xi_2|^{\kappa - j}} \widehat{R}(\xi_1)| |\widehat{w}(z, \xi_2)| \rmd \xi_1 \rmd \xi_2 \notag \\
			&\quad + \int_{\R^2} 1_{|\xi_1 + 2 \vep \xi_2| \geq \delta} \frac{\vep^{\kappa - j - \alpha}}{|\xi_1+2\vep\xi_2|^{\kappa - j}} |\widehat{R}(\xi_1)| \int_0^z  |\partial_s \widehat{\psi}(s, \xi_2)| \rmd s \rmd \xi_1 \rmd \xi_2 \notag \\
			&=\int_{\R} |\widehat{w}(z, \xi_2)| \int_{\R} 1_{|\xi_1 + 2 \vep \xi_2| \geq \delta} \frac{\vep^{\kappa - j - \alpha}}{|\xi_1+2\vep\xi_2|^{\kappa - j}} |\widehat{R}(\xi_1)| \rmd \xi_1 \rmd \xi_2 \notag \\
			&\quad + \int_{\R} \int_0^z  |\partial_s \widehat{\psi}(s, \xi_2)| \rmd s \int_{\R} 1_{|\xi_1 + 2 \vep \xi_2| \geq \delta} \frac{\vep^{\kappa - j - \alpha}}{|\xi_1+2\vep\xi_2|^{\kappa - j}} |\widehat{R}(\xi_1)| \rmd \xi_1 \rmd \xi_2 \notag \\
			&\lesssim \vep^{\kappa-j-\alpha} (\frac{\| \widehat{R} \|_{L^\infty}}{\delta^{\kappa-j - 1}} + \| \widehat{R} \|_{L^1}) (\| w(z) \| + \int_0^z \| \partial_s \psi(s) \| \rmd s ) \notag \\
			&\lesssim \frac{\vep^{\kappa-j-\alpha}}{\delta^{\kappa - j - 1}}(\vep^{1-\alpha/\kappa} + 1) \lesssim \frac{\vep^{\kappa-j-\alpha}}{\delta^{\kappa - j - 1}}. 
		\end{align*}
		Combining the above estimates, we have, by the same choice of $\delta$ as in \cref{A3_est},  
		\begin{equation}\label{A4_est}
			\| \widehat{A_{4}} \|_{L^1} \lesssim \frac{\delta^{1+j}}{\vep^j} \vep^{1-\frac{\alpha}{\ka}} +  \delta \int_0^z \| \partial_x^j w(s) \| \rmd s + \frac{\vep^{\kappa-j-\alpha}}{\delta^{\kappa-j-1}} \lesssim \int_0^z \| \partial_x^j w(s) \| \rmd s +  
			\left\{
			\begin{aligned}
				&\vep^{1 - \frac{j+1}{\ka}\alpha}, && 1 \leq j \leq \ka - 2, \\
				&\vep^{1-\alpha} |\ln \vep|, && j=\ka - 1. 
			\end{aligned}
			\right. 
		\end{equation}
		Combining \cref{A1_est,A2_est,A3_est,A4_est}, the proof is completed using Gronwall's inequality. 
	\end{proof}
	
	\section{Error estimates of the exponential integrator}\label{sec:3}
	In this section, we establish the error estimate \cref{thm:main} of the first-order exponential integrator \cref{eq:EWI_mu} applied to \cref{eq:mu}. As mentioned before, we work with the twisted function $\psi$ \cref{eq:psi_def}. Hence, for the numerical solution, we define $\psi^n:=e^{-iz_n\vep^\alpha D_\kappa}\mu^n$, which plugged into \cref{eq:EWI_mu} yields 
	\begin{equation}\label{eq:EWI_psi}
		\begin{aligned}
			&\psi^{n+1} = \psi^n + \tau e^{-i z_{n+1} \vep^\alpha D_\kappa} \vphi_1(i \tau \vep^\alpha D_\kappa) (R_\vep e^{i z_{n} \vep^\alpha D_\kappa} \psi^n), \quad n \geq 0, \\
			&\psi^0 = \mu_0. 
		\end{aligned}
	\end{equation} 
	By the isometry property, we have
	\begin{equation*}
		\| \mu(z_n) - \mu^n \| = \| \psi(z_n) - \psi^n \|, 
	\end{equation*}
	and we derive the error bound for $\psi$. 

    Before presenting the proof, we first summarize some notation to be used in this section: 
    \begin{itemize}

        \item For $z \geq 0$, $\fl{z} := z_n$, when $z_n \leq z < z_{n+1}$. 
	    % Define
	    % \begin{equation}
     %        \mathfrak{u}_\text{d}(z) := e^{i \lfloor z \rfloor \vep^\alpha D_\kappa} \psi(0) = \mathfrak{u}(\fl{z}). 
	    % \end{equation}
    
        \item $\mathcal{I}$ is the integral operator defined in \cref{eq:I_def}, and we further define $\mathcal{I}^m \phi := \mathcal{I} (\mathcal{I}^{m-1}\phi)$ for $m \geq 2$ and $\mathcal{I}^0 = \mathrm{Id}$.

        \item Denote by $\mathcal{I}_\text{d}$ another integral operator defined as
	    \begin{equation}\label{eq:Id_def}
            (\mathcal{I}_\text{d} \phi)(z) := e^{i \lfloor z \rfloor \vep^\alpha D_\kappa} \int_{0}^{\lfloor z \rfloor} e^{-i s \vep^\alpha D_\kappa} R_\vep \phi(s) \rmd s = (\mathcal{I}\phi)(\fl{z}), \quad z \geq 0, \quad \phi \in L^\infty([0, T]; X),
	    \end{equation}
	    with $\mathcal{I}_\rmd^{m} \phi := \mathcal{I}_\rmd (\mathcal{I}_\rmd^{m-1} \phi)$ for $m \geq 2$ and $\mathcal{I}_\rmd^0 = \mathrm{Id}$.

        \item For $z \geq 0$, 
        \begin{equation}\label{eq:uv_def}
            \mathfrak{u}(z)=e^{iz\vep^\alpha D_\ka}\psi(0) \text{ as defined in \cref{eq:I_def}}, \quad \mathfrak{u}_\text{d}(z) := e^{i \lfloor z \rfloor \vep^\alpha D_\kappa} \psi(0), \quad v(z):=e^{iz\vep^\alpha D_\ka} R_\vep \psi(0).
        \end{equation}
    \end{itemize}
    Here, $\mathcal{I}_\rmd$ and $\mathfrak{u}_\rmd$ are the discrete counterparts of $\mathcal{I}$ and $\mathfrak{u}$.
	
	\subsection{Duhamel's expansion}\label{sec:3.2}
	In this subsection, we present the Duhamel expansion for the exact solution $\psi(z_n)$ and the numerical solution $\psi^n$, and derive an equation for the error term. 
	
	We recall that the Duhamel formula \cref{eq:psi_int} holds for the exact solution $\psi(z)$. 
    % Recalling \cref{eq:I_def}, 
 %    Define an operator $\mathcal{I}:L^\infty([0, T]; X) \rightarrow L^\infty([0, T]; X)$ as
	% \begin{equation}\label{eq:I_def}
	% 	(\mathcal{I}\phi)(z) := e^{iz\vep^\alpha D_\kappa} \int_0^z e^{-is\vep^\alpha D_\kappa} R_\vep \phi(s) \rmd s, \quad z \geq 0, \quad \phi \in L^\infty([0, T]; X).   
	% \end{equation}
 %    \gb{Wasn't this operator already used earlier? To be defined and used there?}
    Iterating \cref{eq:psi_int}, we obtain %the Duhamel's expansion of the exact solution:
	\begin{equation}\label{eq:duhamel_expansion_continuous}
		\psi(z) = \sum_{m \geq 0} \psi_m(z), \qquad z\geq0,  
	\end{equation}
	where $\psi_0(z) = \psi(0) = \mu_0$, and
	\begin{equation}\label{eq:duhamel_continuous_compact}
		\psi_m (z) = \int_0^z e^{-i s \vep^\alpha D_\ka} R_\vep (\mathcal{I}^{m-1}\mathfrak{u})(s) \rmd s, \quad m \geq 1. 
	\end{equation}
	We next consider the Duhamel expansion for the numerical solution. Recalling that 
	\begin{equation}
		\tau \vphi_1(i \tau \vep^\alpha D_\kappa) = \int_0^\tau e^{i(\tau - s) \vep^\alpha D_\kappa} \rmd s, 
	\end{equation}
	we can rewrite \cref{eq:EWI_psi} as
	\begin{align}
		\psi^{n+1} 
		&= \psi^{n} + \int_{0}^{\tau} e^{-i (z_n+s) \vep^\alpha D_\kappa} R_\vep e^{iz_n \vep^\alpha D_\kappa} \psi^n \rmd s, \notag \\
		&=\psi^{n} + \int_{z_n}^{z_{n+1}} e^{-i s \vep^\alpha D_\kappa} R_\vep e^{iz_n \vep^\alpha D_\kappa} \psi^n \rmd s, \qquad n \geq 0. 
	\end{align}
	Solving the recursion, we obtain
	\begin{equation}\label{eq:duhamel_discrete}
		\psi^{n} = \psi^{0} + \sum_{k=0}^{n-1} \int_{z_k}^{z_{k+1}} e^{-i s \vep^\alpha D_\kappa} R_\vep e^{iz_k \vep^\alpha D_\kappa} \psi^k \rmd s, \qquad n \geq 0. 
	\end{equation}
	Iterating \cref{eq:duhamel_discrete}, we obtain% the Duhamel's expansion for the numerical solution as
	\begin{equation}
		\psi^n = \sum_{m \geq 0} \psi^n_m, 
	\end{equation}
	where $\psi^n_0 = \psi(0)$, and with $k_0 := n$, for $ m\geq 1 $, 
	\begin{equation}\label{eq:duhamel_summation}
		\psi^n_m = \sum_{k_1=0}^{n-1} \int_{z_{k_1}}^{z_{k_1+1}} e^{-i s_1 \vep^\alpha D_\kappa} R_\vep e^{iz_{k_1} \vep^\alpha D_\kappa} \cdots \sum_{k_{m}=0}^{k_{m-1}-1} \int_{z_{k_m}}^{z_{k_m+1}} e^{-i s_m \vep^\alpha D_\kappa} R_\vep e^{iz_{k_m} \vep^\alpha D_\kappa}  \psi^0 \rmd s_m \cdots \rmd s_1.
	\end{equation}
	Here, we note that 
	\begin{equation}
		\psi^n_m = 0, \qquad m \geq n+1. 
	\end{equation}
	% Introduce the notation: for $z \geq 0$,  
	% \begin{equation}
	% 	\fl{z} := z_n, \quad z_n \leq z < z_{n+1}. 
	% \end{equation}
	% Define
	% \begin{equation}
	% 	\mathfrak{u}_\text{d}(z) := e^{i \lfloor z \rfloor \vep^\alpha D_\kappa} \psi(0) = \mathfrak{u}(\fl{z}). 
	% \end{equation}
	Using $\mathcal{I}_\rmd$ and $\mathfrak{u}_\rmd$, \cref{eq:duhamel_summation} can be rewritten in a compact form as (see Appendix \ref{append:B} for the details) %%% \gb{I do not understand as what}
	\begin{equation}\label{eq:duhamel_discrete_compact}
		\psi_m^n = \int_{0}^{z_{n}} e^{-is\vep^\alpha D_\kappa} R_\vep (\mathcal{I}_\text{d}^{m-1} \mathfrak{u}_\text{d})(s) \rmd s, \qquad m \geq 1. 
	\end{equation}
	
	Define the error function $e^n := \psi(z_n) - \psi^n$ for $n \geq 0$. Then we have
	\begin{equation}\label{eq:error_eq}
		e^n = \sum_{m \geq 0} e^n_m, \qquad e^n_m:=\psi_m(z_n) - \psi_m^n. 
	\end{equation}
	When $m=0$, we have $e^n_0 = \psi_0(z_n) - \psi^n_0 = 0$ for any $n \geq 0$. When $m=1$, we have
	\begin{equation}\label{eq:e_1}
		e^n_1 = \psi_1 (z_n) - \psi_1^n = \int_0^{z_n} e^{-i s \vep^\alpha D_\kappa} R_\vep (\mathfrak{u}(s) - \mathfrak{u}_\rmd(s)) \rmd s. 
		%		=  \int_0^{z_n} e^{-i s \vep D_\kappa} R_\vep (e^{is\vep D_\kappa} - e^{i\fl{s}\vep D_\kappa})\psi(0) \rmd s. 
	\end{equation}
	For $m \geq 2$, we have, from \cref{eq:duhamel_continuous_compact,eq:duhamel_discrete_compact}, 
	\begin{equation}\label{eq:e_2}
		e^n_m = \psi_m (z_n) - \psi_m^n = \mathcal{E}^n_{m, 1} + \mathcal{E}^n_{m, 2}, 
	\end{equation}
	where
	\begin{align}
		&\mathcal{E}^n_{m, 1}: = \int_{0}^{z_{n}} e^{-is\vep^\alpha D_\kappa} R_\vep (\mathcal{I}_\text{d}^{m-1} (\mathfrak{u} - \mathfrak{u}_\text{d}))(s) \rmd s, \label{eq:E1} \\
		&\mathcal{E}^n_{m, 2} := \sum_{j=0}^{m-2} \mathcal{E}_{m, 2, j}^n = \sum_{j=0}^{m-2} \int_{0}^{z_{n}} e^{-is\vep^\alpha D_\kappa} R_\vep (\mathcal{I}_\text{d}^{m-2 - j} (\mathcal{I} - \mathcal{I}_\text{d}) \mathcal{I}^j \mathfrak{u})(s) \rmd s \label{eq:E2}. 
	\end{align}
	In particular, we also have $e_1^n = \mathcal{E}^n_{1, 1}$. 
	%	In fact, the errors can be classified into two classes: (i) the approximation of $u$ by $u_\rmd$ and (ii) the approximation of $\mathcal{I}$ by $\mathcal{I}_\rmd$. 
	In the following, we estimate $\mathcal{E}^n_{m, 1} \ (m \geq 1)$ and $\mathcal{E}^n_{m, 2} \ (m \geq 2)$. The final error estimate can be obtained by the following stability result, which is a direct consequence of the algebra, isometry, and scaling-invariant properties recalled at the end of the introduction. 
	\begin{lemma}\label{lem:statbility_Id}
		For any $\phi \in L^\infty([0, T]; X)$ and $m \geq 0$, we have
		%		\begin{equation*}
			%			\| (\mathcal{I}_\rmd^m \phi)(z) \| \leq \| R \|^m \int_0^{\fl{z}} \int_0^{\fl{s_1}} \cdots \int_0^{\fl{s_{m-1}}} \| \phi(s_m) \| \rmd s_m \cdots \rmd s_2 \rmd s_1, \quad z \geq 0. 
			%		\end{equation*}
		%		Moreover, we have
		\begin{equation*}
			\| (\mathcal{I}^m \phi)(z) \| + \| (\mathcal{I}_\rmd^m \phi)(z) \| \leq 2\frac{(\| R \|T)^m }{m!} \sup_{0 \leq s \leq z} \| \phi(s) \|, \qquad 0 \leq z \leq T.
		\end{equation*}
        %%%\gb{Is this $\lesssim$ or $\leq 2$ since there is 2 terms now?}
	\end{lemma}
	
	\subsection{Estimates for $\mathcal{E}^n_{m, 1}$}
	In this subsection, we estimate $\mathcal{E}^n_{m, 1}$ \cref{eq:E1}. 
    
    % As an immediate corollary of \cref{lem:statbility_Id,prop:est1_discrete}, we have the following. 
	\begin{proposition}\label{cor:E1_est}
		Under assumptions \cref{eq:assumption}, we have, 
		\begin{equation*}
			\| \mathcal{E}_{m, 1}^n \| \lesssim \sup_{0 \leq z \leq z_n} \| (\mathcal{I}_\rmd(\mathfrak{u} - \mathfrak{u}_\rmd))(z) \| \frac{C^{m-1}}{(m-1)!}, \qquad n \geq 0, \quad m \geq 1. 
		\end{equation*}
	\end{proposition}
	
	\begin{proof}
		Recalling \cref{eq:E1} and \cref{eq:Id_def}, we have
		\begin{equation*}
			\mathcal{E}_{m, 1}^n = e^{-i z_n \vep^\alpha D_\ka} (\mathcal{I}_\rmd^m(\mathfrak{u} - \mathfrak{u}_\rmd))(z_n). 
		\end{equation*}
		% Then by the triangle inequality, using \cref{lem:statbility_Id,prop:est1_discrete}, we have
        Then by the triangle inequality, using \cref{lem:statbility_Id}, we have
		\begin{align}
			\| \mathcal{E}_{m, 1}^n \| 
			&= \| (\mathcal{I}_\rmd^m(\mathfrak{u} - \mathfrak{u}_\rmd))(z_n) \| \notag \leq \frac{(\| R \|T)^{m-1}}{(m-1)!} \sup_{0 \leq s \leq z_n} \| (\mathcal{I}_\rmd(\mathfrak{u} - \mathfrak{u}_\rmd))(s) \|,
            % \lesssim \tau \vep^{1 + \frac{\ka - 1}{\ka}\alpha}(1 + \tau^{1-\frac{1}{\ka}} \ln \vep^{-1}) \frac{(\| R \|T)^{m-1}}{(m-1)!}, 
		\end{align}
		which concludes the proof. 
%		Thus, the estimate reduces to the estimate of $\| (\mathcal{I}_\rmd(\mathfrak{u} - \mathfrak{u}_\rmd))(z) \|$ for $0 \leq z \leq z_n$. Using \cref{prop:est1_discrete} below, we conclude the proof. 
	\end{proof}

    Now we establish the estimate of $\| (\mathcal{I}_\rmd(\mathfrak{u} - \mathfrak{u}_\rmd))(z) \|$, which is the discrete counterpart of \cref{prop:est1_continuous}. However, unlike the continuous case, for the discrete group $e^{i k \tau \Phi} \  (k \in \Z)$, we have a much larger resonant set, i.e., $\Phi \in 2 \pi \Z / \tau$, which introduces additional technical difficulties compared to the continuous case. 
	\begin{lemma}\label{prop:est1_discrete}
		Under assumptions \cref{eq:assumption}, we have 
		\begin{equation}\label{eq:prop_est1_disc_1}
			\| \mathcal{I}_\rmd (\mathfrak{u} - \mathfrak{u}_\rmd) (z) \| \lesssim \tau \vep^\alpha, \qquad 0 \leq z \leq T. 
		\end{equation}
		Moreover, assuming, in addition, \cref{eq:assumption_2}, we have
		\begin{equation}\label{eq:prop_est1_disc_2}
			\| \mathcal{I}_\rmd (\mathfrak{u} - \mathfrak{u}_\rmd) (z) \| \lesssim \tau \vep^{1 + \frac{\ka - 1}{\ka}\alpha}(1 + (\frac{\tau}{\vep^\frac{\alpha}{\ka-1}})^{(\frac{\ka-1}{\ka})^2} \ln \vep^{-1}), \qquad 0 \leq z \leq T.  
		\end{equation}
		Furthermore, assuming, in addition, \cref{eq:assumption_3}, we have
		\begin{equation}\label{eq:prop_est1_disc_3}
			\| \mathcal{I}_\rmd (\mathfrak{u} - \mathfrak{u}_\rmd) (z) \| \lesssim \tau \vep^{1 + \frac{\ka - 1}{\ka}\alpha}(1 + \tau^{1-\frac{1}{\ka}} \ln \vep^{-1}), \qquad 0 \leq z \leq T.
		\end{equation}
	\end{lemma}
	
	\begin{proof}
		We assume that $z_n \leq z < z_{n+1}$ for some $n \geq 0$. Recalling \cref{eq:Id_def}, we have
		\begin{align}\label{eq:J_def}
			e^{-i\fl{z} \vep^\alpha D_\ka}\mathcal{I}_\rmd (\mathfrak{u} - \mathfrak{u}_\rmd) (z) 
			&= e^{-i z_n \vep^\alpha D_\ka} \mathcal{I}_\rmd (\mathfrak{u} - \mathfrak{u}_\rmd) (z_n) \notag \\ 
			&=\sum_{k=0}^{n-1} \int_{0}^{\tau} e^{- i(z_k+s) \vep^\alpha D_\kappa} R_\vep (e^{ i(z_k + s) \vep^\alpha D_\kappa} - e^{i z_k \vep^\alpha D_\kappa}) \psi(0) \rmd s \notag \\
			&= \sum_{k=0}^{n-1} \int_{0}^{\tau} e^{- i(z_k+s) \vep^\alpha D_\kappa} R_\vep e^{i z_k \vep^\alpha D_\kappa} \int_0^s e^{i s_1 \vep^\alpha D_\kappa} (i\vep^\alpha D_\kappa) \psi(0) \rmd s_1 \rmd s \notag \\
			&= i\vep^\alpha \sum_{k=0}^{n-1} \int_{0}^{\tau} e^{- i(z_k+s) \vep^\alpha D_\kappa} R_\vep e^{i z_k \vep^\alpha D_\kappa} \int_0^s e^{i s_1 \vep^\alpha D_\kappa} \phi \rmd s_1 \rmd s = : i\vep^\alpha J, 
		\end{align}
		where $\phi: = D_\kappa \psi(0)$. Taking the Fourier transform of $J$ above, by the change of variable, we obtain
		\begin{align}\label{eq:J_def_1}
			\widehat{J}(\xi) 
			&= \int_{\frac{\xi_1}{\vep} + \xi_2 = \xi}	\sum_{k=0}^{n-1} \int_{0}^{\tau} e^{ i (z_k+s) \vep^\alpha P(\xi)} \widehat{R}(\xi_1) e^{- i z_k \vep^\alpha P(\xi_2)}  \int_0^s e^{ -i s_1 \vep^\alpha P(\xi_2)} \widehat{\phi}(\xi_2) \rmd s_1 \rmd s \rmd \xi_1 \notag \\
			&= \int_{\frac{\xi_1}{\vep} + \xi_2 = \xi} \sum_{k=0}^{n-1}  e^{ i z_k \vep^\alpha (P(\xi) - P(\xi_2))} \int_{0}^{\tau} \int_0^s e^{ i s \vep^\alpha (P(\xi) - P(\xi_2))} e^{i (s - s_1) \vep^\alpha P(\xi_2)} \rmd s_1 \rmd s \widehat{R}(\xi_1) \widehat{\phi}(\xi_2) \rmd \xi_1 \notag \\
			&= \int_{\frac{\xi_1}{\vep} + \xi_2 = \xi} \sum_{k=0}^{n-1}  e^{ i k \theta} \int_{0}^{\tau} \int_0^s e^{ i s \vep^\alpha (P(\xi) - P(\xi_2))} e^{i (s - s_1) \vep^\alpha P(\xi_2)} \rmd s_1 \rmd s \widehat{R}(\xi_1) \widehat{\phi}(\xi_2) \rmd \xi_1, 
		\end{align}
		where we defined
		\begin{equation}\label{eq:theta_def_1}
			\theta := \tau \vep^\alpha (P(\frac{\xi_1}{\vep} + \xi_2) - P(\xi_2)) = \tau \Phi(\xi_1, \xi_2) = a f(\xi_1, \xi_2),
		\end{equation}
		with
		\begin{equation}
			a:= \frac{\tau}{\vep^{\kappa-\alpha}}, \quad f(\xi_1, \xi_2) := \sum_{0 \leq 2j < \ka} d_{\ka - 2j} \vep^{2j}(\xi_1 + \vep \xi_2)^{\kappa - 2j} - \vep^\ka P(\xi_2).  
		\end{equation}
		First, a direct estiamte gives
		\begin{align}
			\| \widehat{J}(\xi) \|_{L^1_\xi} \lesssim  \int_{\R^2} n \tau^2 |\widehat{R}(\xi_1)| |\widehat{\phi}(\xi_2)| \rmd \xi_1 \rmd \xi_2 \lesssim \tau \| \widehat{R} \|_{L^1} \| \widehat{\phi} \|_{L^1}, 
		\end{align}
		which proves \cref{eq:prop_est1_disc_1}. 
		
		To further capture the phase cancellation, we consider the following decomposition of $\R^2$: for some $L_0, L_1 \gg 1$ and $0<\delta \ll 1$ to be determined, and some (large) $C$ depending on $\{d_{\ka-2j}\}$, 
		\begin{equation}\label{eq:decomp_1}
			\begin{aligned}
				&\Lambda_0 := \{(\xi_1, \xi_2) \in \R^2 : (\xi_1, \xi_2) \notin (-L_0, L_0) \times (-L_1, L_1)\}, \\
				&\Lambda_1 := \{(\xi_1, \xi_2) \in (-L_0, L_0) \times (-L_1, L_1): |\theta| \leq \pi\}, \\
				&\Lambda_0' := \{(\xi_1, \xi_2) \in (-L_0, L_0) \times (-L_1, L_1): |\theta| > \pi, \  |\xi_1 + \vep \xi_2| \leq C \vep \}, \\
				&\Lambda_2^\delta := \{(\xi_1, \xi_2) \in (-L_0, L_0) \times (-L_1, L_1): |\theta| > \pi, \ |\xi_1 + \vep \xi_2| > C \vep, \  \text{dist}(\theta, 2\pi \mathbb{Z}) \leq \delta \},  \\
				&\Lambda_3^\delta := \{(\xi_1, \xi_2) \in (-L_0, L_0) \times (-L_1, L_1): |\theta| > \pi, \ |\xi_1 + \vep \xi_2| > C \vep, \ \text{dist}(\theta, 2\pi \mathbb{Z}) > \delta \}. 
			\end{aligned}
		\end{equation}
		It follows that $1 = 1_{\Lambda_0} + 1_{\Lambda_1} + 1_{\Lambda_0'} + 1_{\Lambda_2^\delta} + 1_{\Lambda_3^\delta}$ and, recalling \cref{eq:bilinear_multiplier}, 
		\begin{equation}
			J = J_{\Lambda_0} + J_{\Lambda_1} + J_{\Lambda_0'} + J_{\Lambda_2^\delta} + J_{\Lambda_3^\delta}. 
		\end{equation}
        Before estimating each term, we briefly explain the reason for the above decomposition. The region $\Lambda_1$ is the ``small phase" regime, which admits a continuous analogue \cref{prop:est1_continuous}. $\Lambda_2^\delta$ and $\Lambda_3^\delta$ are the ``large phase" regime that has no continuous counterparts. Here $\Lambda_2^\delta$ is the numerical resonance regime and $\Lambda_3^\delta$ is the (numerical) non-resonance regime. These two regimes contribute an error that is comparable to the error in $\Lambda_1$, and need to be treated carefully. The region $\Lambda_0$ truncates the frequencies to a bounded domain such that the ``large phase" regime is now bounded, and the error in this region is controlled by the decay assumption \cref{eq:assumption_2}.  $\Lambda_0'$ separates the region where we cannot identify a dominant term in the phase, while its contribution to the error is small (of order $O(\tau\vep)$). In addition, the estimates in the ``large phase" regime $\Lambda_2^\delta$ and $\Lambda_3^\delta$ do not rely on the phase structure \cref{lem:phase}, and thus apply uniformly to both even and odd $\ka$. 
		
        For $J_{\Lambda_0}$, we have, by assumptions \cref{eq:assumption_2}, 
		\begin{align}
			\| \widehat{J}_{\Lambda_0} \|_{L^1} 
			&\lesssim n \tau^2 \int_{\R^2} 1_{\Lambda_0} |\widehat{R}(\xi_1)||\widehat{\phi}(\xi_2)| \rmd \xi_1 \rmd \xi_2 \notag \\
			&\lesssim \tau \int_{\R^2} 1_{|\xi_1| \geq L_0} |\widehat{R}(\xi_1)||\widehat{\phi}(\xi_2)| \rmd \xi_1 \rmd \xi_2 + \tau \int_{\R^2} 1_{|\xi_2| \geq L_1} |\widehat{R}(\xi_1)||\widehat{\phi}(\xi_2)| \rmd \xi_1 \rmd \xi_2 \notag \\
			&\lesssim \tau \| \widehat{\phi} \|_{L^1} \frac{1}{L_0^\sigma} + \tau \| \widehat{R} \|_{L^1} \frac{1}{L_1^{\sigma}} \lesssim \frac{\tau}{L_0^\sigma} + \frac{\tau}{L_1^\sigma}. 
		\end{align}
		
		Next, we estimate $J_{\Lambda_1}$, which resembles the estimate in \cref{prop:est1_continuous} by using the trivial bound 
        \begin{equation*}
            |\int_{0}^{\tau} \int_0^s e^{ i s \vep^\alpha (P(\xi) - P(\xi_2))} e^{i (s - s_1) \vep^\alpha P(\xi_2)} \rmd s_1 \rmd s| \lesssim \tau^2, 
        \end{equation*} 
        and noting that when $(\xi_1, \xi_2) \in \Lambda_1$, we have $|\theta| \leq \pi$ and thus
		\begin{equation}
			| \sum_{k=0}^{n-1}  e^{ i k \theta} | = |\frac{1 - e^{i n \theta}}{1 - e^{i \theta}}|  \lesssim \frac{1}{|\theta|}. 
		\end{equation} 
		Then we have
		\begin{equation}
			\| \widehat{J_{\Lambda_1}} \|_{L^1} \lesssim \tau \vep^{1 - \frac{1}{\ka}\alpha}. 
		\end{equation}
		For $J_{\Lambda_0'}$, from \cref{eq:J_def_1}, the estimate goes simply as follows:
		\begin{align}
			\| \widehat{J}_{\Lambda_0'} \|_{L^1} 
			&\lesssim n \tau^2 \int_{\R^2} 1_{\Lambda_0'} |\widehat{R}(\xi_1)||\widehat{\phi}(\xi_2)| \rmd \xi_1 \rmd \xi_2 \lesssim \tau \int_{\R^2} 1_{|\xi_1 + \vep \xi_2| \leq C\vep} |\widehat{R}(\xi_1)||\widehat{\phi}(\xi_2)| \rmd \xi_1 \rmd \xi_2 \notag \\
			&= \tau \int_{\R} |\widehat{\phi}(\xi_2)| \int_{\R} 1_{|\xi_1 + \vep \xi_2| \leq C\vep} |\widehat{R}(\xi_1)| \rmd \xi_1 \rmd \xi_2 \lesssim \tau \vep \| \widehat{R} \|_{L^\infty} \| \widehat{\phi} \|_{L^1}. 
		\end{align}
		
		In the following, we estimate $J_{\Lambda^\delta_2}$ and $J_{\Lambda^\delta_3}$. When $(\xi_1, \xi_2) \in \Lambda^\delta_2$,  we have
		\begin{equation}
			\Lambda^\delta_2 = \cup_{|l| \geq 1} \Lambda^\delta_{2, l} := \cup_{|l| \geq 1} \{(\xi_1, \xi_2) \in (-L_0, L_0) \times (-L_1, L_1):|\xi_1 + \vep \xi_2| > C \vep, \  |\theta - 2 \pi l| \leq \delta\}.  
		\end{equation}
		  Then we have, from \cref{eq:J_def_1}, 
		\begin{equation}
			\|\widehat {J}_{\Lambda^\delta_2}(\xi)\|_{L^1_\xi} \leq \sum_{|l| \geq 1} \|\widehat {J}_{\Lambda^\delta_{2, l}}(\xi) \|_{L^1_\xi} \lesssim \sum_{|l| \geq 1} \tau \int_{\R^2} 1_{\Lambda^\delta_{2, l}} |\widehat{R}(\xi_1)| |\widehat{\phi}(\xi_2)| \rmd \xi_1 \rmd \xi_2. 
		\end{equation}
        We note that as $|\theta| \lesssim a (L_0 + \vep L_1)^\ka$, $\Lambda_{2, l}^\delta = \emptyset$ when $|l| \gtrsim a (L_0 + \vep L_1)^\ka$. Hence, there are only finitely many terms in the summation, which is the reason why we truncate the frequency in $\Lambda_0$. We now use the co-area formula to calculate the integral. Note that, with $f_{\xi_\ell} := \partial_{\xi_\ell} f(\xi_1, \xi_2)$ for $\ell=1, 2$, 
		\begin{equation}\label{eq:grad_theta}
			|\nabla \theta| = a \sqrt{f_{\xi_1}^2 + f_{\xi_2}^2}, \quad f_{\xi_1} =  \sum_{0 \leq 2j < \ka} (\ka - 2j) d_{\ka - 2j} \vep^{2j}(\xi_1 + \vep \xi_2)^{\kappa - 2j - 1}. 
		\end{equation}
		As $|\xi_1 + \vep \xi_2| \geq C \vep$, we have
		\begin{equation}
			|f_{\xi_1}| \gtrsim |\xi_1 + \vep \xi_2|^{\ka-1} - \sum_{2 \leq 2j < \ka} |d_{\ka - 2j}| \vep^{2j} |\xi_1 + \vep \xi_2|^{\ka-2j-1} \gtrsim |\xi_1 + \vep \xi_2|^{\ka-1} > 0. 
		\end{equation}
		In particular, we have $|\nabla \theta| \neq 0$. Then by a co-area formula, 
		we get
		\begin{equation}\label{eq:J2_def}
			J_{2, l} := \int_{\R^2} 1_{\Lambda^\delta_{2, l}} |\widehat{R}(\xi_1)| |\widehat{\phi}(\xi_2)| \rmd \xi_1 \rmd \xi_2 = \int_{2 \pi l - \delta}^{2 \pi l + \delta} \rmd u \int_{\theta = u} 1_{\Lambda^\delta_{2, l}} \frac{|\widehat{R}(\xi_1)| |\widehat{\phi}(\xi_2)|}{|\nabla \theta|} \rmd \mathcal{H}^1(\xi_1, \xi_2), 
		\end{equation}
		where $\mathcal{H}^1(\xi_1, \xi_2)$ is the arclength measure, and we have, along the curve $\theta = u$, 
		\begin{equation}\label{eq:dH}
			\rmd\mathcal{H}^1(\xi_1, \xi_2) = \sqrt{1 + \left(\frac{\rmd \xi_1}{\rmd \xi_2}\right)^2} \rmd \xi_2 = \sqrt{1 + \left(\frac{f_{\xi_2}}{f_{\xi_1}}\right)^2} \rmd \xi_2. 
		\end{equation}
		Plugging \cref{eq:grad_theta,eq:dH} into \cref{eq:J2_def}, we obtain
		\begin{equation}\label{eq:J2l_coarea}
			J_{2, l} = \int_{2 \pi l - \delta}^{2 \pi l + \delta} \rmd u \int_{\theta = u} 1_{\Lambda^\delta_{2, l}} \frac{|\widehat{R}(\xi_1)| |\widehat{\phi}(\xi_2)|}{a|f_{\xi_1}|} \rmd \xi_2.
		\end{equation}
		On the curve $\theta = u$, recalling \cref{eq:theta_def_1}, we have
		\begin{equation}
			\sum_{0 \leq 2j <\ka} d_{\ka - 2j} \vep^{2j}(\xi_1 + \vep \xi_2)^{\kappa-2j} = (\xi_1 + \vep \xi_2)^{\kappa} \sum_{0 \leq 2j <\ka} \frac{d_{\ka - 2j} \vep^{2j}}{(\xi_1 + \vep \xi_2)^{2j}} = \frac{u}{a} + \vep^\kappa P(\xi_2), 
		\end{equation}
		which implies, recalling $f_{\xi_1}$ in \cref{eq:grad_theta} and noting $\sum_{0 \leq 2j <\ka} (\ka - 2j) \frac{d_{\ka - 2j} \vep^{2j}}{(\xi_1 + \vep \xi_2)^{2j}} \gtrsim 1$ when $|\xi_1 + \vep \xi_2| \geq C \vep$, 
		\begin{align}\label{eq:f_xi1}
			|f_{\xi_1}| 
			&= |\xi_1 + \vep \xi_2|^{\kappa - 1} \sum_{0 \leq 2j < \ka} (\ka-2j) \frac{d_{\ka - 2j} \vep^{2j}}{(\xi_1 + \vep \xi_2)^{2j}} \notag \\
			&\gtrsim |\xi_1 + \vep \xi_2|^{\kappa-1} (\sum_{0 \leq 2j <\ka} \frac{d_{\ka - 2j} \vep^{2j}}{(\xi_1 + \vep \xi_2)^{2j}})^\frac{\ka-1}{\ka} %\notag \\
			%&
              \ = \ |\frac{u}{a} + \vep^\ka P(\xi_2)|^\frac{\ka-1}{\ka}.
		\end{align}
		We note that \cref{eq:assumption_3} implies $\int_{|\xi| \geq L} |\widehat{\phi}(\xi)| \rmd \xi \lesssim \frac{1}{L^{\ka-1}}$ for all $L \geq 1$. From \cref{eq:J2l_coarea}, using \cref{eq:f_xi1} and \cref{lem:int_aux}, noting $|u| \sim 2 \pi |l|$, we obtain
		\begin{align}
			J_{2, l} 
			&\lesssim \| \widehat{R} \|_{L^\infty} \int_{2 \pi l - \delta}^{2 \pi l + \delta} \frac{1}{a^\frac{1}{\ka} |u|^\frac{\ka-1}{\ka}} \rmd u \int_{-L_1}^{L_1} \frac{|\widehat{\phi}(\xi_2)|}{|1 + \frac{a\vep^\ka}{u} P(\xi_2)|^\frac{\ka-1}{\ka}} \rmd \xi_2 \notag \\
			&\lesssim 
			\left\{
				\begin{aligned}
					&\int_{2 \pi l - \delta}^{2 \pi l + \delta} \frac{1}{a^\frac{1}{\ka} |u|^\frac{\ka-1}{\ka}} \rmd u \lesssim \frac{\delta}{a^\frac{1}{\ka} |l|^\frac{\ka-1}{\ka}} , && \text{with \cref{eq:assumption_3}} \\
					&(a \vep^\ka)^{-\frac{\ka-1}{\ka^2}}\int_{2 \pi l - \delta}^{2 \pi l + \delta} \frac{|u|^\frac{\ka-1}{\ka^2}}{a^\frac{1}{\ka} |u|^\frac{\ka-1}{\ka}} \rmd u \lesssim \frac{\delta \vep^{1 - \frac{\alpha}{\ka}} \vep^{-\alpha\frac{\ka-1}{\ka^2}}} {\tau^\frac{2\ka-1}{\ka^2}|l|^{1 - \frac{2\ka-1}{\ka^2}}}, && \text{without \cref{eq:assumption_3}}. 
				\end{aligned}
			\right.
		\end{align}
		Then, with \cref{eq:assumption_3}, we have
		\begin{equation}
			\|\widehat {J}_{\Lambda^\delta_2}(\xi)\|_{L^1_\xi} \lesssim \frac{\tau \delta}{a^\frac{1}{\ka}} \sum_{|l| \geq 1} 1_{|l| \leq a (L_0 + \vep L_1)^\ka}\frac{1}{|l|^\frac{\ka-1}{\ka}} \lesssim \tau \delta (L_0 + \vep L_1), 
		\end{equation}
		and, without \cref{eq:assumption_3}, we have
		\begin{equation}
			\|\widehat {J}_{\Lambda^\delta_2}(\xi)\|_{L^1_\xi} \lesssim \frac{\tau \delta \vep^{1 - \frac{\alpha}{\ka}} \vep^{-\alpha\frac{\ka-1}{\ka^2}}}{\tau^\frac{2\ka-1}{\ka^2}} \sum_{|l| \geq 1} \frac{1_{|l| \leq a (L_0 + \vep L_1)^\ka}}{|l|^{1- \frac{2\ka-1}{\ka^2}}} \lesssim \tau \delta \vep^{1 - \frac{\alpha}{\ka}} \vep^{-\alpha\frac{\ka-1}{\ka^2}} \vep^{-\frac{(\ka-\alpha)(2\ka-1)}{\ka^2}} (L_0 + \vep L_1)^\frac{2\ka-1}{\ka}.  
		\end{equation}
		
		Finally, we estimate $J_{\Lambda_{3}^\delta}$. We have, 
		\begin{equation}
			\Lambda_{3}^\delta = \cup_{|l| \geq 1} \Lambda_{3, l}^\delta, \quad \Lambda_{3, l}^\delta := \{(\xi_1, \xi_2) \in \Lambda_3^\delta:|\xi_1 + \vep \xi_2| > C \vep, \  \theta \in (2\pi l-\pi, 2 \pi l - \delta) \cup (2 \pi l + \delta, 2 \pi l+\pi) \}. 
		\end{equation}
		When $(\xi_1, \xi_2) \in \Lambda_{3, l}^\delta$, we have $|\theta - 2 \pi l| \in (\delta, \pi)$ and thus
		\begin{equation}\label{eq:sum_by_part}
			|1 - e^{i \theta}| = 2|\sin(\theta/2)| = 2|\sin((\theta - 2 \pi l)/2)| = 2\sin(|\theta - 2 \pi l|/2) \gtrsim |\theta - 2 \pi l|. 
		\end{equation}
		Then we have
		\begin{equation}\label{eq:sum_by_part_est}
			\left| \sum_{k=0}^{n-1} e^{i z_k \Phi} \right| = \left| \sum_{k=0}^{n-1} e^{i k \theta} \right| = \left| \frac{1-e^{i n \theta}}{1-e^{i\theta}} \right| \lesssim \frac{1}{|\theta - 2 \pi l|}. 
		\end{equation}
		Moreover, for the integration in \cref{eq:J_def_1}, by integration by parts, noting that $\theta = \tau \vep^\alpha(P(\xi) - P(\xi_2)) \neq 0$ when $(\xi_1, \xi_2) \in \Lambda_{3}^\delta$, we have
		\begin{align}\label{eq:inner_int}
			&\int_{0}^{\tau} \int_0^s e^{ i s \vep^\alpha (P(\xi) - P(\xi_2))} e^{i (s - s_1) \vep^\alpha P(\xi_2)} \rmd s_1 \rmd s \notag \\
			&=\int_{0}^{\tau} e^{ i s \vep^\alpha (P(\xi) - P(\xi_2))}  \int_0^s  e^{i s_1 \vep^\alpha P(\xi_2)} \rmd s_1 \rmd s \notag \\
			&=\frac{1}{i\vep^\alpha (P(\xi) - P(\xi_2))} \int_{0}^{\tau} \partial_s (e^{ i s \vep^\alpha (P(\xi) - P(\xi_2))}) \int_0^s  e^{i s_1 \vep^\alpha P(\xi_2)} \rmd s_1 \rmd s \notag \\
			&=\frac{\tau}{i \theta} \left( e^{ i \tau \vep^\alpha (P(\xi) - P(\xi_2))} \int_0^\tau  e^{i s_1 \vep^\alpha P(\xi_2)} \rmd s_1 - \int_{0}^{\tau} e^{ i s \vep^\alpha (P(\xi) - P(\xi_2))}  e^{i s \vep^\alpha P(\xi_2)} \rmd s \right). 
		\end{align}
		It follows from the triangle inequality that 
		\begin{equation}\label{eq:inner_int_est_1}
			\left| \int_{0}^{\tau} \int_0^s e^{ i s \vep^\alpha (P(\xi) - P(\xi_2))} e^{i (s - s_1) \vep^\alpha P(\xi_2)} \rmd s_1 \rmd s \right| \lesssim \frac{\tau^2}{|\theta|}. 
		\end{equation}
		Using the above two estimates \cref{eq:sum_by_part_est,eq:inner_int_est_1}, from \cref{eq:J_def_1}, we have, noting $|\theta| \sim |l|$,  
		\begin{align}
			\|\widehat {J}_{\Lambda^\delta_3}(\xi)\|_{L^1_\xi} 
			&\leq \sum_{|l| \geq 1} \|\widehat {J}_{\Lambda^\delta_{3, l}}(\xi) \|_{L^1_\xi} \notag \\
			&\lesssim \tau^2 \sum_{|l| \geq 1} \int_{\R^2} 1_{\Lambda_{3, l}^\delta} \frac{|\widehat{R}(\xi_1)| |\widehat{\phi}(\xi_2)|}{|\theta|| \theta - 2 \pi l |} \rmd \xi_1 \rmd \xi_2 
            %\notag 
            %\\
		%&
            \lesssim \tau^2 \sum_{|l| \geq 1} \frac{1}{|l|} \int_{\R^2} 1_{\Lambda_{3, l}^\delta} \frac{|\widehat{R}(\xi_1)| |\widehat{\phi}(\xi_2)|}{| \theta - 2 \pi l |} \rmd \xi_1 \rmd \xi_2.
		\end{align}
		Using again the co-area formula, we obtain, with $\Omega_{l}^\delta:=(2\pi l-\pi, 2 \pi l - \delta) \cup (2 \pi l + \delta, 2 \pi l+\pi)$, 
		\begin{align}
			J_{3, l} := \int_{\R^2} 1_{\Lambda_{3, l}^\delta} \frac{|\widehat{R}(\xi_1)| |\widehat{\phi}(\xi_2)|}{| \theta - 2 \pi l |} \rmd \xi_1 \rmd \xi_2 = \int_{\Omega_l^\delta} \frac{1}{| u - 2 \pi l |} \rmd u \int_{\theta = u} 1_{\Lambda_{3, l}^\delta} \frac{|\widehat{R}(\xi_1)| |\widehat{\phi}(\xi_2)|}{|\nabla \theta|} \rmd \mathcal{H}^1(\xi_1, \xi_2).  
		\end{align}
		By the same change of measure \cref{eq:dH}, using \cref{eq:f_xi1} and \cref{lem:int_aux}, we have
		\begin{align}
			J_{3, l} 
			&=\int_{\Omega_l^\delta} \frac{1}{|2 \pi l - u|} \rmd u \int_{-L_1}^{L_1} 1_{\Lambda_{3, l}^\delta} \frac{|\widehat{R}(\xi_1)| |\widehat{\phi}(\xi_2)|}{a|f_{\xi_1}|} \rmd \xi_2 \notag \\
			&\lesssim \frac{\| \widehat{R} \|_{L^\infty}}{a^{\frac{1}{\ka}} | l |^{\frac{\ka-1}{\ka}}} \int_{\Omega_l^\delta} \frac{1}{|2 \pi l - u|} \rmd u \int_{-L_1}^{L_1} \frac{|\widehat{\phi}(\xi_2)|}{|1 + \frac{a\vep^\ka}{u} P(\xi_2)|^\frac{\ka-1}{\ka}} \rmd \xi_2 \lesssim \left\{
			\begin{aligned}
				& \frac{|\ln \delta|}{a^\frac{1}{\ka} |l|^{1 - \frac{1}{\ka}}}, &&\text{with \cref{eq:assumption_3}},\\
				& \frac{|\ln \delta| \vep^{1 - \frac{\alpha}{\ka}} \vep^{-\alpha\frac{\ka-1}{\ka^2}}} {\tau^\frac{2\ka-1}{\ka^2}|l|^{1 - \frac{2\ka-1}{\ka^2}}}, &&\text{without \cref{eq:assumption_3}}. 
			\end{aligned}
			\right. 
		\end{align}
		It then follows that, with \cref{eq:assumption_3}, 
		\begin{equation}
			\|\widehat {J}_{\Lambda^\delta_3}(\xi)\|_{L^1_\xi} \lesssim \frac{\tau^2}{a^\frac{1}{\ka}} \sum_{|l| \geq 1} \frac{1}{|l|^{2-\frac{1}{\kappa}}} |\ln \delta| \lesssim \tau^{2-\frac{1}{\kappa}} \vep^{1 - \frac{\alpha}{\ka}} |\ln \delta|,   
		\end{equation}
		and without \cref{eq:assumption_3}, 
		\begin{equation}
			\|\widehat {J}_{\Lambda^\delta_3}(\xi)\|_{L^1_\xi} \lesssim \tau^{2 - \frac{2\ka-1}{\ka^2}} \vep^{1-\frac{\alpha}{\ka}} \vep^{-\alpha \frac{\ka-1}{\ka^2}} \sum_{|l| \geq 1} \frac{1}{|l|^{2-\frac{2\ka-1}{\kappa^2}}} |\ln \delta| \lesssim \tau\vep^{1 - \frac{\alpha}{\ka}} (\frac{\tau}{\vep^\frac{\alpha}{\ka-1}})^{(\frac{\ka-1}{\ka})^2}  |\ln \delta|,   
		\end{equation}
		Finally, choosing $L_0 = L_1 = \vep^{- \frac{1}{\sigma}}$, and $\delta = \vep^{1 + \frac{2}{\sigma}+10}$, we have
		\begin{align}
			\| J \| 
			&\leq \| J_{\Lambda_0} \| + \| J_{\Lambda_1} \| + \| J_{\Lambda_0'} \| + \| J_{\Lambda_2^\delta} \| +  \| J_{\Lambda_3^\delta} \| \notag \\
			&\lesssim \left\{
			\begin{aligned}
				&\tau \vep + \tau \vep^{1 - \frac{\alpha}{\ka}} + \tau \vep + \tau^{2- \frac{1}{\ka}} \vep^{1-\frac{\alpha}{\ka}} |\ln \vep| \lesssim \tau \vep^{1 - \frac{\alpha}{\ka}}(1 + \tau^{1-\frac{1}{\ka}}|\ln \vep|), &&\text{with \cref{eq:assumption_3}}\\
				&\tau \vep + \tau \vep^{1 - \frac{\alpha}{\ka}} + \tau \vep + \tau \vep^{1-\frac{\alpha}{\ka}} (\frac{\tau^{\ka-1}}{\vep^\alpha})^\frac{\ka-1}{\ka^2} |\ln \vep| \lesssim \tau \vep^{1 - \frac{\alpha}{\ka}}(1 + (\frac{\tau}{\vep^\frac{\alpha}{\ka-1}})^{(\frac{\ka-1}{\ka})^2}|\ln \vep|), && \text{without \cref{eq:assumption_3}}. 
			\end{aligned} 
			\right.
		\end{align}
		which completes the proof of \cref{eq:prop_est1_disc_2,eq:prop_est1_disc_3} from \cref{eq:J_def}. 
	\end{proof}
	
	\subsection{Estimates for $\mathcal{E}_{m, 2}^n$}
%	\begin{lemma}\label{lem:stability_I}
%		For $m \geq 1$, we have
%		\begin{equation*}
%			\| (\mathcal{I}^m \phi)(z) \| \leq \| R \|^m \int_0^z \int_0^{s_1} \cdots \int_0^{s_{m-1}} \| \phi(s_m) \| \rmd s_m \cdots \rmd s_2 \rmd s_1, \quad z \geq 0. 
%		\end{equation*}
%		Moreover, we have
%		\begin{equation*}
%			\| (\mathcal{I}^m \phi)(z) \| \leq \frac{\|R\|^m}{m!}\sup_{0 \leq s \leq z} \| \phi(s) \|. 
%		\end{equation*}
%	\end{lemma}
	
	As a direct application of \cref{lem:int_by_part}, we have the following identity. 
	\begin{lemma}\label{lem:Lap_I}
		For any $\phi \in C([0, \infty); X^\ka) \cap C^1([0, \infty); X)$, we have
		\begin{equation*}
			i \vep^\alpha D_\kappa (\mathcal{I}\phi)(z) = (\mathcal{I} \partial_z \phi)(z) + e^{i z \vep^\alpha D_\kappa} R_\vep \phi(0) - R_\vep \phi(z), \quad z \geq 0.  
		\end{equation*}
	\end{lemma}
	
	\begin{proposition}\label{prop:error_equation_2}
		For $m \geq 2$ and $0 \leq j \leq m-2$, we have
		\begin{equation}\label{eq:E2_decomp}
			\mathcal{E}^n_{m, 2, j} = e^{-iz_n \vep^\alpha D_\kappa} (\mathcal{I}_\mathrm{d}^{m-2-j} (\mathcal{J}_{j, 1} + \mathcal{J}_{j, 2}))(z_n), \quad n \geq 0, 
		\end{equation}
		where 
		\begin{align}
			&\mathcal{J}_{j, 1}(z):= i \vep^\alpha e^{i z_{\tilde n} \vep^\alpha D_\kappa} \sum_{k=0}^{{\tilde n}-1} \int_{z_k}^{z_{k+1}} e^{-i s \vep^\alpha D_\kappa} R_\vep \int_{z_k}^s (\mathcal{I}^{j+1} D_\kappa \mathfrak{u})(s_1) \rmd s_1 \rmd s, \quad z_{\tilde n} \leq z <z_{{\tilde n}+1}, \label{eq:Jj1_def} \\
			&\mathcal{J}_{j, 2}(z):= e^{i z_{\tilde n} \vep^\alpha D_\kappa} \sum_{k=0}^{{\tilde n}-1} \int_{z_k}^{z_{k+1}} e^{-i s \vep^\alpha D_\kappa} R_\vep \int_{z_k}^s (\mathcal{I}^jv)(s_1) \rmd s_1 \rmd s, \quad z_{\tilde n} \leq z <z_{{\tilde n}+1}. \label{eq:Jj2_def}
		\end{align}
	\end{proposition}
	
	\begin{proof}
		Recalling \cref{eq:E2,eq:Id_def}, we have
		\begin{equation}\label{Enm2j}
			\mathcal{E}^n_{m, 2, j} = e^{-iz_n \vep^\alpha D_\kappa} (\mathcal{I}_\text{d}^{m-1-j} (\mathcal{I} - \mathcal{I}_\rmd) \mathcal{I}^j \mathfrak{u})(z_n) = e^{-iz_n \vep^\alpha D_\kappa} (\mathcal{I}_\text{d}^{m-2-j} \mathcal{I}_\rmd (\mathcal{I} - \mathcal{I}_\text{d}) \mathcal{I}^j \mathfrak{u})(z_n). 
		\end{equation}
		For any $z \geq 0$, with $\fl{z} = z_{\tilde n}$, recalling \cref{eq:Id_def}, we have 
		\begin{align}\label{eq:I-Id}
			\mathcal{I}_\rmd((\mathcal{I} - \mathcal{I}_\text{d}) \mathcal{I}^j \mathfrak{u})(z) 
			&=e^{i z_{\tilde n} \vep^\alpha D_\kappa} \int_0^{z_{\tilde n}} e^{-i s \vep^\alpha D_\kappa} R_\vep ((\mathcal{I} - \mathcal{I}_\text{d}) \mathcal{I}^j \mathfrak{u})(s) \rmd s \notag \\
			&= e^{i z_{\tilde n} \vep^\alpha D_\kappa} \sum_{k=0}^{{\tilde n}-1} \int_{z_k}^{z_{k+1}} e^{-i s \vep^\alpha D_\kappa} R_\vep (\mathcal{I}^{j+1} \mathfrak{u}(s) - \mathcal{I}^{j+1} \mathfrak{u}(z_k)) \rmd s \notag \\
			&= e^{i z_{\tilde n} \vep^\alpha D_\kappa} \sum_{k=0}^{{\tilde n}-1} \int_{z_k}^{z_{k+1}} e^{-i s \vep^\alpha D_\kappa} R_\vep \int_{z_k}^s \partial_z (\mathcal{I}^{j+1} \mathfrak{u})(s_1) \rmd s_1 \rmd s. 
		\end{align}
		Recalling \cref{eq:I_def}, we have
		\begin{equation}
			\partial_z (\mathcal{I} \phi)(z) = i \vep^\alpha D_\kappa (\mathcal{I}\phi)(z) + R_\vep \phi(z)
		\end{equation}
		Using \cref{lem:Lap_I}, we obtain
		\begin{equation}\label{eq:chain}
			\partial_z (\mathcal{I} \phi)(z) = (\mathcal{I}\partial_z \phi)(z) + e^{i z \vep^\alpha D_\kappa} R_\vep \phi(0). 
		\end{equation}
		Then we have, by \cref{eq:chain} and noting that $(\mathcal{I} \phi)(0) = 0$ for any $\phi = \phi(z, x)$, 
		\begin{align}\label{eq:partial_Ij}
			\partial_z (\mathcal{I}^{j+1} \mathfrak{u})(z) 
			&= \partial_z (\mathcal{I} (\mathcal{I}^{j} \mathfrak{u}))(z) \notag \\
			&= (\mathcal{I} \partial_z \mathcal{I}^{j} \mathfrak{u})(z) + e^{i z \vep^\alpha D_\kappa} R_\vep (\mathcal{I}^j \mathfrak{u})(0) = (\mathcal{I} \partial_z \mathcal{I}^{j} \mathfrak{u})(z) 
            %\notag \\	&
            \ =\ \cdots  \ = \ (\mathcal{I}^j \partial_z \mathcal{I} \mathfrak{u})(z) \notag \\
			%			&= (\mathcal{I}^{j+1} \partial_z u)(z) + \mathcal{I}^j (e^{i \cdot \vep^\alpha D_\kappa} R_\vep u(0))(z) \notag \\
			&= i \vep^\alpha (\mathcal{I}^{j+1} D_\kappa \mathfrak{u})(z) + (\mathcal{I}^j v)(z), \quad j \geq 0. 
		\end{align}
		Plugging \cref{eq:partial_Ij} into \cref{eq:I-Id}, we get
		\begin{equation}\label{eq:I-Id_simp}
			\mathcal{I}_\rmd((\mathcal{I} - \mathcal{I}_\text{d}) \mathcal{I}^j \mathfrak{u})(z) = \mathcal{J}_{j, 1}(z) + \mathcal{J}_{j, 2}(z), \quad z \geq 0, 
		\end{equation}
		which plugged into \cref{Enm2j} concludes the proof.
	\end{proof}
	
	For the first part of $\mathcal{E}^n_{m, 2}$ in \cref{eq:E2_decomp} which involves $\mathcal{J}_{j, 1}$, we have the following estimate which is a direct corollary of \cref{prop:est1_continuous}. 
	\begin{corollary}\label{cor:E2_est_1}
		Under assumptions \cref{eq:assumption}, 
        %%\gb{Something missing here?} 
        we have
		\begin{equation*}
			\sum_{j=0}^{m-2} \| e^{-iz_n \vep^\alpha D_\kappa} (\mathcal{I}_\mathrm{d}^{m-2-j} \mathcal{J}_{j, 1})(z_n) \| \lesssim \tau \vep^{1 + \frac{\ka-1}{\ka}\alpha} \frac{C^{m-2}}{(m-2)!}, 
		\end{equation*}
		where $C$ depends on $\| R \|$ and $T$. 
	\end{corollary}
	
	\begin{proof}
		By \cref{lem:statbility_Id}, we have
		\begin{equation}\label{eq:Jj1}
			\| e^{-iz_n \vep^\alpha D_\kappa} (\mathcal{I}_\mathrm{d}^{m-2-j} \mathcal{J}_{j, 1})(z_n) \| = \| (\mathcal{I}_\mathrm{d}^{m-2-j} \mathcal{J}_{j, 1})(z_n) \| \lesssim \frac{(\|  R \|T)^{m-2-j}}{(m - 2 - j)!} \sup_{0 \leq z \leq z_n} \| \mathcal{J}_{j, 1}(z) \|. 
		\end{equation}
		For $ \mathcal{J}_{j, 1}(z) $, from \cref{eq:Jj1_def}, we have, for some $1 \leq \tilde n \leq n$ such that $z_{\tilde n} \leq z <z_{\tilde n +1}$, 
		\begin{align}
			\|  \mathcal{J}_{j, 1}(z) \| 
			&\lesssim \vep^\alpha \| R \| \sum_{k=0}^{\tilde n-1} \int_{z_k}^{z_{k+1}}   \int_{z_k}^{s} \| (\mathcal{I}^{j+1} D_\ka \mathfrak{u}) (s_1) \| \rmd s_1 \rmd s \ \lesssim \ \tau \vep^\alpha \| R \| T \sup_{0 \leq z \leq T} \| (\mathcal{I}^{j+1} D_\ka \mathfrak{u}) (z) \| \notag \\
			&\lesssim \tau \vep^\alpha \frac{\| R \|^{j} T^j}{j!} \sup_{0 \leq z \leq T} \| (\mathcal{I} D_\ka \mathfrak{u}) (z) \| \ \lesssim \ \tau \vep^\alpha \frac{\| R \|^j T^j}{j!} \vep^{1 - \frac{\alpha}{\ka}},  
		\end{align}
		where we apply \cref{prop:est1_continuous} with $\phi = D_\ka \psi(0)$ to obtain
		\begin{equation}
			\| (\mathcal{I} D_\ka \mathfrak{u}) (z) \| \lesssim \vep^{1-\frac{\alpha}{\ka}}, \quad 0 \leq z \leq T. 
		\end{equation}
		Hence, we obtain
		\begin{equation}
			\| e^{-iz_n \vep^\alpha D_\kappa} (\mathcal{I}_\mathrm{d}^{m-2-j} \mathcal{J}_{j, 1})(z_n) \| \lesssim \tau \vep^{1 + \frac{\ka-1}{\ka}\alpha} \frac{(\| R \|T)^{m-2}}{j! (m - 2 - j)!} \lesssim \tau \vep^{1 + \frac{\ka-1}{\ka}\alpha} \frac{(\| R \|T)^{m-2}}{(m-2)!} \binom{m-2}{j}, 
		\end{equation}
		which completes the proof by summing the RHS above in $j$. 
	\end{proof}
	
	Next we estimate the second part of $\mathcal{E}^n_{m, 2}$ in \cref{eq:E2_decomp} involving $\mathcal{J}_{j, 2}$ \cref{eq:Jj2_def}. 
    \begin{proposition}\label{cor:E2_est_2}
		Under assumptions \cref{eq:assumption,eq:assumption_2}, we have
		\begin{equation*}
			\sum_{j=0}^{m-2} \| e^{-iz_n \vep^\alpha D_\kappa} (\mathcal{I}_\mathrm{d}^{m-2-j} \mathcal{J}_{j, 2})(z_n)\| \lesssim \tau \vep^{2 - \frac{2}{\kappa} \alpha} \ln \vep^{-1} (\ln \vep^{-1} + \ln \tau^{-1})^{\mathfrak{\vartheta}(\ka)} \frac{C^{m-2}}{(m-2)!},
		\end{equation*}
		where $\vartheta(\ka) = 2 \times 1_{\ka = 2}$. 
	\end{proposition}

    \begin{proof}
        Similar to \cref{eq:Jj1}, we have
		\begin{equation}\label{eq:Jj2_est}
			\| e^{-iz_n \vep^\alpha D_\kappa} (\mathcal{I}_\mathrm{d}^{m-2-j} \mathcal{J}_{j, 2})(z_n) \| \lesssim \frac{(\|  R \|T)^{m-2-j}}{(m - 2 - j)!} \sup_{0 \leq z \leq z_n} \| \mathcal{J}_{j, 2}(z) \|, \quad 0 \leq j \leq m-2.  
		\end{equation}
        For $\mathcal{J}_{j, 2}(z)$, we consider the cases $j \geq 1$ and $ j = 0$ separately. When $j \geq 1$, from \cref{eq:Jj2_def}, using \cref{lem:statbility_Id}, we have
        \begin{align}
            \| \mathcal{J}_{j, 2}(z) \| 
            &\lesssim  \| R \| \sum_{k=0}^{n-1} \int_{z_k}^{z_{k+1}} \int_{z_k}^s \| (\mathcal{I}^jv)(s_1) \| \rmd s_1 \rmd s \lesssim \tau \| R \| \sup_{0 \leq z \leq T} \| (\mathcal{I}^jv)(z) \| \notag \\
            &\lesssim \tau \frac{(\| R \|T)^{j-1}}{(j-1)!}\sup_{0 \leq z \leq T}\| (\mathcal{I}v) \|, 
        \end{align}
        % When $j = 0$, we have, by \cref{eq:Jj2_def}, 
        % \begin{equation}
        %     \| \mathcal{J}_{0, 2}(z) \| = \| \sum_{k=0}^{n-1} \int_{z_k}^{z_{k+1}} e^{-i s \vep^\alpha D_\kappa} R_\vep \int_{z_k}^s v(s_1) \rmd s_1 \rmd s \|. 
        % \end{equation}
        which plugged into \cref{eq:Jj2_est} yields
        \begin{equation}
            \| e^{-iz_n \vep^\alpha D_\kappa} (\mathcal{I}_\mathrm{d}^{m-2-j} \mathcal{J}_{j, 2})(z_n) \| \lesssim \tau \binom{m-3}{j-1} \frac{(\|  R \|T)^{m-3}}{(m - 3)!} \sup_{0 \leq z \leq T}\| (\mathcal{I}v) \|, \quad 1 \leq j \leq m-2. 
        \end{equation}
        Hence, it suffices to establish estimates for $\| (\mathcal{I} v)(z) \|$ and $\| \mathcal{J}_{0, 2}(z) \|$, which are done in \cref{prop:est2_continuous,prop:est2_discrete} below.
    \end{proof}
       
    % With \cref{prop:est2_continuous,prop:est2_discrete}, we are able to obtain the estimate of the second part of $\mathcal{E}^n_{m, 2}$ similar to \cref{cor:E2_est_1}. 
    
    For $\| (\mathcal{I} v)(z) \|$, we have the following estimate. Recall that $v(z)=e^{iz\vep^\alpha D_\ka} R_\vep \psi(0)$ is defined in \cref{eq:uv_def}.
	\begin{lemma}\label{prop:est2_continuous}
		Under assumptions \cref{eq:assumption}, we have, for $0 \leq z \leq T$, 
		\begin{equation*}
			\| (\mathcal{I} v) (z) \| \lesssim \left\{
			\begin{aligned}
				&\vep^{2 - \frac{2}{\ka}\alpha} \ln \vep^{-1}, && \ka \geq 3\\
				&\vep^{2 - \frac{2}{\ka}\alpha} (\ln \vep^{-1})^2, && \ka = 2. 
			\end{aligned}
			\right. 
		\end{equation*}
	\end{lemma} 
	
	\begin{proof}
		As for \cref{prop:est1_continuous}, we first present the proof for even $\kappa$. Let 
		\begin{equation}
			I(z) : = e^{-iz\vep^\alpha D_\ka}(\mathcal{I} v)(z) = \int_{0}^z e^{-i s \vep^\alpha D_\ka} R_\vep e^{i s \vep^\alpha D_\kappa} R_\vep \psi(0) \rmd s.  
		\end{equation}
		By the Fourier transform and the change of variable, recalling $\widehat{R_\vep}(\xi) = \vep \widehat{R}(\vep \xi)$ and \cref{eq:phase}, we have
		\begin{align}\label{eq:three_term_Fourier}
			\widehat{I}(z, \xi) 
			&= \int_{\xi_1 + \xi_2 = \xi} \int_0^z e^{i s \vep^\alpha (P(\xi) - P(\xi_2))} \widehat{R_\vep}(\xi_1) \widehat{(R_\vep \psi)}(0, \xi_2) \rmd s \rmd \xi_1 \notag \\
			&= \int_{\xi_1 + \xi_2 + \xi_3 = \xi} \int_0^z e^{i s \vep^\alpha (P(\xi) - P(\xi_2 + \xi_3))} \widehat{R_\vep}(\xi_1) \widehat{R_\vep}(\xi_2) \widehat{\psi}(0, \xi_3) \rmd s \rmd \xi_1 \rmd \xi_2 \notag \\
			&= \int_{\frac{\xi_1}{\vep} + \frac{\xi_2}{\vep} + \xi_3 = \xi} \int_0^z e^{i s \vep^\alpha (P(\xi) - P(\frac{\xi_2}{\vep} + \xi_3))} \widehat{R}(\xi_1) \widehat{R}(\xi_2) \widehat{\psi}(0, \xi_3) \rmd s \rmd \xi_1 \rmd \xi_2 \notag \\
			&= \int_{\frac{\xi_1}{\vep} + \frac{\xi_2}{\vep} + \xi_3 = \xi} \int_0^z e^{i s \Phi(\xi_1, \frac{\xi_2}{\vep} + \xi_3)} \widehat{R}(\xi_1) \widehat{R}(\xi_2) \widehat{\psi}(0, \xi_3) \rmd s \rmd \xi_1 \rmd \xi_2. 
		\end{align}
    %     where the outer integral is defined as for some function $\phi = \phi(\xi, \xi_1, \xi_2, \xi_3)$ with $\xi$ fixed, 
    % \begin{equation*}
    %     \int_{\frac{\xi_1}{\vep} + \frac{\xi_2}{\vep} + \xi_3 = \xi} \phi(\xi, \xi_1, \xi_2, \xi_3) \rmd \xi_1 \rmd \xi_2 := \int_{\R^2} \phi(\xi, \xi_1, \xi_2, \xi - \frac{\xi_1}{\vep} - \frac{\xi_2}{\vep}) \rmd \xi_1 \rmd \xi_2. 
    % \end{equation*}
		We define the resonant and non-resonant sets as: for some $0< \delta \ll 1$ to be determined, 
		\begin{equation}
		    \begin{aligned}
                &\Lambda_1^\delta := \{(\xi_1, \xi_2, \xi_3) \in \R^3: |\xi_1| < \delta^2 \text{ or } |\xi_1 + 2 \xi_2 + 2 \vep \xi_3| < \delta^2 \text{ or } |\xi_1|, |\xi_1 + 2 \xi_2 + 2 \vep \xi_3|<\delta\}, \\
		          &\Lambda_2^\delta := \R^3 \setminus \Lambda_1^\delta. 
		\end{aligned}
		\end{equation}
		Then we have
		\begin{equation}
			I(z) = I_{\Lambda_1^\delta}(z) + I_{\Lambda_2^\delta}(z), 
		\end{equation} 
		where, similar to \cref{eq:bilinear_multiplier}, $I_\Lambda$ with $\Lambda \subset \R^3$ is defined as  
        %by applying a trilinear multiplier $1_{\Lambda}$ to $I$, i.e.
		\begin{equation}\label{eq:trilinear_multiplier}
			\widehat{I_\Lambda}(z, \xi) = \int_{\frac{\xi_1}{\vep} + \frac{\xi_2}{\vep} + \xi_3 = \xi} 1_{\Lambda} \int_0^z e^{i s \Phi(\xi_1, \frac{\xi_2}{\vep} + \xi_3)} \widehat{R}(\xi_1) \widehat{R}(\xi_2) \widehat{\psi}(0, \xi_3) \rmd s \rmd \xi_1 \rmd \xi_2, 
		\end{equation}
        where $1_{\Lambda} = 1_{\Lambda}(\xi_1, \xi_2, \xi_3)$ is the indicator function of $\Lambda$. 
		%		where
		%		\begin{equation}
			%			\widehat{I_j}(z, \xi) = \int_{\frac{\xi_1}{\vep} + \frac{\xi_2}{\vep} + \xi_3 = \xi} \int_0^z 1_{\Lambda_j^\delta}(\xi_1, \xi_2, \xi_3) e^{i s \vep^\alpha (|\xi|^2 - |\frac{\xi_2}{\vep} + \xi_3|^2)} \widehat{R}(\xi_1) \widehat{R}(\xi_2) \widehat{\psi}(0, \xi_3) \rmd s \rmd \xi_1 \rmd \xi_2, \quad j = 1, 2.  
			%		\end{equation}
		For $I_{\Lambda^\delta_1}$, noting that 
		\begin{equation}
			1_{\Lambda_1^\delta} \leq 1_{|\xi_1|<\delta^2} + 1_{|\xi_1 + 2 \xi_2 + 2 \vep \xi_3|<\delta^2} + 1_{|\xi_1|<\delta} 1_{|\xi_1 + 2 \xi_2 + 2 \vep \xi_3|<\delta},  
		\end{equation}
		we have from \cref{eq:trilinear_multiplier}
		\begin{align}\label{eq:33}
			\| \widehat{I_{\Lambda^\delta_1}}(z, \xi) \|_{L^1_\xi} 
			&\leq z \int_{\R^3} 1_{\Lambda_1^\delta} |\widehat{R}(\xi_1)| |\widehat{R}(\xi_2)| |\widehat{\psi}(0, \xi_3)| \rmd \xi_1 \rmd \xi_2 \rmd \xi_3 \notag \\
			&\leq z \int_{\R^2} |\widehat{R}(\xi_2)| |\widehat{\psi}(0, \xi_3)| \left(\int_{\R} (1_{|\xi_1|<\delta^2} + 1_{|\xi_1 + 2\xi_2 + 2 \vep \xi_3|<\delta^2}) |\widehat{R}(\xi_1)| \rmd \xi_1 \right) \rmd \xi_2 \rmd \xi_3 \notag \\
			&\quad + z \int_{\R}  |\widehat{\psi}(0, \xi_3)| \int_{\R} 1_{|\xi_1|<\delta} |\widehat{R}(\xi_1)| \left(\int_{\R} 1_{|\xi_1 + 2\xi_2 + 2 \vep \xi_3|<\delta}  |\widehat{R}(\xi_2)| \rmd \xi_2 \right)  \rmd \xi_1 \rmd \xi_3 \notag \\
			&\lesssim z \delta^2 \| \widehat{R} \|_{L^\infty} \| \widehat{R} \|_{L^1} \| \widehat{\psi}(0) \|_{L^1} + z \delta^2 \|\widehat{R} \|_{L^\infty}^2 \| \widehat{\psi}(0) \|_{L^1}. 
		\end{align}
		For $I_{\Lambda_2^\delta}$, when $(\xi_1, \xi_2, \xi_3) \in \Lambda_2^\delta$, we have $|\xi_1| \geq \delta$ or $|\xi_1+2\xi_2 + 2\vep \xi_3| \geq \delta$, and $\delta \geq C_0 \vep$ according to the choice of $\delta$ in \cref{eq:delta_3}. Then by \cref{eq:phase_est_3_1}, we have $\Phi(\xi_1, \frac{\xi_2}{\vep} + \xi_3) \neq 0$, and
		% \begin{equation}\label{eq:phase_low_bound}
		% 	|| \geq \frac{1}{\vep^{\ka - \alpha}}|\xi_1(\xi_1 + 2 \vep \xi_2)|(\xi_1^{\ka - 2} + (\xi_1 + 2 \vep \xi_2)^{\ka - 2}) > 0. 
		% \end{equation}
		\begin{align}
			\widehat{I_{\Lambda_2^\delta}}(z, \xi) 
			&= \int_{\frac{\xi_1}{\vep} + \frac{\xi_2}{\vep} + \xi_3 = \xi} \frac{1_{\Lambda_2^\delta}}{i\Phi(\xi_1, \frac{\xi_2}{\vep} + \xi_3)} \int_0^z \partial_s e^{i s \Phi(\xi_1, \frac{\xi_2}{\vep} + \xi_3)} \rmd s \widehat{R}(\xi_1) \widehat{R}(\xi_2) \widehat{\psi}(0, \xi_3) \rmd \xi_1 \rmd \xi_2 \notag \\
			&= \int_{\frac{\xi_1}{\vep} + \frac{\xi_2}{\vep} + \xi_3 = \xi} 1_{\Lambda_2^\delta} \frac{e^{i z \Phi(\xi_1, \frac{\xi_2}{\vep} + \xi_3)} - 1}{i\Phi(\xi_1, \frac{\xi_2}{\vep} + \xi_3)} \widehat{R}(\xi_1) \widehat{R}(\xi_2) \widehat{\psi}(0, \xi_3) \rmd \xi_1 \rmd \xi_2. 
		\end{align}
		Also, note that
		\begin{align}
			1_{\Lambda_2^\delta} 
			&\leq 1_{|\xi_1| \geq \delta^2} 1_{|\xi_1 + 2 \xi_2 + 2 \vep \xi_3| \geq \delta^2} (1_{|\xi_1| \geq \delta > |\xi_1 + 2 \xi_2 + 2 \vep \xi_3|} + 1_{|\xi_1 + 2 \xi_2 + 2 \vep \xi_3| \geq \delta}) \notag \\
			&= 1_{|\xi_1| \geq \delta^2} 1_{|\xi_1 + 2 \xi_2 + 2 \vep \xi_3| \geq \delta^2} (1_{|\xi_1| \geq \delta} 1_{|\xi_1 + 2 \xi_2 + 2 \vep \xi_3| < \delta} + 1_{|\xi_1 + 2 \xi_2 + 2 \vep \xi_3| \geq \delta}) \notag \\
			&= 1_{|\xi_1| \geq \delta} 1_{\delta^2 \leq |\xi_1 + 2 \xi_2 + 2 \vep \xi_3| < \delta} + 1_{|\xi_1| \geq \delta^2} 1_{|\xi_1 + 2 \xi_2 + 2 \vep \xi_3| \geq \delta}, 
		\end{align}
		We have, by \cref{eq:phase_est_3_1}, 
		\begin{align}
			\| \widehat{I_{\Lambda_2^\delta}}(z, \xi) \|_{L^1_\xi} 
			&\lesssim \vep^{\ka-\alpha} \int_{\R^3} 1_{\Lambda_2^\delta} \frac{|\widehat{R}(\xi_1)| |\widehat{R}(\xi_2)| |\widehat{\psi}(0, \xi_3)|}{|\xi_1||\xi_1 + 2 \xi_2 + 2 \vep \xi_3|(\xi_1^{\ka - 2}+ (\xi_1+2\xi_2+2\vep\xi_3)^{\ka-2})} \rmd \xi_1 \rmd \xi_2 \rmd \xi_3 \notag \\
			&= \vep^{\ka-\alpha}(A_1 + A_2), 
		\end{align}
		where
		\begin{align}
			&A_1:=\int_{\R^3} 1_{|\xi_1| \geq \delta} 1_{\delta^2 \leq |\xi_1 + 2 \xi_2 + 2 \vep \xi_3| < \delta} \frac{|\widehat{R}(\xi_1)| |\widehat{R}(\xi_2)| |\widehat{\psi}(0, \xi_3)|}{|\xi_1||\xi_1 + 2 \xi_2 + 2 \vep \xi_3|(\xi_1^{\ka - 2}+ (\xi_1+2\xi_2+2\vep\xi_3)^{\ka-2})} \rmd \xi_1 \rmd \xi_2 \rmd \xi_3, \\
			&A_2:=\int_{\R^3} 1_{|\xi_1| \geq \delta^2} 1_{|\xi_1 + 2 \xi_2 + 2 \vep \xi_3| \geq \delta} \frac{|\widehat{R}(\xi_1)| |\widehat{R}(\xi_2)| |\widehat{\psi}(0, \xi_3)|}{|\xi_1||\xi_1 + 2 \xi_2 + 2 \vep \xi_3|(\xi_1^{\ka - 2}+ (\xi_1+2\xi_2+2\vep\xi_3)^{\ka-2})} \rmd \xi_1 \rmd \xi_2 \rmd \xi_3. 
		\end{align}
		For $A_1$, using that $(\xi_1^{\ka - 2}+ (\xi_1+2\xi_2+2\vep\xi_3)^{\ka-2}) \geq \xi_1^{\kappa-2}$, we have
		\begin{align}
			A_1 
			&\leq \int_{\R^3} 1_{|\xi_1| \geq \delta} 1_{\delta^2 \leq |\xi_1 + 2 \xi_2 + 2 \vep \xi_3| < \delta} \frac{|\widehat{R}(\xi_1)| |\widehat{R}(\xi_2)| |\widehat{\psi}(0, \xi_3)|}{|\xi_1|^{\kappa - 1}|\xi_1 + 2 \xi_2 + 2 \vep \xi_3|} \rmd \xi_1 \rmd \xi_2 \rmd \xi_3 \notag \\
			&= \int_{\R^2} 1_{|\xi_1| \geq \delta}  \frac{|\widehat{R}(\xi_1)|}{|\xi_1|^{\ka-1}} |\widehat{\psi}(0, \xi_3)| \int_{\R} 1_{\delta^2 \leq |\xi_1 + 2 \xi_2 + 2 \vep \xi_3| < \delta} \frac{ |\widehat{R}(\xi_2)|}{|\xi_1 + 2 \xi_2 + 2 \vep \xi_3|} \rmd \xi_2 \rmd \xi_1  \rmd \xi_3 \notag \\
			&\lesssim \| \widehat{R} \|_{L^\infty} |\ln \delta| (\frac{\| \widehat{R} \|_{L^\infty}}{\delta^{\ka - 2}} + \| \widehat{R} \|_{L^1}) \| \widehat{\psi}(0) \|_{L^1} \lesssim \frac{|\ln \delta|}{\delta^{\ka-2}},  
		\end{align}
		where $\frac{1}{\delta^0}$ is understood as $\ln \delta^{-1}$. 
		For $A_2$, using $(\xi_1^{\ka - 2}+ (\xi_1+2\xi_2+2\vep\xi_3)^{\ka-2}) \geq (\xi_1+2\xi_2+2\vep\xi_3)^{\kappa-2}$, we have 
		\begin{align}
			A_2
			&\leq \int_{\R^3} 1_{|\xi_1| \geq \delta^2} 1_{|\xi_1 + 2 \xi_2 + 2 \vep \xi_3| \geq \delta} \frac{|\widehat{R}(\xi_1)| |\widehat{R}(\xi_2)| |\widehat{\psi}(0, \xi_3)|}{|\xi_1| |\xi_1 + 2 \xi_2 + 2 \vep \xi_3|^{\ka-1}} \rmd \xi_1 \rmd \xi_2 \rmd \xi_3 \notag \\
			&= \int_{\R^2} 1_{|\xi_1| \geq \delta^2}  \frac{|\widehat{R}(\xi_1)|}{|\xi_1|} |\widehat{\psi}(0, \xi_3)| \int_{\R} 1_{|\xi_1 + 2 \xi_2 + 2 \vep \xi_3| \geq \delta} \frac{ |\widehat{R}(\xi_2)|}{|\xi_1 + 2 \xi_2 + 2 \vep \xi_3|^{\ka-1}} \rmd \xi_2 \rmd \xi_1  \rmd \xi_3 \notag \\
			&\lesssim (\frac{\| \widehat{R} \|_{L^\infty}}{\delta^{\ka-2}} + \| \widehat{R} \|_{L^1}) (|\ln \delta| \| \widehat{R} \|_{L^\infty} + \| \widehat{R}\|_{L^1}) \| \widehat{\psi}(0) \|_{L^1} \lesssim \frac{|\ln \delta|}{\delta^{\ka-2}}. 
		\end{align}
		Combining the above, choosing $\delta = \vep^{1-\frac{\alpha}{\ka}}$ when $\alpha > 0$ and $\delta = C_0 \vep$ when $\alpha = 0$ with $C_0$ given by \cref{cor:phase}, we obtain
		\begin{equation}\label{eq:delta_3}
			\| \widehat{I}(z, \xi) \|_{L^1_\xi} \lesssim \delta^2 + \vep^{\ka-\alpha} \frac{|\ln \delta|}{\delta^{\ka-2}} \sim \left\{
			\begin{aligned}
				&\vep^{2 - \frac{2}{\ka}\alpha} \ln \vep^{-1}, && \ka \neq 2,\\
				&\vep^{2 - \frac{2}{\ka}\alpha} (\ln \vep^{-1})^2, && \ka = 2. 
			\end{aligned}
			\right. 
		\end{equation}
		
		For odd $\ka$, we consider the resonant and non-resonant decomposition
		\begin{equation}
			\Lambda_1^\delta := \{(\xi_1, \xi_2, \xi_3) \in \R^3: |\xi_1| < \delta^2 \text{ or } |\xi_1|, |\xi_1 + 2 \xi_2 + 2 \vep \xi_3|<\delta\}, \quad \Lambda_2^\delta := \R^3 \setminus \Lambda_1^\delta. 
		\end{equation}
		In this case, we have
		\begin{equation}
			1_{\Lambda_1^\delta} \leq 1_{|\xi_1|<\delta^2} + 1_{|\xi_1|<\delta} 1_{|\xi_1 + 2 \xi_2 + 2 \vep \xi_3|<\delta},  
		\end{equation}
		and
		\begin{align}
			1_{\Lambda_2^\delta} 
			&\leq 1_{|\xi_1| \geq \delta^2} (1_{|\xi_1| \geq \delta > |\xi_1 + 2 \xi_2 + 2 \vep \xi_3|} + 1_{|\xi_1 + 2 \xi_2 + 2 \vep \xi_3| \geq \delta}) \notag \\
			&= 1_{|\xi_1| \geq \delta} 1_{|\xi_1 + 2 \xi_2 + 2 \vep \xi_3| < \delta} + 1_{|\xi_1| \geq \delta^2} 1_{|\xi_1 + 2 \xi_2 + 2 \vep \xi_3| \geq \delta}. 
		\end{align}
		  Then the estimate for $I_{\Lambda_1}$ follows similarly as 
		\begin{align}
			\| \widehat{I_{\Lambda_1^\delta}}(z, \xi) \|_{L^1_\xi} 
			&\leq z \int_{\R^3} 1_{\Lambda_1^\delta} |\widehat{R}(\xi_1)| |\widehat{R}(\xi_2)| |\widehat{\psi}(0, \xi_3)| \rmd \xi_1 \rmd \xi_2 \rmd \xi_3 \notag \\
			&\leq z \int_{\R^2} |\widehat{R}(\xi_2)| |\widehat{\psi}(0, \xi_3)| \int_{\R} 1_{|\xi_1|<\delta^2} |\widehat{R}(\xi_1)| \rmd \xi_1 \rmd \xi_2 \rmd \xi_3 \notag \\
			&\quad + z \int_{\R}  |\widehat{\psi}(0, \xi_3)| \int_{\R} 1_{|\xi_1|<\delta} |\widehat{R}(\xi_1)| \int_{\R} 1_{|\xi_1 + 2\xi_2 + 2 \vep \xi_3|<\delta}  |\widehat{R}(\xi_2)| \rmd \xi_2  \rmd \xi_1 \rmd \xi_3 \notag \\
			&\lesssim z \delta^2 \| \widehat{R} \|_{L^\infty} \| \widehat{R} \|_{L^1} \| \widehat{\psi}(0) \|_{L^1} + z \delta^2 \|\widehat{R} \|_{L^\infty}^2 \| \widehat{\psi}(0) \|_{L^1} \lesssim z \delta^2. 
		\end{align}
        Next, when $(\xi_1, \xi_2, \xi_3) \in \Lambda_2^\delta$, we still have \cref{eq:phase_est_3_1}, and $I_{\Lambda_2}$ is similarly estimated as
		\begin{equation}
			\| \widehat{I_{\Lambda_2^\delta}}(z, \xi) \|_{L^1_\xi} \lesssim \vep^{\ka- \alpha}(B_1 + B_2), 
		\end{equation}
		where
		\begin{align}
			&B_1:=\int_{\R^3} 1_{|\xi_1| \geq \delta} 1_{|\xi_1 + 2 \xi_2 + 2 \vep \xi_3| < \delta} \frac{|\widehat{R}(\xi_1)| |\widehat{R}(\xi_2)| |\widehat{\psi}(0, \xi_3)|}{|\xi_1|(\xi_1^{\ka - 1}+ (\xi_1+2\xi_2+2\vep\xi_3)^{\ka-1})} \rmd \xi_1 \rmd \xi_2 \rmd \xi_3, \\
			&B_2:=\int_{\R^3} 1_{|\xi_1| \geq \delta^2} 1_{|\xi_1 + 2 \xi_2 + 2 \vep \xi_3| \geq \delta} \frac{|\widehat{R}(\xi_1)| |\widehat{R}(\xi_2)| |\widehat{\psi}(0, \xi_3)|}{|\xi_1|(\xi_1^{\ka - 1}+ (\xi_1+2\xi_2+2\vep\xi_3)^{\ka-1})} \rmd \xi_1 \rmd \xi_2 \rmd \xi_3. 
		\end{align}
		Due to different singularity in the denominator, the estimate for $B_1$ is slightly different from $A_1$ above, where we need to use $1_{|\xi_1 + 2 \xi_2 + 2 \vep \xi_3| < \delta}$. Using that $\xi_1^{\ka - 1}+ (\xi_1+2\xi_2+2\vep\xi_3)^{\ka-1} \geq \xi_1^{\kappa-1}$, we have
		\begin{align}
			B_1 
			&\lesssim \int_{\R^3} 1_{|\xi_1| \geq \delta} 1_{|\xi_1 + 2 \xi_2 + 2 \vep \xi_3| < \delta} \frac{|\widehat{R}(\xi_1)| |\widehat{R}(\xi_2)| |\widehat{\psi}(0, \xi_3)|}{|\xi_1|^{\kappa}} \rmd \xi_1 \rmd \xi_2 \rmd \xi_3 \notag \\
			&= \int_{\R^2} 1_{|\xi_1| \geq \delta}  \frac{|\widehat{R}(\xi_1)|}{|\xi_1|^{\ka}} |\widehat{\psi}(0, \xi_3)| \int_{\R} 1_{|\xi_1 + 2 \xi_2 + 2 \vep \xi_3| < \delta} |\widehat{R}(\xi_2)| \rmd \xi_2 \rmd \xi_1  \rmd \xi_3 \notag \\
			&\lesssim \delta \| \widehat{R} \|_{L^\infty} (\frac{\| \widehat{R} \|_{L^\infty}}{\delta^{\ka - 1}} + \| \widehat{R} \|_{L^1}) \| \widehat{\psi}(0) \|_{L^1} \lesssim \frac{1}{\delta^{\ka-2}}. 
		\end{align}
		Using that $\xi_1^{\ka - 1}+ (\xi_1+2\xi_2+2\vep\xi_3)^{\ka-1} \geq (\xi_1+2\xi_2+2\vep\xi_3)^{\kappa-1}$, we have 
		\begin{align}
			B_2
			&\lesssim \int_{\R^3} 1_{|\xi_1| \geq \delta^2} 1_{|\xi_1 + 2 \xi_2 + 2 \vep \xi_3| \geq \delta} \frac{|\widehat{R}(\xi_1)| |\widehat{R}(\xi_2)| |\widehat{\psi}(0, \xi_3)|}{|\xi_1| |\xi_1 + 2 \xi_2 + 2 \vep \xi_3|^{\ka-1}} \rmd \xi_1 \rmd \xi_2 \rmd \xi_3 \notag \\
			&= \int_{\R^2} 1_{|\xi_1| \geq \delta^2}  \frac{|\widehat{R}(\xi_1)|}{|\xi_1|} |\widehat{\psi}(0, \xi_3)| \int_{\R} 1_{|\xi_1 + 2 \xi_2 + 2 \vep \xi_3| \geq \delta} \frac{ |\widehat{R}(\xi_2)|}{|\xi_1 + 2 \xi_2 + 2 \vep \xi_3|^{\ka-1}} \rmd \xi_2 \rmd \xi_1  \rmd \xi_3 \notag \\
			&\lesssim (\frac{\| \widehat{R} \|_{L^\infty}}{\delta^{\ka-2}} + \| \widehat{R} \|_{L^1}) (|\ln \delta| \| \widehat{R} \|_{L^\infty} + \| \widehat{R}\|_{L^1}) \| \widehat{\psi}(0) \|_{L^1} \lesssim \frac{|\ln \delta|}{\delta^{\ka-2}}. 
		\end{align}
		Combining the above, we can conclude the proof for odd $\kappa$.
	\end{proof}
	
	For $\mathcal{J}_{0, 2}$ \cref{eq:Jj2_def}, we have the following result, which can be understood as the discrete version of \cref{prop:est2_continuous}. Similar to \cref{prop:est1_discrete}, for the time discrete group, we have a much larger resonant set which needs to be handled carefully. 
	\begin{lemma}\label{prop:est2_discrete}
		Under assumptions \cref{eq:assumption,eq:assumption_2}, for any $0 \leq z \leq T$, we have
		\begin{equation*}
			\| \mathcal{J}_{0, 2}(z) \| \lesssim \tau \left\{
			\begin{aligned}
				&\vep^{2 - \frac{2}{\ka}\alpha} \ln \vep^{-1}, && \ka \geq 3\\
				&\vep^{2 - \alpha} \ln \vep^{-1}(\ln \vep^{-1} + \ln \tau^{-1})^2, && \ka = 2. 
			\end{aligned}
			\right. 
		\end{equation*}
	\end{lemma}
	
	\begin{proof}
		We assume that $z_n \leq z < z_{n+1}$ for some $n \geq 0$. Similar to \cref{prop:est2_continuous}, we present the proof for even $\ka$ first. Let $I:=e^{-iz_n\vep^\alpha D_\ka} \mathcal{J}_{0, 2}(z_n)$. It suffices to estimate $I$. Recalling \cref{eq:Jj2_def}, by a change of variable, we have
		\begin{equation}
			I = \sum_{k=0}^{n-1} \int_{0}^{\tau} e^{-i (z_k + s) \vep^\alpha D_\ka} R_\vep \int_{0}^s e^{i(z_k + s_1) \vep^\alpha D_\ka} R_\vep \psi(0) \rmd s_1 \rmd s. 
		\end{equation}
		Taking the Fourier transform, similar to \cref{eq:three_term_Fourier}, we have
		\begin{equation}\label{eq:I_est2_disc}
			\widehat{I}(\xi) = \int_{\frac{\xi_1}{\vep} + \frac{\xi_2}{\vep} + \xi_3 = \xi} \sum_{k=0}^{n-1} e^{i z_k \Phi(\xi_1, \frac{\xi_2}{\vep} + \xi_3)} \int_{0}^{\tau} \int_0^s e^{i s \vep^\alpha P(\xi)} e^{-is_1\vep^\alpha P(\frac{\xi_2}{\vep} + \xi_3)} \rmd s_1 \rmd s  \widehat{R}(\xi_1) \widehat{R}(\xi_2) \widehat{\psi}(0, \xi_3) \rmd \xi_1 \rmd \xi_2.  
		\end{equation}
		Define 
		\begin{equation}
			\theta := \tau \Phi(\xi_1, \frac{\xi_2}{\vep} + \xi_3), 
		\end{equation}
		with which we have $e^{i z_k \Phi(\xi_1, \frac{\xi_2}{\vep} + \xi_3)} = e^{i k\theta}$ in \cref{eq:I_est2_disc}. Let $C_0' \geq C_0$ sufficiently large depending on $\{d_{\ka - 2j}\}$ and $\ka$, where $C_0$ is the constant in \cref{cor:phase}. Define $\Lambda_{C_0'}$ as  
		\begin{equation*}
			\Lambda_{C_0'} := \{(\xi_1, \xi_2, \xi_3) \in \R^3: |\xi_1| \leq C_0' \vep, \ |\xi_1 + 2\xi_2 + 2 \vep \xi_3| \leq C_0' \vep \}. 
		\end{equation*}
		Then consider the following decomposition of $\R^3$:
		\begin{equation}
			\R^3 = \Lambda_0 \cup \Lambda_1 \cup \Lambda_0' \cup \Lambda_2^\delta \cup \Lambda_3^\delta, 
		\end{equation}
		where, with $L_0 \gg 1$ and $0<\delta\ll1$ to be determined, 
		\begin{equation}
			\begin{aligned}
				&\Lambda_0 := \{(\xi_1, \xi_2, \xi_3) \in \R^3 :(\xi_1, \xi_2, \xi_3) \not \in (-L_0, L_0)^3\}, \\
				&\Lambda_1 := \{(\xi_1, \xi_2, \xi_3) \in (-L_0, L_0)^3 : |\theta| \leq \pi\}, \\
				&\Lambda_0' := \{(\xi_1, \xi_2, \xi_3) \in (-L_0, L_0)^3 : |\theta| > \pi\} \cap \Lambda_{C_0'} \notag \\				
				&\Lambda_2^\delta := \{(\xi_1, \xi_2, \xi_3) \in (-L_0, L_0)^3 : |\theta| > \pi, \  \text{dist}(\theta, 2\pi \mathbb{Z}) \leq \delta \} \cap \Lambda_{C_0'}^c,  \\
				&\Lambda_3^\delta := \{(\xi_1, \xi_2, \xi_3) \in (-L_0, L_0)^3 : |\theta| > \pi, \ \text{dist}(\theta, 2\pi \mathbb{Z}) > \delta \} \cap \Lambda_{C_0'}^c. 
			\end{aligned}
		\end{equation}
		We have $1 = 1_{\Lambda_0} + 1_{\Lambda_1} + 1_{\Lambda_0'} + 1_{\Lambda_2^\delta} + 1_{\Lambda_3^\delta}$ and, recalling \cref{eq:trilinear_multiplier}, 
		\begin{equation}
			I = I_{\Lambda_0} + I_{\Lambda_1} + I_{\Lambda_0'} + I_{\Lambda_2^\delta} + I_{\Lambda_3^\delta}.  
		\end{equation}
        The above decomposition should be compared with \cref{eq:decomp_1}, where the five sets play exactly the same roles. There are, however, a few subtle differences. First, in the ``large phase" regime given by $\Lambda_2^\delta$ and $\Lambda_3^\delta$, we make an essential use of the phase structure \cref{lem:phase}. As a result, the singularity in the integrals obtained from the co-area formula differs significantly for even $\ka$ and odd $\ka$, and we treat the two cases separately. Also, to identify a leading order term in the phase, we separate the region $\Lambda_0'$ based on \cref{cor:phase}, which now contributes a smaller error of order $O(\tau\vep^2)$; this is needed as there is a higher order phase cancellation in the ``small phase" regime $\Lambda_1$. 
        
		For $I_{\Lambda_0}$, we have, by \cref{eq:assumption_2}, 
		\begin{align}
			\|\widehat{I_{\Lambda_0}}\|_{L^1} 
			&\lesssim \tau \int_{\R^3} 1_{\Lambda_0} |\widehat{R}(\xi_1)| |\widehat{R}(\xi_2)| |\widehat{\psi}(0, \xi_3)| \rmd \xi_1 \rmd \xi_2 \rmd \xi_3 \notag \\
			&\lesssim \tau \int_{\R^3} (1_{|\xi_1| \geq L_0} + 1_{|\xi_2| \geq L_0} + 1_{|\xi_3| \geq L_0}) |\widehat{R}(\xi_1)| |\widehat{R}(\xi_2)| |\widehat{\psi}(0, \xi_3)| \rmd \xi_1 \rmd \xi_2 \rmd \xi_3 \lesssim \tau \frac{1}{L_0^\sigma}.  
		\end{align}
		
		Then we consider $I_{\Lambda_1}$, the estimate of which follows the same line as \cref{prop:est2_continuous} by using the trivial bound $|\int_{0}^{\tau} \int_0^s e^{i s \vep^\alpha P(\xi)} e^{-is_1\vep^\alpha P(\frac{\xi_2}{\vep} + \xi_3)} \rmd s_1 \rmd s| \lesssim \tau^2$ and noting that when $(\xi_1, \xi_2, \xi_3) \in \Lambda_1$, we have $|\theta| < \pi$ and thus
		\begin{equation}
			| \sum_{k=0}^{n-1}  e^{ i k \theta} | = |\frac{1 - e^{i n \theta}}{1 - e^{i \theta}}|  \lesssim \frac{1}{|\theta|}. 
		\end{equation}
		Hence, we can obtain
		\begin{equation*}
			\| \widehat{I_{\Lambda_1}} \|_{L^1} \lesssim \tau \left\{
			\begin{aligned}
				&\vep^{2 - \frac{2}{\ka}\alpha} \ln \vep^{-1}, && \ka \geq 3\\
				&\vep^{2 - \frac{2}{\ka}\alpha} (\ln \vep^{-1})^2, && \ka = 2. 
			\end{aligned}
			\right. 
		\end{equation*}
		For $I_{\Lambda_{0}'}$, we have
		\begin{align}
			\| \widehat{I_{\Lambda'_0}} \|_{L^1_\xi} 
			&\leq n \tau^2 \int_{\R^3} 1_{\Lambda_0'} |\widehat{R}(\xi_1)| |\widehat{R}(\xi_2)| |\widehat{\psi}(0, \xi_3)| \rmd \xi_1 \rmd \xi_2 \rmd \xi_3 \notag \\
			&\lesssim \tau \int_{\R}  |\widehat{\psi}(0, \xi_3)| \int_{\R} 1_{|\xi_1|<C\vep} |\widehat{R}(\xi_1)| \left(\int_{\R} 1_{|\xi_1 + 2\xi_2 + 2 \vep \xi_3|<C\vep} |\widehat{R}(\xi_2)| \rmd \xi_2 \right)  \rmd \xi_1 \rmd \xi_3 \notag \\
			&\lesssim \tau \vep^2 \| \widehat{R} \|^2_{L^\infty} \| \widehat{\psi}(0) \|_{L^1}. 
		\end{align}
		
		In the following, we estimate $I_{\Lambda_2^\delta}$ and $I_{\Lambda_3^\delta}$. For $\Lambda_{2}^\delta$, we can further decompose it into
		\begin{equation}
			\Lambda_{2}^\delta = \cup_{|l| \geq 1} \Lambda_{2, l}^\delta := \cup_{|l| \geq 1} \{(\xi_1, \xi_2, \xi_3) \in \Lambda_2^\delta: \theta \in (2 \pi l - \delta, 2\pi l + \delta) \} \cap \Lambda_{C_0'}^c.  
		\end{equation}
		Again, we note that when $|l| \gtrsim a L_0^\ka$, $\Lambda_2^\delta = \emptyset$. Then we have
		\begin{equation}
			\| \widehat{I}_{\Lambda_2^\delta} \|_{L^1} \leq \sum_{|l| \geq 1} 1_{|l|\lesssim aL_0^\ka} \| \widehat{I}_{\Lambda_{2, l}^\delta} \|_{L^1}. 
		\end{equation}
		For each ${I}_{\Lambda_{2, l}^\delta}$, we have, 
		\begin{equation}\label{eq:I2_l}
			\| \widehat{I}_{\Lambda_{2, l}^\delta} \|_{L^1} \lesssim \tau \int_{\R^3} 1_{\Lambda_{2, l}^\delta} |\widehat{R}(\xi_1)| |\widehat{R}(\xi_2)| |\widehat{\psi}(0, \xi_3)| \rmd \xi_1 \rmd \xi_2 \rmd \xi_3. 
		\end{equation}
		Introduce the change of variables
		\begin{equation}\label{eq:change_of_variable_3}
			\xi_1 \rightarrow \eta_1, \quad \xi_1 + 2 \xi_2 + 2 \vep \xi_3 \rightarrow \eta_2, \quad \xi_3 \rightarrow \eta_3, 
		\end{equation}
		and define (by \cref{lem:phase}, $\Phi(\xi_1, \frac{\xi_2}{\vep} + \xi_3)$ is also a polynomial in $\xi_1 = \eta_1$ and $\xi_1 + 2\xi_2 + 2\vep\xi_3 = \eta_2$)
		\begin{equation}\label{eq:theta_tilde_def}
			\tilde \theta = \tilde \theta(\eta_1, \eta_2) := \theta = \tau \Phi(\xi_1, \frac{\xi_2}{\vep} + \xi_3). 
		\end{equation}
        We further define
        \begin{equation}\label{eq:f_def}
            a := \frac{\tau}{\vep^{\ka - \alpha}}, \quad f = f(\eta_1, \eta_2):=\vep^{\ka - \alpha} \Phi(\xi_1, \frac{\xi_2}{\vep} + \xi_3) \quad \text{such that} \quad \tilde \theta(\eta_1, \eta_2) = a f(\eta_1, \eta_2).   
        \end{equation}
  %       In fact, using \cref{lem:phase}, $f$ can be explicitly given by
  %       \begin{equation}\label{eq:f_def_2}
		% 	f(\eta_1, \eta_2)= \eta_1 \eta_2 \sum_{0 \leq 2j < \ka} \vep^{2j} \tilde d_{\ka - 2j} Q_{\ka-2j}(\eta_1^2, \eta_2^2). 
		% \end{equation}        
        On $\Lambda_2^\delta$, due to the exclusion of $\Lambda_{C_0'}$, \cref{eq:phase_est_3_1} holds and thus, with $f_{\eta_\ell} := \partial_{\eta_\ell}f$ for $\ell=1,2$, 
        \begin{equation}\label{eq:f_low_bound}
			|f| \gtrsim |\eta_1\eta_2|(\eta_1^{\ka - 2} + \eta_2^{\ka - 2}), \quad  |f_{\eta_1}| \gtrsim |\eta_2|(\eta_1^{\ka - 2} + \eta_2^{\ka - 2}), \quad |f_{\eta_2}| \gtrsim |\eta_1|(\eta_1^{\ka - 2} + \eta_2^{\ka - 2}), 
		\end{equation}
        where the derivative estimates, though not following from \cref{eq:phase_est_3_1} directly, can be obtained using the same arguments. 
		% \begin{equation}\label{eq:theta_tilde}
		% 	\tilde \theta(\eta_1, \eta_2) = \frac{\tau}{\vep^{\ka - \alpha}} \eta_1 \eta_2 \sum_{0 \leq 2j < \ka} \vep^{2j} \tilde d_{\ka - 2j} Q_{\ka-2j}(\eta_1^2, \eta_2^2) = a f,  
		% \end{equation}
		% where
		Then from \cref{eq:I2_l}, using the new variables \cref{eq:change_of_variable_3}, we have
		\begin{align}\label{eq:Ilambda_2}
			\| \widehat{I}_{\Lambda_{2, l}^\delta} \|_{L^1} 
			&\lesssim \tau \int_{\R^3} 1_{\Lambda_{2, l}^\delta} |\widehat{R}(\eta_1)| |\widehat{R}(\frac{\eta_2 - \eta_1 - 2 \vep \eta_3}{2})| |\widehat{\psi}(0, \eta_3)| \rmd \eta_1 \rmd \eta_2 \rmd \eta_3 \notag \\
			&\lesssim \tau \|\widehat{R}\|_{L^\infty} \int_{\R^3} (1_{|\eta_1| \geq C_0' \vep} + 1_{|\eta_2| \geq C_0' \vep}) 1_{|\eta_1| \leq L_0} 1_{|\eta_2| \lesssim L_0} 1_{\tilde \theta \in (2\pi l - \delta, 2 \pi l + \delta)} |\widehat{R}(\eta_1)| |\widehat{\psi}(0, \eta_3)| \rmd \eta_1 \rmd \eta_2 \rmd \eta_3 \notag \\
			&\leq \tau \|\widehat{R}\|_{L^\infty}^2 \| \widehat{\psi}(0) \|_{L^1} \int_{\R^2} (1_{|\eta_1| \geq C_0' \vep} + 1_{|\eta_2| \geq C_0' \vep}) 1_{|\eta_1| \leq L_0}  1_{|\eta_2| \lesssim L_0} 1_{\tilde \theta \in (2\pi l - \delta, 2 \pi l + \delta)} \rmd \eta_2 \rmd \eta_1. 
		\end{align}
		We again use the co-area formula to estimate the above integral. It suffices to consider $l \geq 1$ and $\eta_1, \eta_2 > 0$: 
		\begin{align}\label{eq:J_def_2}
			J 
			&:= \int_{\R^2_+} (1_{\eta_1 \geq C_0' \vep} + 1_{\eta_2 \geq C_0' \vep}) 1_{\eta_1 \leq L_0} 1_{\eta_2 \lesssim L_0} 1_{\tilde \theta(\eta_1, \eta_2) \in (2\pi l - \delta, 2 \pi l + \delta)} \rmd \eta_2 \rmd \eta_1 \notag \\
			&= \int_{2 \pi l - \delta}^{2 \pi l + \delta} \rmd u \int_{\tilde \theta = u} (1_{\eta_1 \geq C_0' \vep} + 1_{\eta_2 \geq C_0' \vep}) 1_{\eta_1 \leq L_0} 1_{\eta_2 \lesssim L_0} \frac{1}{|\nabla \tilde \theta|} \rmd \mathcal{H}^1(\eta_1, \eta_2),  
		\end{align}
		where we note that 
        % from \cref{eq:theta_tilde}, 
		% \begin{align}\label{eq:grad_theta_2}
		% 	\nabla \tilde \theta = a
		% 	\begin{pmatrix}
		% 		f_{\eta_1} \\
		% 		f_{\eta_2}
		% 	\end{pmatrix} 
		% 	&= a \sum_{0 \leq 2j < \ka} \vep^{2j} \tilde d_{\ka - 2j}
		% 	\begin{pmatrix}
		% 		\eta_2 Q_{\ka-2j}(\eta_1^2, \eta_2^2) + 2 \eta_1^2 \eta_2 (\partial_x Q_{\ka-2j})(\eta_1^2, \eta_2^2) \\
		% 		\eta_1 Q_{\ka-2j}(\eta_1^2, \eta_2^2) + 2 \eta_1 \eta_2^2 (\partial_y Q_{\ka-2j})(\eta_1^2, \eta_2^2)
		% 	\end{pmatrix}  \notag \\
		% 	&= a \sum_{0 \leq 2j < \ka} \vep^{2j} \tilde d_{\ka - 2j}
		% 	\begin{pmatrix}
		% 		\eta_2 (Q_{\ka-2j} + 2 x \partial_x Q_{\ka-2j})(\eta_1^2, \eta_2^2) \\
		% 		\eta_1 (Q_{\ka-2j} + 2 y \partial_y Q_{\ka-2j})(\eta_1^2, \eta_2^2)
		% 	\end{pmatrix} \notag \\
		% 	&= a \sum_{0 \leq 2j < \ka} \vep^{2j} \tilde d_{\ka - 2j} (\ka-2j)
		% 	\begin{pmatrix}
		% 		\eta_2 Q_{\ka-2j-1}(\eta_2^2, \eta_1^2) \\
		% 		\eta_1 Q_{\ka-2j-1}(\eta_1^2, \eta_2^2)
		% 	\end{pmatrix}, 
		% \end{align}
		% where we use that $(Q_{\ka-2j} + 2 x \partial_x Q_{\ka-2j}) (x, y) = Q_{\ka-2j-1}(y, x)$ and $(Q_{\ka-2j} + 2 y \partial_y Q_{\ka-2j}) (x, y) = Q_{\ka-2j-1}(x, y)$, but these particular relations are not essential in the proof. 
        % As we also have $\eta_1 \geq C_0 \vep$ or $\eta_2 \geq C_0 \vep$ on $\Lambda_{2}^\delta$ and $\Lambda_3^\delta$, by \cref{cor:phase}, we have, from \cref{eq:f_def_2,eq:grad_theta_2}
		since $\tilde \theta = a f \neq 0$, we have $f_{\eta_1} \neq 0$, $f_{\eta_2} \neq 0$, and $|\nabla \tilde \theta| \neq 0$ from \cref{eq:f_low_bound}. Then we perform the following change of measure
		\begin{equation}\label{eq:dH_2}
			\rmd \mathcal{H}^1(\eta_1, \eta_2) = \sqrt{1 + \left(\frac{\rmd \eta_2}{\rmd \eta_1}\right)^2} \rmd \eta_1 = \sqrt{1 + \left( \frac{f_{\eta_1}}{f_{\eta_2}} \right)^2} \rmd \eta_1. 
		\end{equation}
		Plugging \cref{eq:dH_2} and $|\nabla \tilde \theta| = a\sqrt{f_{\eta_1}^2 + f_{\eta_2}^2}$ into \cref{eq:J_def_2}, using \cref{eq:f_low_bound}, we obtain
		\begin{align}
			J
			&\lesssim \int_{2 \pi l - \delta}^{2 \pi l + \delta} \rmd u \int_{\tilde \theta = u} (1_{\eta_1 \geq C_0' \vep} + 1_{\eta_2 \geq C_0' \vep}) 1_{0 \leq \eta_1 \leq L_0} 1_{0 \leq \eta_2 \lesssim L_0} \frac{1}{a|f_{\eta_2}|} \rmd \eta_1  \notag \\
			&\lesssim \int_{2 \pi l - \delta}^{2 \pi l + \delta} \rmd u \int_{\tilde \theta = u}(1_{\eta_1 \geq C_0' \vep} + 1_{\eta_2 \geq C_0' \vep}) \frac{1_{0 \leq \eta_1 \leq L_0} 1_{0 \leq \eta_2 \lesssim L_0}}{a\eta_1(\eta_1^{\ka - 2} + \eta_2^{\ka-2})} \rmd \eta_1 \ =J_< + J_>, 
		\end{align}
		where
		\begin{equation}\label{eq:J_sq_def}
			J_\square := \int_{2 \pi l - \delta}^{2 \pi l + \delta} \rmd u \int_{\tilde \theta = u} 1_{\eta_1 \square \eta_2} (1_{\eta_1 \geq C_0' \vep} + 1_{\eta_2 \geq C_0' \vep}) \frac{1_{0 \leq \eta_1 \leq L_0} 1_{0 \leq \eta_2 \lesssim L_0}}{a\eta_1(\eta_1^{\ka - 2} + \eta_2^{\ka-2})} \rmd \eta_1, \quad \square \in \{<, >\}.  
		\end{equation}
		Note that on the curve $\tilde \theta = u$, there exists a unique point satisfying $\eta_1 = \eta_2 \geq C_0' \vep$ given by
		\begin{equation}\label{eq:eta_tilde}
			\tilde \eta \sim (\frac{u}{a})^\frac{1}{\ka} = (\frac{u}{\tau \vep^\alpha})^\frac{1}{\ka} \vep \geq C\vep, \text{ where $C$ can be arbitrarily large depending on $\tau_0$},    
		\end{equation}
		which follows from the observation: recalling \cref{eq:f_low_bound}, since $\tilde \theta(\tilde \eta, \tilde \eta ) = a f(\tilde \eta, \tilde \eta) = u$, we have
		\begin{equation}\label{eq:eta_tilde_est}
			\tilde \eta^\ka \lesssim f(\tilde \eta, \tilde \eta) = \frac{u}{a}, \quad \frac{u}{a} = f(\tilde \eta, \tilde \eta) \lesssim \tilde \eta^\ka + \vep^\ka \lesssim \tilde \eta^\ka.
		\end{equation}
		Then we estimate $J_<$. Using $\eta_1 \leq \eta_2$ and $\vep \lesssim \eta_2$, we have
		\begin{equation}
			u = \tilde{\theta} = a f(\eta_1, \eta_2) \lesssim a \eta_1 \eta_2 \eta_2^{\ka - 2} \quad \implies \quad \eta_2 \gtrsim (\frac{u}{a \eta_1})^\frac{1}{\ka-1}, 
		\end{equation}
		which further implies the denominator in \cref{eq:J_sq_def} satisfies
		\begin{equation}
			a\eta_1(\eta_1^{\ka - 2} + \eta_2^{\ka-2}) \geq a\eta_1 \eta_2^{\ka-2} \gtrsim a\eta_1 (\frac{u}{a \eta_1})^\frac{\ka-2}{\ka-1} = (a\eta_1)^\frac{1}{\ka-1} u^\frac{\ka-2}{\ka-1}. 
		\end{equation}
		It then follows that 
		\begin{align}
			J_< \lesssim \int_{2 \pi l - \delta}^{2 \pi l + \delta} \rmd u \int_{\tilde \theta = u} 1_{0 \leq \eta_1 \leq L_0} 1_{0 \leq \eta_2 \lesssim L_0}\frac{1_{\eta_1 \leq \tilde \eta}}{(a\eta_1)^\frac{1}{\ka-1} u^\frac{\ka-2}{\ka-1}} \rmd \eta_1.  
		\end{align}
		Then we have, from the equation above, when $\ka \geq 3$, 
		\begin{equation}
			J_<\lesssim  \int_{2 \pi l - \delta}^{2 \pi l + \delta} \rmd u \int_0^{\tilde \eta} \frac{1}{(a\eta_1)^\frac{1}{\ka-1} u^\frac{\ka-2}{\ka-1}} \rmd \eta_1 \lesssim \frac{\delta}{l^\frac{\ka-2}{\ka-1}} \frac{1}{a^\frac{1}{\ka-1}} \tilde \eta^\frac{\ka-2}{\ka-1} \lesssim  \delta \frac{1}{l^{1 - \frac{2}{\ka}}} \frac{1}{a^\frac{2}{\ka}}, 
		\end{equation}
		and when $\ka = 2$, since $af = u$ and $\eta_2 \lesssim L_0$ implies $\eta_1 \gtrsim u/(aL_0) \sim l/(aL_0)$, 
		% \begin{equation}
		% 	J_< \lesssim \int_{2 \pi l - \delta}^{2 \pi l + \delta} \rmd u \int_{0}^{L_0} 1_{\eta_1 \gtrsim l/(aL_0)} \frac{1}{a\eta_1} \rmd \eta_1 \lesssim \frac{\delta}{a} \int_0^{L_0} \frac{1_{\eta_1 \gtrsim l/(aL_0)}}{\eta_1} \rmd \eta_1.   
		% \end{equation}
        \begin{equation}
			J_< \lesssim \int_{2 \pi l - \delta}^{2 \pi l + \delta} \rmd u \int_{0}^{L_0} 1_{\eta_1 \gtrsim l/(aL_0)} \frac{1}{a\eta_1} \rmd \eta_1 \lesssim \frac{\delta}{a}(\ln L_0 + |\ln \tau| + |\ln \vep|).   
		\end{equation}
		The we consider $J_>$. Using $\eta_1 \geq \tilde \eta \geq \eta_2$ and $a\eta_1(\eta_1^{\ka - 2} + \eta_2^{\ka-2}) \geq a\eta_1^{\ka-1}$, when $\ka \geq 3$, 
		\begin{equation}
			J_> 
			\lesssim \int_{2 \pi l - \delta}^{2 \pi l + \delta} \rmd u \int_{\tilde \eta}^{\infty} \frac{1}{a\eta_1^{\ka-1}} \rmd \eta_1 \lesssim \int_{2 \pi l - \delta}^{2 \pi l + \delta} \frac{1}{a} \frac{1}{\tilde \eta^{\ka-2}} \rmd u\lesssim  \delta \frac{1}{l^{1 - \frac{2}{\ka}}} \frac{1}{a^\frac{2}{\ka}},  
		\end{equation}
		and, when $\ka = 2$, recalling $\tilde \eta \gtrsim \vep$, 
        %$\tilde \eta \sim (u/a)^\frac{1}{2} \sim (l/a)^\frac{1}{2}$,
		% \begin{equation}
		% 	J_> 
		% 	\lesssim \int_{2 \pi l - \delta}^{2 \pi l + \delta} \rmd u \int_{\tilde \eta}^{L_0} \frac{1}{a\eta_1} \rmd \eta_1 \lesssim \frac{\delta}{a} \int_0^{L_0} \frac{1_{\eta_1 \gtrsim (l/a)^\frac{1}{2}}}{\eta_1} \rmd \eta_1.
		% \end{equation}
        \begin{equation}
			J_> 
			\lesssim \int_{2 \pi l - \delta}^{2 \pi l + \delta} \rmd u \int_{\tilde \eta}^{L_0} \frac{1}{a\eta_1} \rmd \eta_1 \lesssim \frac{\delta}{a} (\ln L_0 + |\ln \tilde \eta|) \lesssim \frac{\delta}{a} (\ln L_0 + |\ln \vep|).
		\end{equation}
		% Hence,
		% \begin{equation}
		% 	\| I_{\Lambda^\delta_2} \| \lesssim \sum_{|l| \geq 1} 1_{|l| \lesssim aL_0^\ka} \| I_{\Lambda^\delta_{2, l}} \| \lesssim 
		% 	\left\{
		% 	\begin{aligned}
		% 		&\tau \delta \frac{1}{a^\frac{2}{\ka}} \sum_{l \geq 1} \frac{1_{l \lesssim aL_0^\ka}}{l^{1 - \frac{2}{\ka}}} \lesssim \tau \delta L_0^2, && \ka \geq 3,\\
		% 		&\tau \frac{\delta}{a} \int_0^{L_0} \sum_{l \geq 1} (1_{\eta_1 \gtrsim l/(aL_0)} + 1_{\eta_1 \gtrsim (l/a)^\frac{1}{2}}) \frac{1}{\eta_1} \rmd \eta_1 \lesssim \tau \delta L_0^2, && \ka = 2,    
		% 	\end{aligned} 
		% 	\right.
		% \end{equation}
  %       where we use $\sum_{l \geq 1} (1_{\eta_1 \gtrsim l/(aL_0)} + 1_{\eta_1 \gtrsim (l/a)^\frac{1}{2}}) \lesssim a \eta_1 L_0 + a\eta_1^2$ when $\ka = 2$. 
  		Hence,
		\begin{equation}
			\| \widehat{I}_{\Lambda^\delta_2} \| \lesssim \sum_{|l| \geq 1} 1_{|l| \lesssim aL_0^\ka} \| \widehat{I}_{\Lambda^\delta_{2, l}} \| \lesssim 
			\left\{
			\begin{aligned}
				&\tau \delta \frac{1}{a^\frac{2}{\ka}} \sum_{l \geq 1} \frac{1_{l \lesssim aL_0^\ka}}{l^{1 - \frac{2}{\ka}}} \lesssim \tau \delta L_0^2, && \ka \geq 3,\\
				&\tau \delta L_0^2 (\ln L_0 + |\ln \tau| + |\ln \vep|), && \ka = 2.    
			\end{aligned} 
			\right.
		\end{equation}
		
		Finally, we estimate $I_{\Lambda_{3}^\delta}$. We further decompose $\Lambda_{3}^\delta$ as, with $\Omega^\delta_l := (2 \pi l-\pi, 2\pi l - \delta) \cup (2 \pi l + \delta, 2 \pi l + \pi)$, 
		\begin{equation}
			\Lambda_{3}^\delta = \cup_{|l| \geq 1} \Lambda_{3, l}^\delta, \quad \Lambda_{3, l}^\delta := \{(\xi_1, \xi_2, \xi_3) \in \Lambda_3^\delta: \theta \in \Omega^\delta_l \} \cap \Lambda_{C_0'}^c.   
		\end{equation}
		As a result, we have
		\begin{equation}
			\| \widehat{I}_{\Lambda_3^\delta} \|_{L^1} \leq \sum_{|l| \geq 1} 1_{|l| \lesssim aL_0^\ka} \| \widehat{I}_{\Lambda_{3, l}^\delta} \|_{L^1}. 
		\end{equation}
		To estimate ${J}_{\Lambda_{3, l}^\delta}$, we need two auxiliary results. When $(\xi_1, \xi_2, \xi_3) \in \Lambda_{3, l}^\delta$, similar to \cref{eq:sum_by_part,eq:sum_by_part_est}, we have $|\theta - 2 \pi l| \in (\delta, \pi)$ and 
		\begin{equation}
			\left| \sum_{k=0}^{n-1} e^{i z_k \Phi(\xi_1, \frac{\xi_2}{\vep} + \xi_3)} \right| = \left| \sum_{k=0}^{n-1} e^{i k \theta} \right| \lesssim \frac{1}{|\theta - 2 \pi l|}. 
		\end{equation}
		  Furthermore, by integration by parts, similar to \cref{eq:inner_int}, we have
		\begin{equation}\label{eq:inner_int_est_2}
			\left| \int_{0}^{\tau} \int_0^s e^{i s \vep^\alpha P(\xi)} e^{-is_1 \vep^\alpha P(\frac{\xi_2}{\vep} + \xi_3)} \rmd s_1 \rmd s \right| = \left| \int_{0}^{\tau} \int_0^s e^{i s \vep^\alpha (P(\xi) - P(\frac{\xi_2}{\vep} + \xi_3))} e^{i (s - s_1) \vep^\alpha P(\frac{\xi_2}{\vep} + \xi_3)} \rmd s_1 \rmd s \right| \lesssim \frac{\tau^2}{|\theta|}. 
		\end{equation}
		Then using the above estimates, and noting that $|\theta| \sim 2 \pi |l|$ when $(\xi_1, \xi_2, \xi_3) \in \Lambda_{3, l}^\delta$, we have
		\begin{align}
			\| \widehat{I}_{\Lambda_{3, l}^\delta}\|_{L^1} 
			&\lesssim \tau^2 \int_{\R^3} 1_{\Lambda_{3, l}^\delta} \frac{1}{|\theta||\theta - 2 \pi l|} |\widehat{R}(\xi_1)| |\widehat{R}(\xi_2)| |\widehat{\psi}(0, \xi_3)| \rmd \xi_1 \rmd \xi_2 \rmd \xi_3 \notag \\
			&\lesssim \frac{\tau^2}{|l|} \int_{\R^3} 1_{\Lambda_{3, l}^\delta} \frac{1}{| \theta - 2 \pi l |} |\widehat{R}(\xi_1)| |\widehat{R}(\xi_2)| |\widehat{\psi}(0, \xi_3)| \rmd \xi_1 \rmd \xi_2 \rmd \xi_3.  
		\end{align}
		Performing the same change of variables \cref{eq:change_of_variable_3} and recalling \cref{eq:theta_tilde_def}, we obtain
		\begin{align}
			\| \widehat{I}_{\Lambda_{3, l}^\delta}\|_{L^1} 
			&\lesssim \frac{\tau^2}{|l|} \int_{\R^3} 1_{\Lambda_{3, l}^\delta}(\xi_1, \xi_2, \xi_3) \frac{1}{| \tilde \theta - 2 \pi l |} |\widehat{R}(\eta_1)| |\widehat{R}(\frac{\eta_2-\eta_1 - 2\vep\eta_3}{2})| |\widehat{\psi}(0, \eta_3)| \rmd \eta_1 \rmd \eta_2 \rmd \eta_3 \notag \\
			&\lesssim \frac{\tau^2}{|l|} \| \widehat{R} \|_{L^\infty} \int_{\R^3} (1_{|\eta_1| \geq C_0' \vep} + 1_{|\eta_2| \geq C_0' \vep})1_{|\eta_1| \leq L_0} 1_{|\eta_2| \lesssim L_0} 1_{\tilde \theta \in \Omega^\delta_{l}} \frac{1}{|\tilde \theta - 2 \pi l|} |\widehat{R}(\eta_1)| |\widehat{\psi}(0, \eta_3)| \rmd \eta_1 \rmd \eta_2 \rmd \eta_3 \notag \\
			&\lesssim \frac{\tau^2}{|l|} \| \widehat{R} \|^2_{L^\infty} \| \widehat{\psi} \|_{L^1} \int_{\R^2} (1_{|\eta_1| \geq C_0' \vep} + 1_{|\eta_2| \geq C_0' \vep}) 1_{|\eta_1| \leq L_0} 1_{|\eta_2| \lesssim L_0} \frac{1_{\tilde \theta \in \Omega^\delta_{l}}}{|\tilde \theta - 2 \pi l|}  \rmd \eta_1 \rmd \eta_2. 
		\end{align}
		Following a similar procedure of using the co-area formula, we can obtain
		\begin{equation}
			\| \widehat{I}_{\Lambda_{3}^\delta}\|_{L^1} \lesssim \sum_{|l| \geq 1} 1_{|l| \lesssim aL_0^\ka} \| \widehat{I}_{\Lambda_{3, l}^\delta}\|_{L^1} \lesssim \left\{
			\begin{aligned}
				&\frac{\tau^2}{a^{\frac{2}{\ka}}} \ln \frac{\pi}{\delta} \sum_{l \geq 1} \frac{1_{l \lesssim aL_0^\ka}}{l^{2 - \frac{2}{\ka}}} \lesssim \tau^{2 - \frac{2}{\ka}} \vep^{2 - \frac{2}{\ka}\alpha} \ln \frac{\pi}{\delta}, && \ka \geq 3, \\
				&\tau \vep^{2-\alpha} \ln \frac{\pi}{\delta} (\ln L_0 + |\ln \tau| + |\ln \vep|)^2, &&\ka = 2. 
			\end{aligned}
			\right. 
		\end{equation}
		Finally, choosing $L_0 = \vep^{- \frac{2}{\sigma}}$ and $\delta = \vep^{2 - \frac{2}{\ka}\alpha + \frac{4}{\sigma}+1}$, we obtain
		\begin{align}
			& \| \widehat{I} \|_{L^1} 
			\lesssim \frac{\tau}{L_0^\sigma} + \tau\vep^2 +   \\ \notag
			\quad  + & \left\{
			\begin{aligned}
				&\tau \vep^{2 - \frac{2}{\ka}\alpha} |\ln \vep| + \tau \delta L_0^2 + \tau^{2 - \frac{2}{\ka}} \vep^{2 - \frac{2}{\ka}\alpha} \ln \frac{1}{\delta} \lesssim  \tau \vep^{2 - \frac{2}{\ka}\alpha} \ln \frac{1}{\vep}, && \ka \geq 3, \\
				&\tau \vep^{2 - \alpha} |\ln \vep|^2 + (\tau \delta L_0^2 + \tau \vep^{2 - \alpha} \ln \frac{1}{\delta}) (\ln L_0 + |\ln \tau| + |\ln \vep|)^2  \lesssim  \tau \vep^{2 - \alpha} \ln \frac{1}{\vep} (\ln \frac{1}{\vep}+ \ln \frac{1}{\tau})^2, && \ka = 2, 
			\end{aligned}
			\right.
		\end{align}
		which completes the proof for even $\ka$. 
		
		We now extend the results to odd $\ka$. The estimate of $I_{\Lambda_0}$ and $I_{\Lambda_0'}$ are the same, and the estimate of $I_{\Lambda_1}$ is again similar to \cref{prop:est2_continuous}. However, the estimates in the ``large phase" regime differs in two aspects: when using the co-area formula, the geometry of the curve $\tilde \theta(\eta_1, \eta_2) = u$ is different, and the singularity of the integrals becomes more complicated. 
        
        We first estimate $I_{\Lambda^\delta_2}$. Using the same change of variables \cref{eq:change_of_variable_3}, we still arrive at \cref{eq:Ilambda_2}. We adopt the same definitions of $\tilde \theta$, $a$, and $f$ in \cref{eq:theta_tilde_def,eq:f_def}, and $\tilde \eta \geq C_0' \vep$ is again the unique solution to $\tilde \theta(\tilde \eta, \tilde \eta) = u$ which satisfies \cref{eq:eta_tilde}. 
        % Similarly, using \cref{lem:phase}, $f$ can be explicitly given as
        % \begin{equation*}
        %     f = \eta_1 \sum_{0 \leq 2j < \ka} \vep^{2j} \tilde d_{\ka - 2j} Q_{\ka-2j}(\eta_1^2, \eta_2^2). 
        % \end{equation*}
        On $\Lambda_2^\delta$, by \cref{eq:phase_est_3_1} (and similar arguments), we have
		% \begin{equation}\label{eq:theta_tilde_odd}
		% 	\tilde \theta(\eta_1, \eta_2) = \frac{\tau}{\vep^{\ka - \alpha}} \eta_1 \sum_{0 \leq 2j < \ka} \vep^{2j} \frac{d_{\ka - 2j}}{2^{\ka - 2j -1}} Q_{\ka-2j}(\eta_1^2, \eta_2^2) = a f, 
		% \end{equation}
		% and, with $\tilde Q_{\ka - 2j - 2}(x, y) : = \partial_y Q_{\ka - 2j}(x, y)$, 
		% \begin{equation}
		% 	a = \frac{\tau}{\vep^{\ka - \alpha}}, \quad f = \eta_1 \sum_{0 \leq 2j < \ka} \vep^{2j} \tilde d_{\ka - 2j} Q_{\ka-2j}(\eta_1^2, \eta_2^2), \quad f_{\eta_2} = 2\eta_1 \eta_2 \sum_{0 \leq 2j < \ka} \vep^{2j} \tilde d_{\ka - 2j} \tilde Q_{\ka-2j-2}(\eta_1^2, \eta_2^2). 
		% \end{equation}
		% As $\eta_1 \gg \vep$ or $\eta_2 \gg \vep$ on $\Lambda_{2}^\delta$ and $\Lambda_3^\delta$, by \cref{cor:phase} (in fact by the same arguments), we have
		\begin{equation}
			|f| \gtrsim |\eta_1| (\eta_1^{\ka-1} + \eta_2^{\ka-1}), \quad |f_{\eta_1}| \gtrsim \eta_1^{\ka-1} + \eta_2^{\ka-1}, \quad |f_{\eta_2}| \gtrsim |\eta_1 \eta_2| (\eta_1^{\ka - 3} + \eta_2^{\ka - 3}). 
		\end{equation}
		As $\tilde \theta = a f \neq 0$, we have $f_{\eta_1} \neq 0$ and $|\nabla \tilde \theta| \neq 0$ along the curve. Using the same co-area formula and the change of measure, by the symmetry, the estimate reduces to the estimate of $J$ below 
		\begin{equation}\label{eq:J_bs_def_2}
			J =  \int_{2 \pi l - \delta}^{2 \pi l + \delta} \rmd u \int_{\tilde \theta = u} (1_{\eta_1<\eta_2} + 1_{\eta_1 > \eta_2}) 1_{0 \leq \eta_1 \leq L_0} 1_{0 \leq \eta_2 \lesssim L_0}\frac{1_{\eta_1 \geq C_0' \vep} + 1_{\eta_2 \geq C_0' \vep}}{a\eta_1 \eta_2 (\eta_1^{\ka - 3} + \eta_2^{\ka - 3})} \rmd \eta_1 =:J_{<} + J_{>}, \quad l \geq 1.
		\end{equation}
		For $J_{<}$, as $\eta_1 \leq \eta_2$, we have $\vep \lesssim \eta_2$, and
		\begin{equation}
			u = \tilde{\theta} = a f \lesssim a \eta_1 \eta_2^{\ka - 1} \quad \implies \quad \eta_2 \gtrsim (\frac{u}{a \eta_1})^\frac{1}{\ka-1}, 
		\end{equation}
		which further implies, for the denominator in $J_<$, 
		\begin{equation}
			a\eta_1 \eta_2 (\eta_1^{\ka - 3} + \eta_2^{\ka - 3}) \geq a\eta_1 \eta_2 \eta_2^{\ka-3} \gtrsim a\eta_1 (\frac{u}{a \eta_1})^\frac{\ka-2}{\ka-1} = (a\eta_1)^\frac{1}{\ka-1} u^\frac{\ka-2}{\ka-1}. 
		\end{equation}
		Then we have, noting $\eta_1 \leq \tilde \eta \sim (\frac{u}{a})^{1/\ka}$,
		\begin{equation}
			J_< \lesssim \int_{2 \pi l - \delta}^{2 \pi l + \delta} \rmd u \int_0^{\tilde \eta} \frac{1}{(a\eta_1)^\frac{1}{\ka-1} u^\frac{\ka-2}{\ka-1}} \rmd \eta_1 \lesssim \delta \frac{1}{l^{1 - \frac{2}{\ka}}} \frac{1}{a^\frac{2}{\ka}}. 
		\end{equation}
		In the following, we consider $J_>$. When $\eta_1 \geq \eta_2$, let $\tilde \eta_{1, 0} > 0$ be the unique solution to (since $\eta_1 \geq C_0' \vep \gg \vep$)
		% \begin{equation}\label{eq:root}
		% 	\tilde \theta(\eta_1, 0) = a\eta_1 \sum_{0 \leq 2j < \ka} \vep^{2j} \tilde d_{\ka - 2j} Q_{\ka-2j}(\eta_1^2, 0) = a \eta_1 \sum_{0 \leq 2j < \ka} \vep^{2j} \tilde d_{\ka - 2j} (\ka - 2j) \eta_1^{\ka - 2j - 1} = u.
		% \end{equation}
        \begin{equation}\label{eq:root}
			\tilde \theta(\eta_1, 0)  = u, \quad \eta_1 > 0. 
		\end{equation}
		Similar to \cref{eq:eta_tilde_est}, we have $\tilde \eta_{1, 0} \sim (\frac{u}{a})^{1/\ka}$. Also, we have that on $\tilde \theta = u$, $\eta_1 \leq \tilde \eta_{1, 0}$. To estimate the denominator in $J_>$, we define 
        % $g(y) := f(y, 0) = \sum_{0 \leq 2j < \ka} \vep^{2j} \tilde d_{\ka - 2j} (\ka - 2j) y^{\ka-2j}$.
        a polynomial $g(y) := f(y, 0)$.
        Then, by \cref{eq:root}, 
		\begin{equation}\label{eq:g_relation}
			a g(\tilde \eta_{1, 0}) = u = a f(\eta_1, \eta_2). 
		\end{equation}
		In fact, $g(\eta_1)$ consists of all the leading order terms of $\eta_1$ at each $O(\vep^{2j})$ scale in $f(\eta_1, \eta_2)$, and we have, for any $0 \leq x \leq y$, $y \geq C_0' \vep$,
        \begin{equation}\label{eq:fg_dif}
            f(y, x) - g(y) \lesssim y^{\ka - 2} {x}^2. 
        \end{equation}
        Moreover, for any $C_0' \vep \leq x \leq y$, we have, when $C_0'$ is sufficiently large, (see Appendix \ref{append:g_est})
        \begin{align}\label{eq:g_low_bound}
			g(y) - g(x) \gtrsim y^{\ka} - {x}^{\ka}. 
		\end{align}
		Then we have, along the curve $\tilde \theta = u$, noting $C_0'\vep \leq \eta_1 \leq \tilde \eta_{1, 0}$ and $0 \leq \eta_2 \leq \eta_1$, using \cref{eq:g_relation,eq:g_low_bound,eq:fg_dif}, 
		\begin{equation}
			\tilde \eta_{1, 0}^{\ka} - \eta_1^{\ka} \lesssim g(\tilde \eta_{1, 0}) - g(\eta_1) = f(\eta_1, \eta_2) - g(\eta_1) 
            %\lesssim \eta_1 \sum_{0 \leq 2j <\ka-2} \eta_1^{2j} \eta_1^{\ka - 2j -3} \eta_2^2
            \lesssim \eta_1^{\ka - 2} \eta_2^2, 
		\end{equation} 
		which implies 
		\begin{equation}
			\eta_2 \gtrsim (\frac{\tilde \eta_{1, 0}^{\ka} - \eta_1^{\ka}}{\eta_1^{\ka - 2}})^\frac{1}{2} \gtrsim \frac{(\tilde \eta_{1, 0} - \eta_1)^\frac{1}{2} \eta_1^\frac{\ka - 1}{2}}{\eta_1^\frac{\ka-2}{2}} = (\tilde \eta_{1, 0} - \eta_1)^\frac{1}{2} \eta_1^\frac{1}{2},  
		\end{equation}
		which further implies, for the denominator in $J_{>}$ \cref{eq:J_bs_def_2}, 
		\begin{equation}
			a\eta_1 \eta_2 (\eta_1^{\ka - 3} + \eta_2^{\ka - 3}) \geq a\eta_1 \eta_2 \eta_1^{\ka-3} \gtrsim a \eta_1^{\ka - 2} \eta_1^\frac{1}{2}(\tilde \eta_{1, 0} - \eta_1)^\frac{1}{2}. 
		\end{equation}
		Then
		\begin{align}
			J_> 
			&\lesssim \int_{2 \pi l - \delta}^{2 \pi l + \delta} \rmd u \int_{\tilde \eta}^{\tilde \eta_{1, 0}} \frac{1}{a \eta_1^{\ka - 2} \eta_1^\frac{1}{2}(\tilde \eta_{1, 0} - \eta_1)^\frac{1}{2}} \rmd \eta_1 \notag \\
			&\lesssim \frac{1}{a (l/a)^\frac{\ka - 2}{\ka}} \int_{2 \pi l - \delta}^{2 \pi l + \delta} \rmd u \int_{0}^{\tilde \eta_{1, 0}} \frac{1}{\eta_1^\frac{1}{2}(\tilde \eta_{1, 0} - \eta_1)^\frac{1}{2}} \rmd \eta_1 \lesssim \frac{\delta}{a^\frac{2}{\ka}} \frac{1}{l^{1 - \frac{2}{\ka}}}. 
		\end{align}
		Hence, we have the same bound of $I_{\Lambda^\delta_2}$ for odd $\ka$ as  
		\begin{equation}
			\| I_{\Lambda^\delta_2} \| \lesssim \sum_{|l| \geq 1} 1_{|l| \lesssim aL_0^\ka} I_{\Lambda^\delta_{2, l}} \lesssim \tau \delta \frac{1}{a^\frac{2}{\ka}} \sum_{l \geq 1} \frac{1_{l \lesssim aL_0^\ka}}{l^{1 - \frac{2}{\ka}}}  \lesssim \tau \delta L_0^2. 
		\end{equation}
		The extension of the estimate of $I_{\Lambda^\delta_3}$ to odd $\ka$ is similar and omitted.		This concludes the proof for odd $\ka$, and completes the proof.
	\end{proof}
	
	Finally, with \cref{cor:E1_est,cor:E2_est_1,cor:E2_est_2}, we show the proof of \cref{thm:main}. 
	\begin{proof}[Proof of \cref{thm:main}]
		From \cref{eq:error_eq}, noting $e^n_0 = 0$, \cref{eq:e_2,eq:E2,eq:E2_decomp}, by the triangle inequality, we have
		\begin{align}\label{en}
			\| e^n \| 
			&\leq \sum_{m=0}^\infty \| e^n_m \| = \sum_{m=1}^\infty \| \mathcal{E}^n_{m, 1} \| + \sum_{m=2}^\infty \| \mathcal{E}^n_{m, 2} \| \leq \sum_{m=1}^\infty \| \mathcal{E}^n_{m, 1} \| + \sum_{m=2}^\infty \sum_{j=0}^{m-2} \| \mathcal{E}^n_{m, 2, j} \| \notag \\
			&\leq \sum_{m=1}^\infty \| \mathcal{E}^n_{m, 1} \| + \sum_{m=2}^\infty \sum_{j=0}^{m-2} (\| e^{-iz_n \vep^\alpha D_\kappa} (\mathcal{I}_\mathrm{d}^{m-2-j} \mathcal{J}_{j, 1})(z_n) \| + \| e^{-iz_n \vep^\alpha D_\kappa} (\mathcal{I}_\mathrm{d}^{m-2-j} \mathcal{J}_{j, 2})(z_n) \|). 
		\end{align}
		By \cref{cor:E1_est,cor:E2_est_1,cor:E2_est_2} and recalling \cref{eq:E2_decomp}, we have, from \cref{en}, 
		\begin{align}\label{eq:error_final}
			\| e^n \| 
			&\lesssim \sup_{0 \leq z \leq z_n} \| (\mathcal{I}_\rmd(\mathfrak{u} - \mathfrak{u}_\rmd))(z) \| \sum_{m \geq 1} \frac{C^{m-1}}{(m-1)!} \notag \\
			&\quad + \tau \vep^{1 + \frac{\ka-1}{\ka}\alpha} \sum_{m \geq 2} \frac{C^{m-2}}{(m-2)!} + \tau \vep^{2 - \frac{2}{\kappa} \alpha} \ln \vep^{-1} (\ln \vep^{-1} + \ln \tau^{-1})^{\vartheta(\ka)} \sum_{m \geq 2} \frac{C^{m-2}}{(m-2)!} \notag \\
			&\lesssim e^{C} \sup_{0 \leq z \leq z_n} \| (\mathcal{I}_\rmd(\mathfrak{u} - \mathfrak{u}_\rmd))(z) \| + e^C \tau \vep^{1 + \frac{\ka-1}{\ka}\alpha} + e^{C} \tau \vep^{2 - \frac{2}{\kappa} \alpha} \ln \vep^{-1} (\ln \vep^{-1} + \ln \tau^{-1})^{\vartheta(\ka)}. 
		\end{align}
		When $\alpha \geq \frac{2\ka}{\ka + 2}$, we have $1 + \frac{\ka-1}{\ka}\alpha \geq \alpha \geq 2 - \frac{2}{\ka}\alpha$. Then using \cref{eq:prop_est1_disc_1}, we have
		\begin{equation*}
			\| e^n \| \lesssim \tau \vep^\alpha + \tau \vep^{1 + \frac{\ka-1}{\ka}\alpha} + \tau \vep^{2 - \frac{2}{\kappa} \alpha} \ln \vep^{-1} (\ln \vep^{-1} + \ln \tau^{-1})^{\vartheta(\ka)} \lesssim \tau \vep^{2 - \frac{2}{\kappa} \alpha} \ln \vep^{-1} (\ln \vep^{-1} + \ln \tau^{-1})^{\vartheta(\ka)}. 
		\end{equation*} 
		When $0 \leq \alpha < \frac{2\ka}{\ka + 2}$, assuming \cref{eq:assumption_3}, we have \cref{eq:prop_est1_disc_3}, and thus
		\begin{equation*}
			\| e^n \| \lesssim \tau \vep^{1 + \frac{\ka - 1}{\ka}\alpha}(1 + \tau^{1-\frac{1}{\ka}} \ln \vep^{-1}) + \tau \vep^{2 - \frac{2}{\kappa} \alpha} \ln \vep^{-1} (\ln \vep^{-1} + \ln \tau^{-1})^{\vartheta(\ka)}. 
		\end{equation*} 
		These complete the proof of \cref{thm:main}.
	\end{proof}
	The error bound \cref{eq:non-optimal} in \cref{rem:non-optimal} can be obtained by replacing \cref{eq:prop_est1_disc_3} with \cref{eq:prop_est1_disc_2} in the proof above. 
	
%	\begin{proof}
%		We have
%		\begin{align}
%			\| e^{-iz_n \vep^\alpha D_\kappa} (\mathcal{I}_\mathrm{d}^{m-2-j} \mathcal{J}_{j, 2})(z_n)\| \lesssim \| (\mathcal{I}_\mathrm{d}^{m-2-j} \mathcal{J}_{j, 2})(z_n)\| \lesssim \frac{(\| R \|T)^{m-2-j}}{(m - 2 - j)!} \sup_{0 \leq z \leq T} \| \mathcal{J}_{j, 2}(z) \|. 
%		\end{align}
%		For $\mathcal{J}_{j, 2}$ \cref{eq:Jj2_def}, when $j=1$,  we have, by \cref{prop:est2_discrete}
%		\begin{equation}
%			\| \mathcal{J}_{1, 2} \| \lesssim \tau \vep^{2 - \frac{2}{\kappa} \alpha} (\ln \vep^{-1})^{1+\sigma(\ka)}, 
%		\end{equation}
%		and, when $j \geq 2$, by \cref{prop:est2_discrete}, we have
%		\begin{align}
%			\| \mathcal{J}_{j, 2} \| 
%			&\lesssim  \| R \| \sum_{k=0}^{n-1} \int_{z_k}^{z_{k+1}} \int_{z_k}^s \| (\mathcal{I}^jv)(s_1) \| \rmd s_1 \rmd s \notag \\
%			&\lesssim  \tau \| R \|T \sup_{0 \leq z \leq T} \| (\mathcal{I}^jv)(z) \| \notag \\
%			&\lesssim  \tau \frac{\| R \|^{j-1}T^{j-1}}{(j-1)!} \sup_{0 \leq z \leq T} \| (\mathcal{I}v)(z) \| \notag \\
%			&\lesssim  \tau \vep^{2 - \frac{2}{\ka}\alpha} (\ln \vep^{-1})^{1+\sigma(\ka)} \frac{\| R \|^{j-1}T^{j-1}}{(j-1)!}. 
%		\end{align}
%	\end{proof}
	
	\section{Numerical results}\label{sec:numer_resutls}
	In this section, we present numerical results that confirm the theoretical analysis in \cref{thm:exact} and the error estimates in \cref{thm:main}. We focus on two typical examples of \cref{eq:mu}: (i) the Schr\"odinger-type
	\begin{equation}\label{eq:model1}
		\partial_z \mu = i \vep^\alpha \partial_{xx} \mu + R_\vep \mu, \quad R(x) = -e^{-\frac{x^2}{8}}, \qquad z \geq 0,  
	\end{equation}
	and (ii) the KdV-type
	\begin{equation}\label{eq:model2}
		\partial_z \mu = \vep^\alpha \partial_{xxx} \mu + R_\vep \mu, \quad R(x) = -e^{-|x|}, \qquad z \geq 0. 
	\end{equation}
	We choose different potential $R$ in \cref{eq:model1,eq:model2} to show that our results \cref{thm:exact,thm:main} are valid for both smooth potentials and low(er) regularity potentials. In all the examples in this section, the initial datum $\mu_0$ is taken as a Gaussian with $\mu_0(x) = e^{-x^2/2}$ for $x \in \R$. 
	
	While the exponential integrator \cref{eq:EWI_mu} is a semidiscretization in $z$, we further discretize it in $x$ by truncating the whole space $\R$ to some bounded domain $(-L, L)$ with $L>0$ large enough and equip it with periodic boundary conditions. Then we apply the standard Fourier pseudospectral method to $x$ and obtain a full discretization scheme \cite{bao2002,spectral2011}. Note that due to the concentration of the potential $R$ which also induces the oscillation in the exact solution $\mu$, we need the degrees of freedom in $x$ to be of order $1/\vep$ (equivalently the mesh size $h \sim \vep$), to obtain an accurate approximation. Non-uniform spatial grids could also be considered but we do pursue this further here.
	
	\subsection{Estimates of the exact solution} In this section, we present numerical results on the growth rate of $\mu$ and its derivatives in terms of $\vep$. A step size $\tau = 10^{-4}$ and mesh size $h = 2^{-11}$ are used for all numerical results in this subsection. 
	
	We start with the Schr\"odinger case \cref{eq:model1} and consider $\alpha \in \{\frac{3}{4}, 1, \frac{4}{3}\}$. Examples for different values of $\vep$ were also plotted in \cref{fig:soln_diff_alpha}. We choose a spatial domain $(-16, 16)$. We compute the solutions up to $z = 1$ for different values of $\vep$ from $2^{-10}$ to $2^{-4}$. In \cref{fig:exact_soln_diff_eps_Schroedinger}, we plot the growth of the norm $\| \partial_x^j (\mu(z) - e^{i z \vep \alpha} \mu_0) \|$ with respect to $\vep$ at $z = 1$ for $j = 0, 1, 2$. The numerical results closely match our theoretical results in \cref{thm:exact}. In particular, when $\alpha = 1$, the logarithmic growth of $\| \partial_x \mu(z) \|$ in terms of $\vep$ is also observed. These suggest the sharpness of the results in \cref{thm:exact}. 
	
	\begin{figure}[htbp]
		\centering
		{\includegraphics[width=0.3\textwidth]{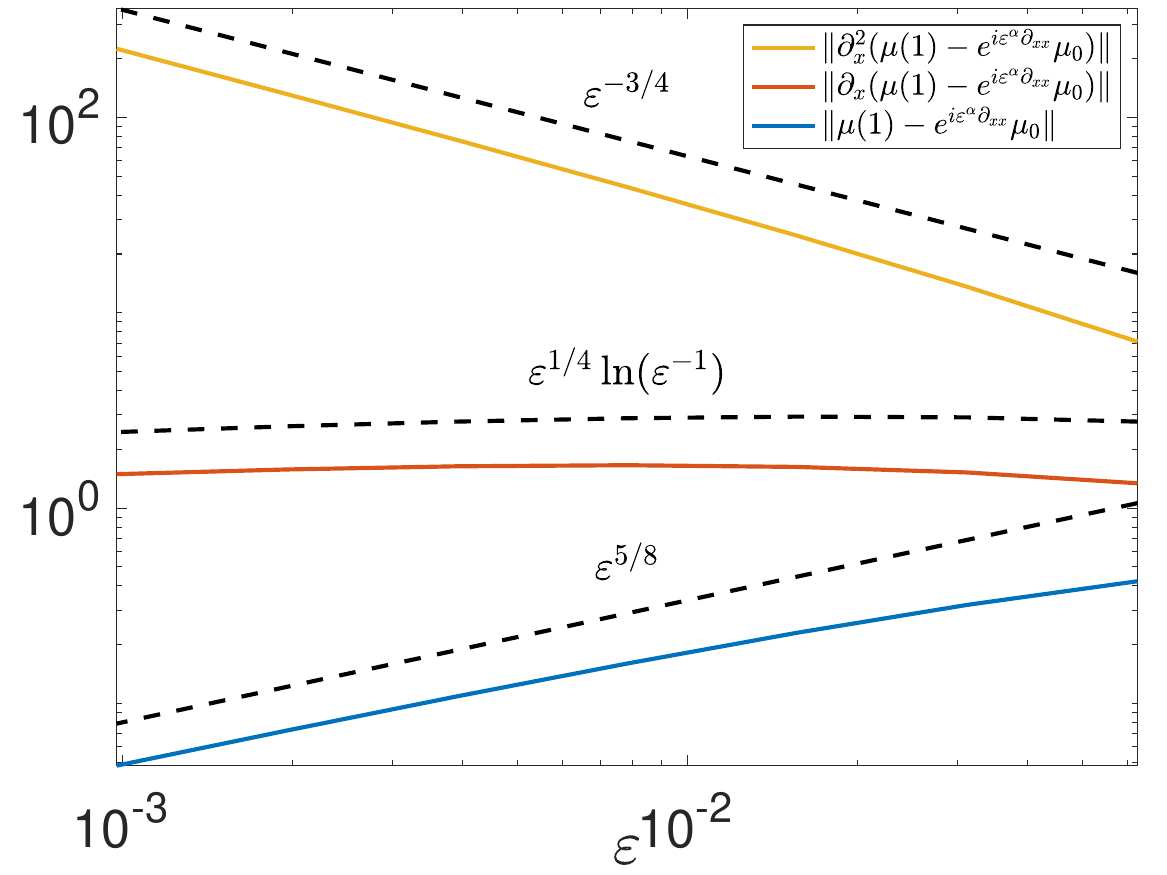}}\hspace{1em}
		{\includegraphics[width=0.3\textwidth]{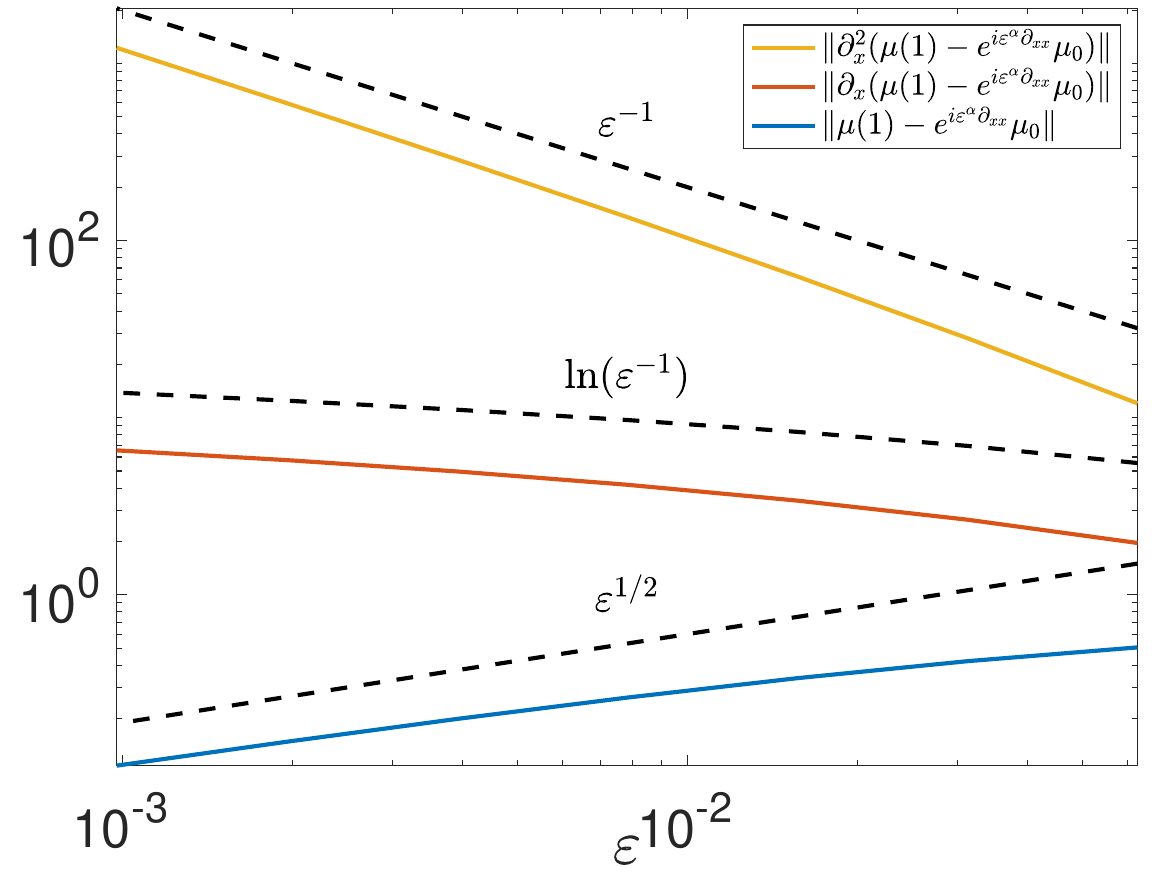}}\hspace{1em}
		{\includegraphics[width=0.3\textwidth]{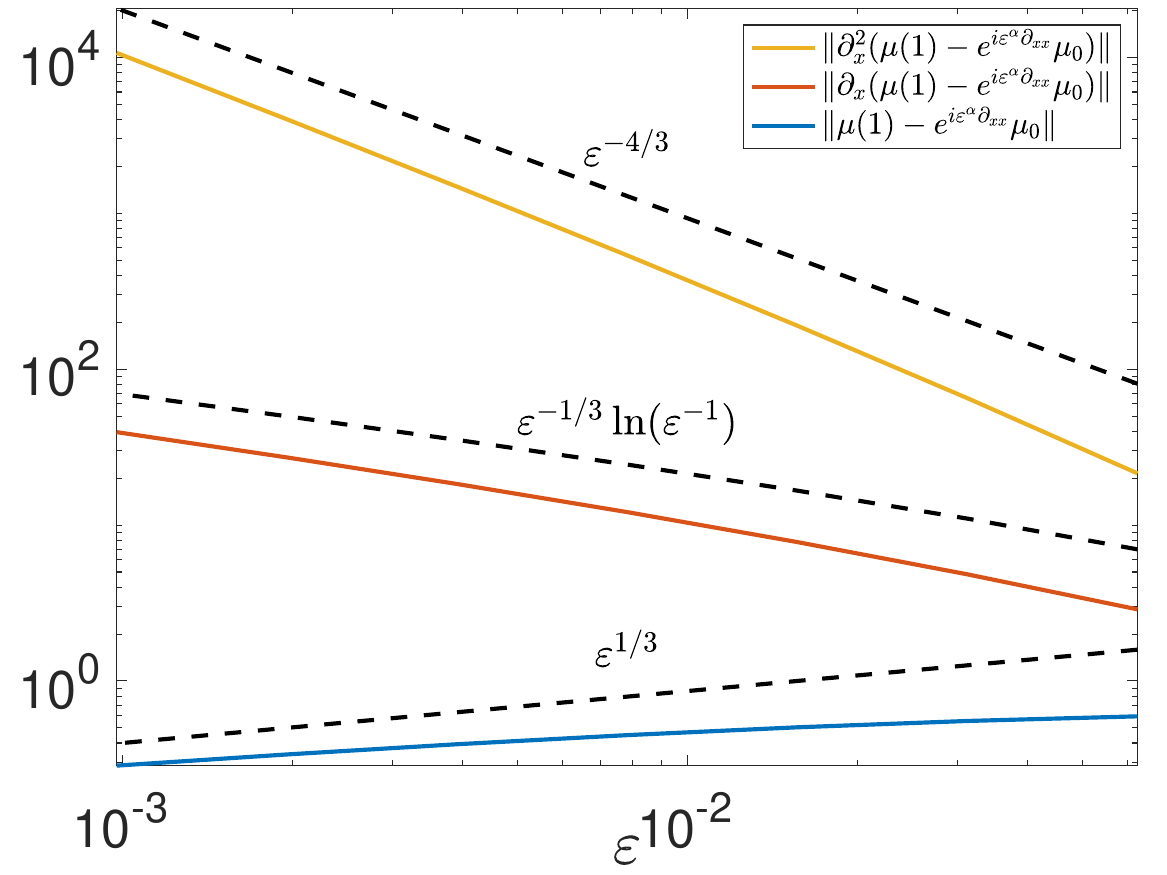}}
		\caption{The difference of the exact solution $\mu(z)$ and the free solution $e^{iz\vep^\alpha}\mu_0$ at $z = 1$ in terms of $\vep$ for $\alpha = \frac{3}{4}$ (left), $\alpha = 1$ (middle), and $\alpha = \frac{4}{3}$ (right)}
		\label{fig:exact_soln_diff_eps_Schroedinger}
	\end{figure}
	
	We next consider the KdV case \cref{eq:model2} with $\alpha \in \{1, \frac{3}{2}, 2\}$ in \cref{eq:mu}. We choose a larger spatial domain $(-32, 32)$. Again, we compute the solutions up to $z = 1$ for different values of $\vep$ from $2^{-10}$ to $2^{-4}$. The observed rates follow our theoretical results in \cref{thm:exact} very accurately, including the logarithmic growth of the second-order derivative when $\alpha = 1$, and the uniform boundedness of the first-order derivative when $\alpha = 3/2$. These again suggest the sharpness of the estimates in \cref{thm:exact}. 
		
	\begin{figure}[htbp]
		\centering
		{\includegraphics[width=0.3\textwidth]{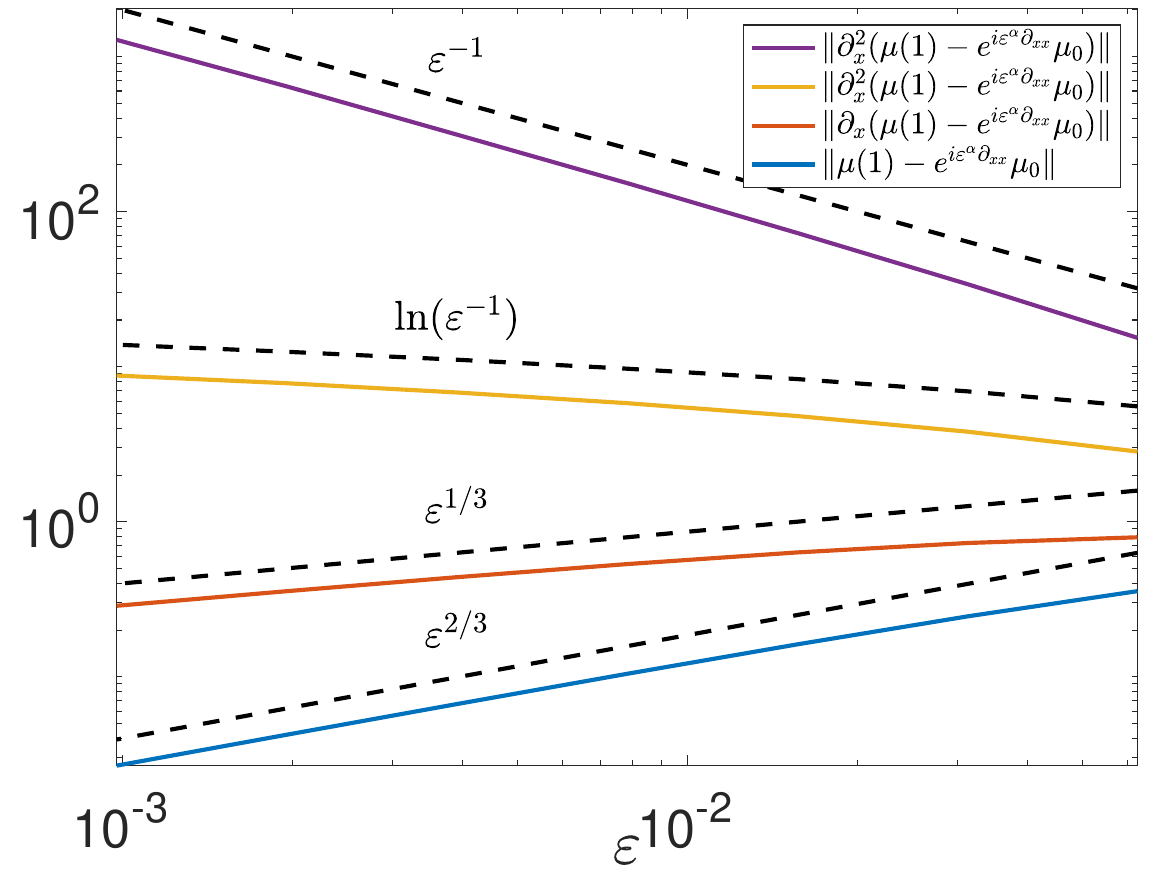}}\hspace{1em}
		{\includegraphics[width=0.3\textwidth]{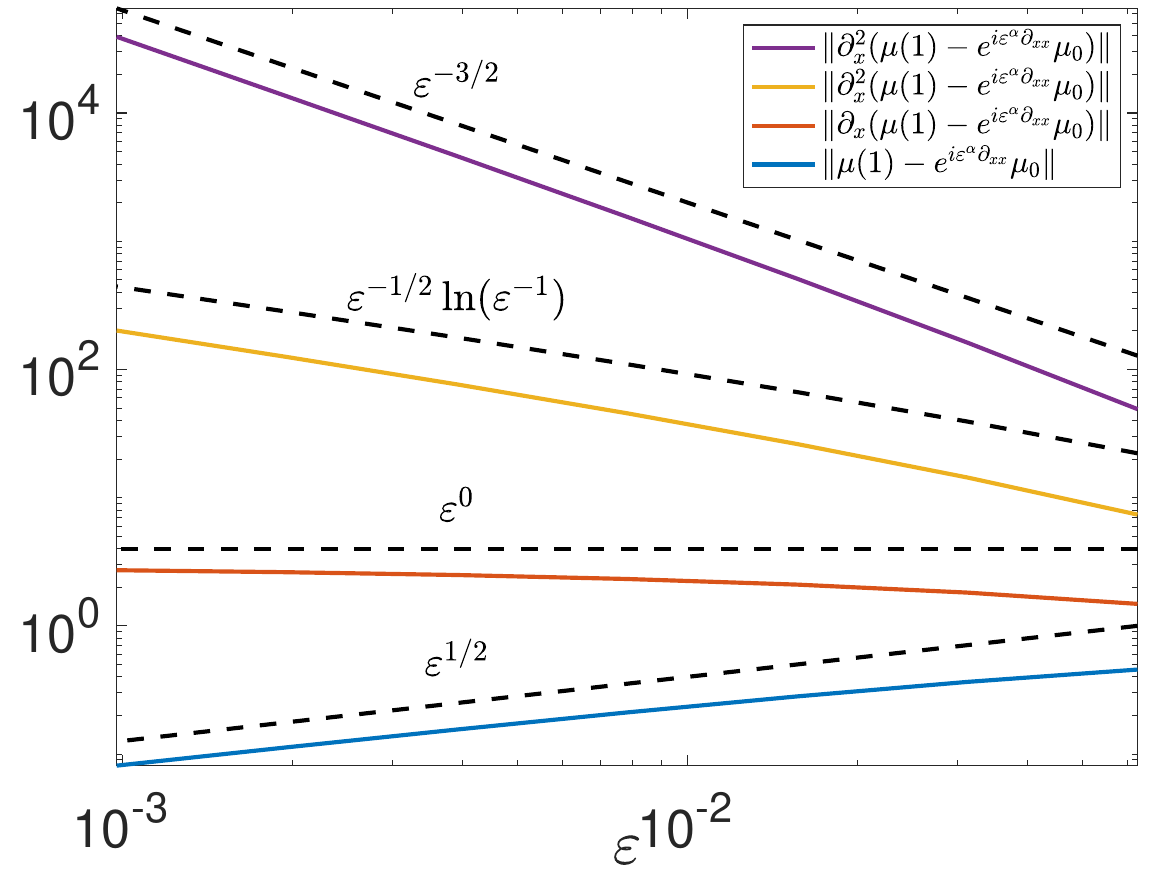}}\hspace{1em}
		{\includegraphics[width=0.3\textwidth]{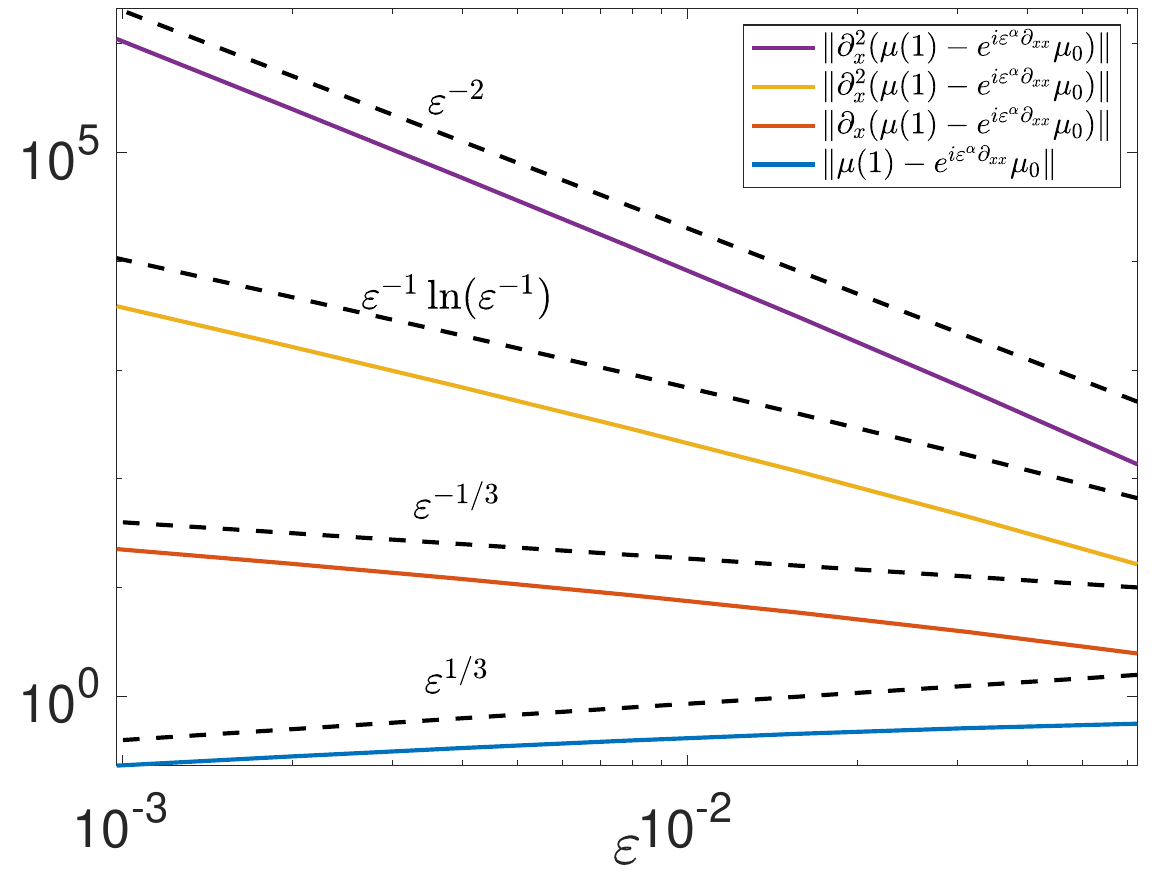}}
		\caption{The difference of the exact solution $\mu(z)$ and the free solution $e^{iz\vep^\alpha}\mu_0$ at $z = 1$ in terms of $\vep$ for $\alpha = 1$ (left), $\alpha = \frac{3}{2}$ (middle), and $\alpha = 2$ (right)}
		\label{fig:exact_soln_diff_eps_KdV}
	\end{figure}
	
	\subsection{Convergence rates of the exponential integrator}
	In this subsection, we use the exponential integrator \cref{eq:EWI_mu} to solve \cref{eq:model1,eq:model2} with different $\alpha$, and $\vep$, and compute the approximation errors. The reference solutions are computed using the same method with $\tau_\text{ref} = 10^{-4}$ and $h_\text{ref} = 2^{-11}$. We also use the same mesh size $h=h_\text{ref}$ for the computation of the numerical solution. 

	We first present the results for the Schr\"odinger case \cref{eq:model1} with $\alpha \in \{\frac{1}{2}, \frac{2}{3}, 1\}$. These three values of $\alpha$ are chosen according to the error estimate \cref{thm:main} such that when $\alpha = \frac{1}{2}$, the first term in \cref{thm:main} dominates, and when $\alpha = 1$, the second term dominates. When $\alpha = \frac{2}{3}$, the two terms balance each other and the convergence rate in $\vep$ achieves a maximum. The computational domain is chosen as $(-16, 16)$ and the final step is $z_n=1$. For each $\alpha$,  we compute the numerical errors for different $\vep$ from $2^{-10}$ to $2^{-4}$ with $\tau$ varying from $10^{-3}$ to $10^{-1}$. The numerical results are presented in \cref{fig:conv_dt_Schroedinger}, where we demonstrate the errors divided by $\min\{\vep^{1+\frac{\alpha}{2}}, \vep^{2-\alpha} \ln \vep^{-1}\}$ which is the expected order in $\vep$ according to \cref{thm:main}. We observe that the curves for different $\vep$ almost collapse to one single line representing $O(\tau)$. This observation suggests that the convergence orders in \cref{thm:main} are optimal in both $\tau$ and $\vep$.   
	 
	 \begin{figure}[htbp]
	 	\centering
	 	{\includegraphics[width=0.3\textwidth]{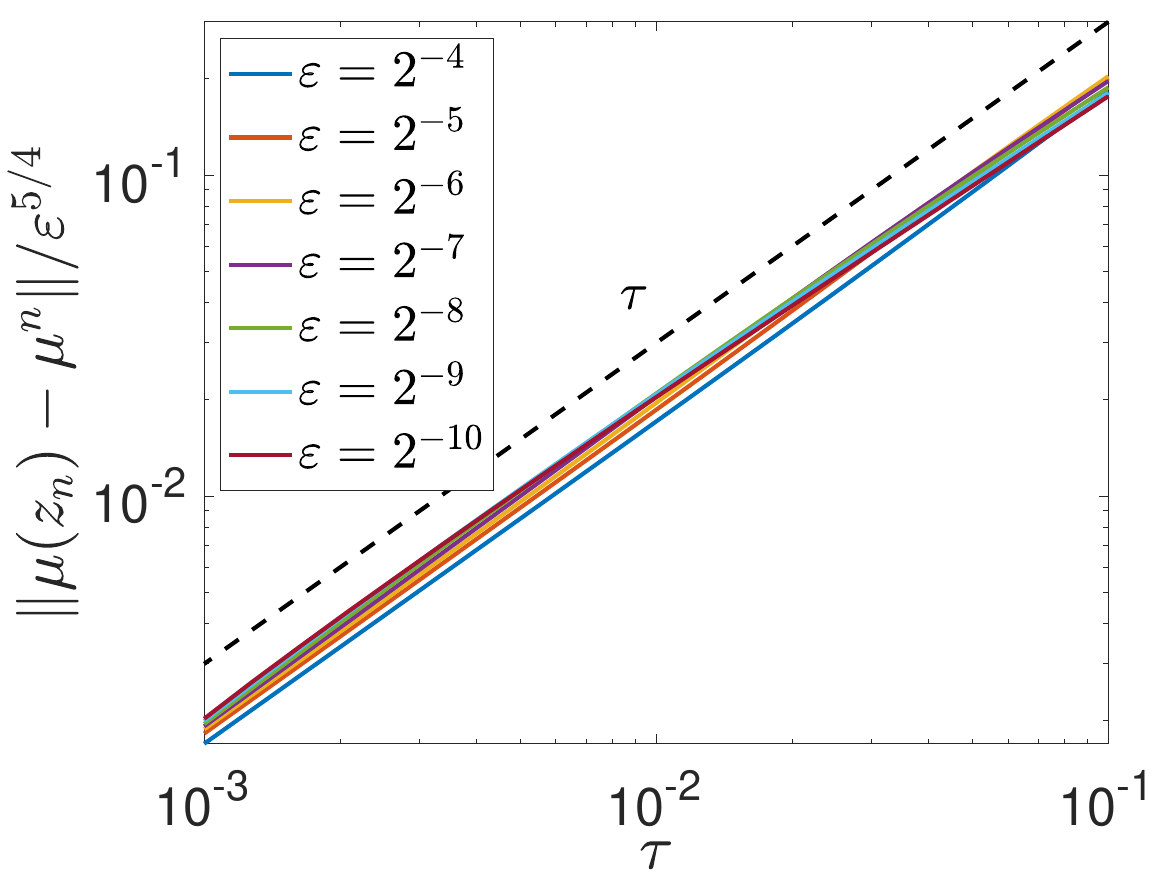}}\hspace{1em}
	 	{\includegraphics[width=0.3\textwidth]{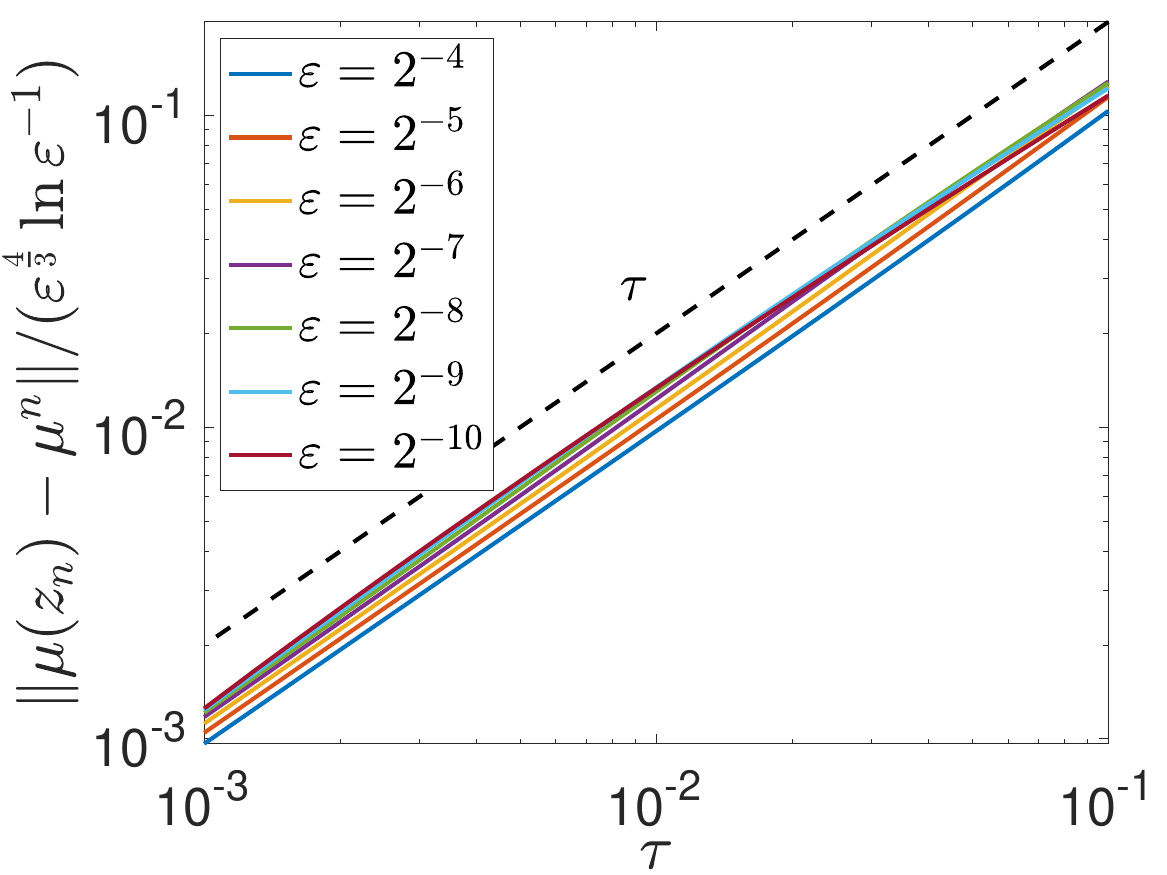}}\hspace{1em}
	 	{\includegraphics[width=0.3\textwidth]{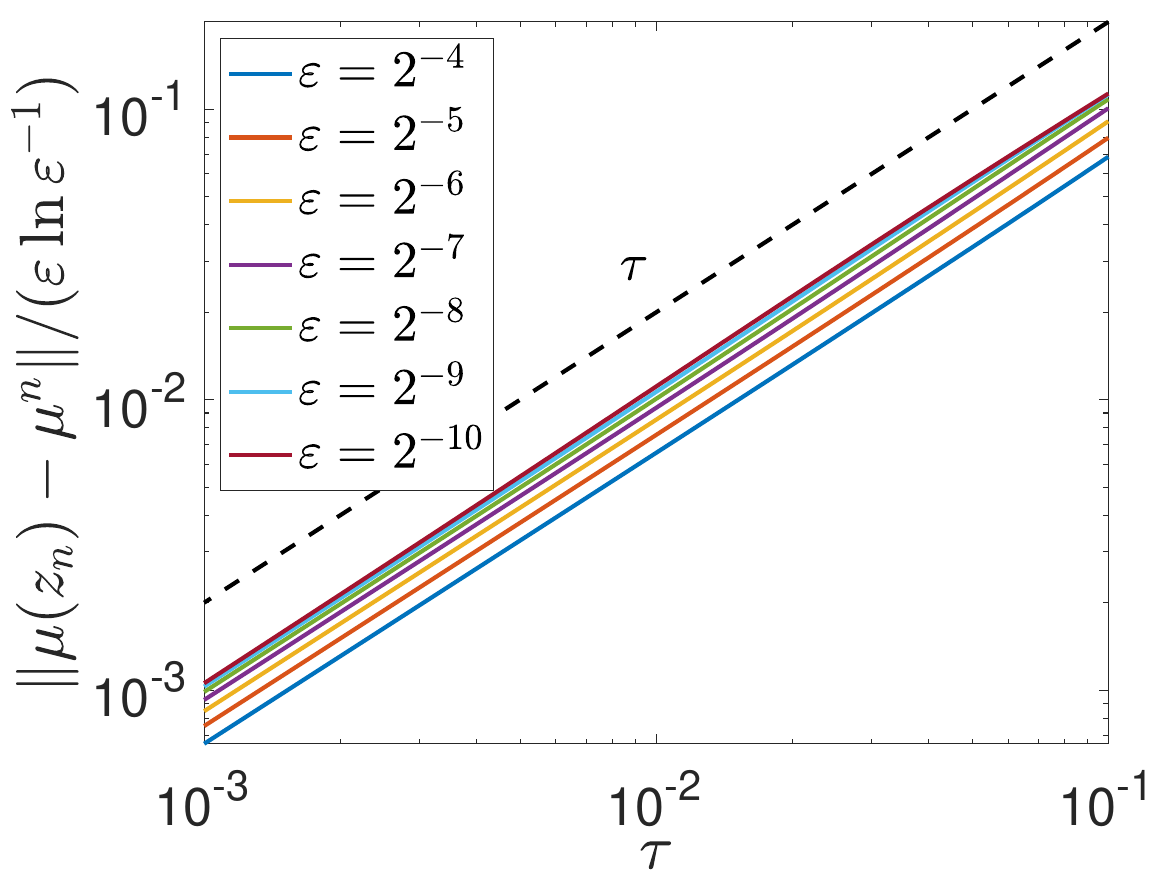}}
	 	\caption{Errors of the exponential integrator divided by $\min\{\vep^{1+\frac{\alpha}{2}}, \vep^{2-\alpha} \ln \vep^{-1}\}$ with $\alpha = \frac{1}{2}$ (left), $\alpha = \frac{2}{3}$ (middle), and $\alpha = 1$ (right) for \cref{eq:mu} with $D_\ka = \partial_x^2$}
	 	\label{fig:conv_dt_Schroedinger}
	 \end{figure}
	 
	We then consider the KdV case \cref{eq:model2} with $\alpha \in \{0, \frac{3}{4}, \frac{3}{2} \}$ chosen according to the same criteria as in the previous example. The computational domain is chosen as $(-32, 32)$ and the final step is set as $z_n=1$. Similarly, we compute numerical solutions for different $\vep$ ranging from $2^{-10}$ to $2^{-4}$ with the step size $\tau$ varying from $10^{-3}$ to $10^{-1}$. The numerical results are exhibited in \cref{fig:conv_dt_KdV}, where we plot the errors divided by $\min\{\vep^{1+\frac{2}{3}\alpha}, \vep^{2-\frac{2}{3} \alpha}\}$. The proved convergence orders in both $\vep$ and $\tau$ are observed, again suggesting the optimality of the error bounds in \cref{thm:main}.  
	 
	  \begin{figure}[htbp]
	 	\centering
	 	{\includegraphics[width=0.3\textwidth]{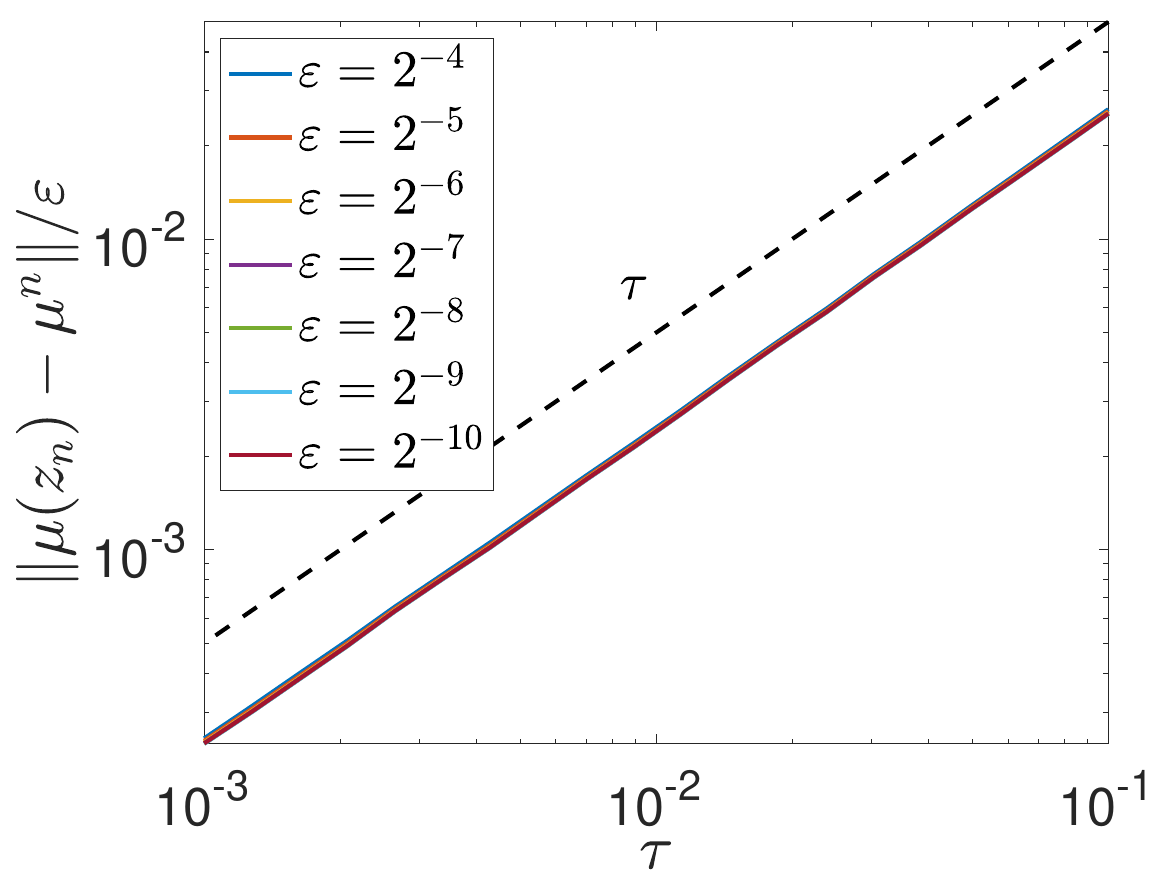}}\hspace{1em}
	 	{\includegraphics[width=0.3\textwidth]{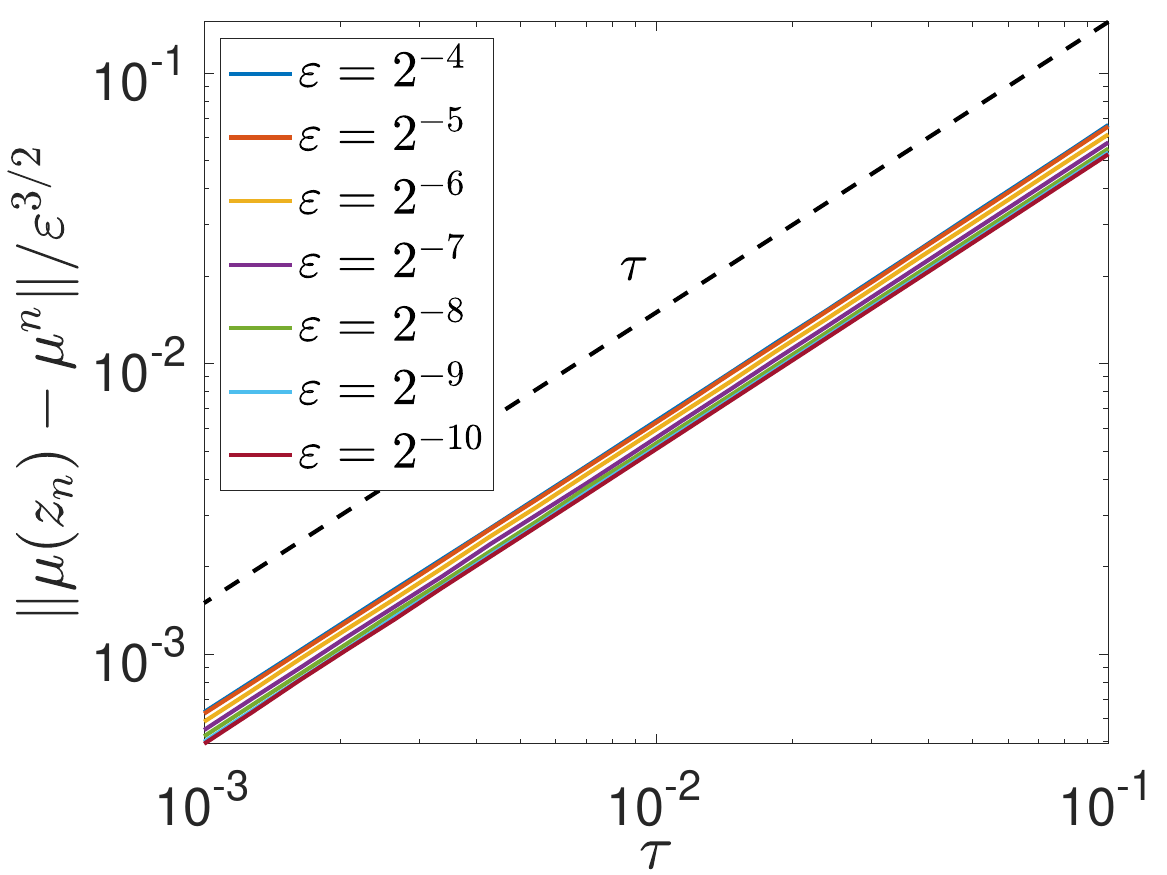}}\hspace{1em}
	 	{\includegraphics[width=0.3\textwidth]{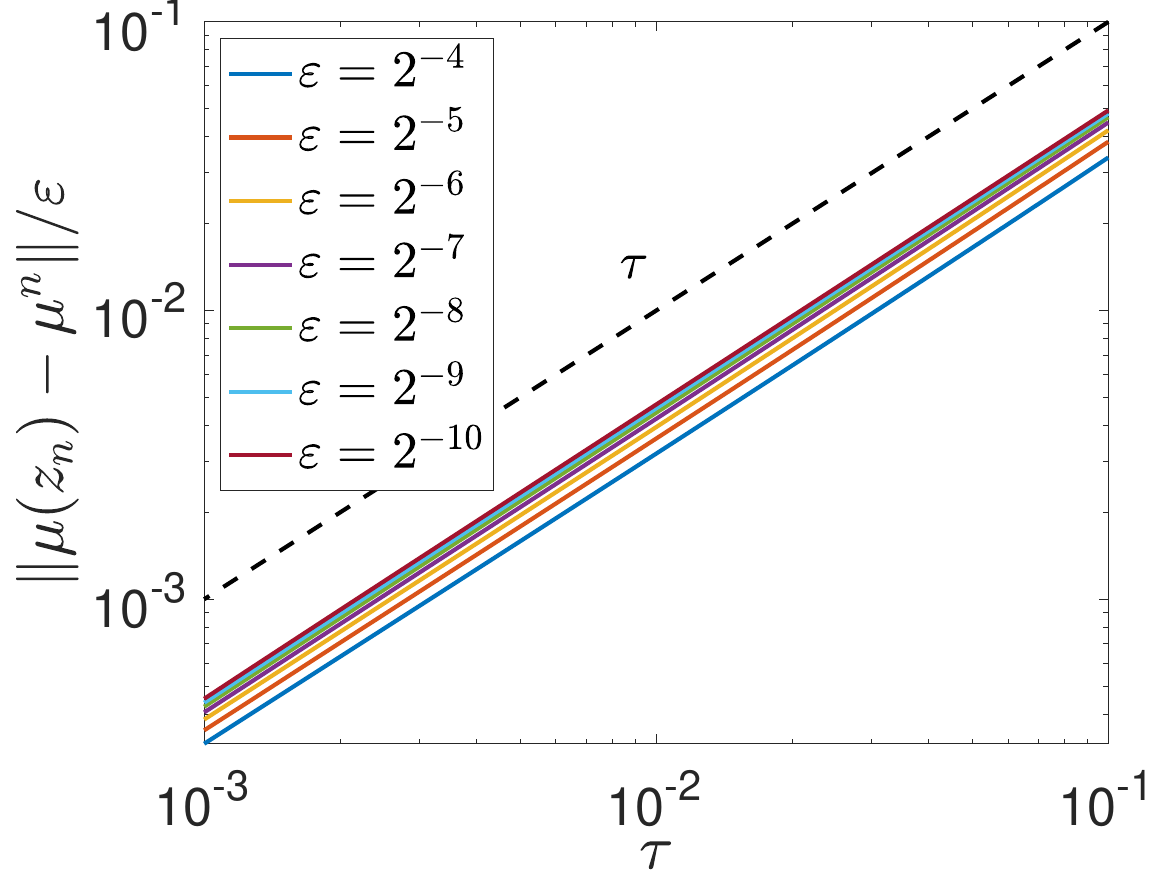}}
	 	\caption{Errors of the exponential integrator divided by $\min\{\vep^{1+\frac{2}{3}\alpha}, \vep^{2-\frac{2}{3}\alpha}\}$ with $\alpha = 0$ (left), $\alpha = \frac{3}{4}$ (middle), and $\alpha = \frac{3}{2}$ (right) for \cref{eq:mu} with $D_\ka = -i \partial_x^3$} 
	 	\label{fig:conv_dt_KdV}
	\end{figure}
	 
    \subsection{Comparisons with other numerical methods}\label{sec:comp}
	In this subsection, we compare the performance of the exponential integrator \cref{eq:EWI_mu} with other numerical methods for the model problem \cref{eq:model1} with $\alpha = 1$ and the same Gaussian initial datum used before. We consider the first-order Lie-Trotter splitting, the second-order Strang splitting, and a first-order low regularity integrator (LRI) proposed in \cite{bronsard2022,bronsard2023}, which are among the most popular numerical methods in the literature of dispersive equations. The precise form of these methods is given in \cref{append:A3} below. 
	 
	Before presenting the numerical results, we carry out an informal analysis of these methods to understand differences in their performances. Following the idea of the iterated Duhamel expansion used in the proof of \cref{thm:main} for the error estimate of the exponential integrator, we consider the twisted function $\psi$ and expand it to the second term in \cref{eq:duhamel_expansion_continuous} as  
	\begin{equation}
		\psi(z_n) = \psi_0(z_n) + \psi_1(z_n) + \cdots = \psi(0) + \int_0^{z_n} e^{-i s \vep \partial_{xx}} R_\vep e^{i s \vep \partial_{xx}} \psi(0) \rmd s + \cdots. 
	\end{equation}
	We stop the expansion here as clear differences already appear for the above numerical algorithms. First, all methods compute $\psi_0$ exactly. However, significant differences appear in the approximation to $\psi_1$ as follows: 
	\begin{enumerate}[(i)]
		\item The exponential integrator \cref{eq:EWI_mu} approximates $\psi_1(z_n)$ by (recalling \cref{sec:3.2})
		\begin{equation*}
			\psi_1(z_n) \approx \psi_1^n = \int_0^{z_n} e^{-i s \vep \partial_{xx}} R_\vep e^{i \fl{s} \vep \partial_{xx}} \psi(0) \rmd s,
		\end{equation*}
		and the approximation error is estimated in \cref{prop:est2_discrete} which is of order $O(\tau \vep^\frac{3}{2})$ without any $\vep$ dependent step size restriction. Compared to the error bound in \cref{thm:main} which is of order $O(\tau \vep)$, this implies the dominant approximation error oocurs, when $\alpha = 1$, in the approximation to higher-order integral terms in the Duhamel expansion. 
		
		\item The Lie-Trotter splitting \cref{eq:LT} approximates $\psi_1(z_n)$ by \cite{lubich2000}
		\begin{equation}\label{eq:LT-approx}
			\psi_1(z_n) \approx \psi_{1, \text{LT}}^n := \int_0^{z_n} e^{-i \fl{s} \vep \partial_{xx}} R_\vep e^{i \fl{s} \vep \partial_{xx}} \psi(0) \rmd s, 
		\end{equation}
		where the approximation error is given by
		\begin{equation*}
			\mathcal{E}_\text{LT} := \psi_1(z_n) - \psi_{1, \text{LT}}^n = \sum_{k=0}^{n-1} (\int_0^\tau e^{-i (z_k+s) \vep \partial_{xx}} R_\vep e^{i (z_k+s) \vep \partial_{xx}} \psi(0) \rmd s - \int_0^\tau e^{-i z_k \vep \partial_{xx}} R_\vep e^{i z_k \vep \partial_{xx}} \psi(0) \rmd s). 
		\end{equation*}
		Taking the Fourier transform of $\mathcal{E}_\text{LT}$, we obtain
		\begin{align*}
			\widehat{\mathcal{E}_\text{LT}}(\xi) 
			&= \int_{\frac{\xi_1}{\vep} + \xi_2 = \xi} \sum_{k=0}^{n-1} e^{i z_k \vep (\xi^2 - \xi_2^2)} \int_0^\tau (e^{i s \vep (\xi^2 - \xi_2^2)} - 1) \rmd s \widehat{R}(\xi_1) \widehat{\psi}(0, \xi_2) \rmd \xi_1 \notag \\
			&= \int_{\frac{\xi_1}{\vep} + \xi_2 = \xi} (i \vep (\xi^2 - \xi_2^2)) \sum_{k=0}^{n-1} e^{i z_k \vep (\xi^2 - \xi_2^2)} \int_0^\tau \int_0^s e^{i s_1 \vep (\xi^2 - \xi_2^2)} \rmd s_1 \rmd s \widehat{R}(\xi_1) \widehat{\psi}(0, \xi_2) \rmd \xi_2,  
		\end{align*}
		where we note that the local truncation error of the Lie-Trotter splitting is given by the commutator $-[i \vep \Delta, R_\vep]$ in the Fourier domain, i.e., 
		\begin{equation*}
			i \vep (\xi^2 - \xi_2^2) \widehat{R}(\xi_1) = \frac{i}{\vep} (\xi_1^2 + 2 \vep \xi_1 \xi_2) \widehat{R}(\xi_1) = O(\frac{1}{\vep}). 
		\end{equation*}
		Supposing $|\tau \vep (\xi^2 - \xi_2^2)| = |\frac{\tau}{\vep}(\xi_1^2 + 2 \vep \xi_1 \xi_2)| \leq \pi$, we have $|i \vep (\xi^2 - \xi_2^2) \sum_{k=0}^{n-1} e^{i z_k \vep (\xi^2 - \xi_2^2)}| \lesssim 1/\tau$, and thus
		\begin{equation*}
			\| \widehat{\mathcal{E}_\text{LT}}(\xi) \|_{L^1_\xi} \lesssim \tau. 
		\end{equation*}
		The above analysis suggests that if the step size restriction is satisfied (typically when $\tau \ll \vep$), the Lie-Trotter splitting approximates the first integral in the Duhamel's expansion with an error of order $O(\tau)$ uniformly in $\vep$. According to the numerical results to be presented later, the overall error of the Lie-Trotter splitting also matches this order. 
		
		\item The Strang splitting \cref{eq:Strang} approximates $\psi_1(z_n)$ in a similar way by replacing the left-rectangle rule with the mid-point rule in \cref{eq:LT-approx} \cite{lubich2000}, i.e. 
		\begin{equation*}
			\psi_1(z_n) \approx \psi_{1, \text{Strang}}^n := \int_0^{z_n} e^{-i \md{s} \vep \partial_{xx}} R_\vep e^{i \md{s} \vep \partial_{xx}} \psi(0) \rmd s, \quad \md{z}=z_k + \frac{\tau}{2}, \quad z_k \leq z < z_{k+1}. 
		\end{equation*}
		A similar analysis as above yields an error bound of $O(\frac{\tau^2}{\vep})$ in the approximation to the first term in the Duhamel expansion when the same step size restriction is satisfied, e.g., when $\tau \ll \vep$, which also matches the overall error of the Strang splitting observed in the numerical experiments. 
		
		\item The LRI \cref{eq:LRI} approximates $\psi_1(z_n)$ by \cite{LRI_1} 
		\begin{equation*}
			\psi_1(z_n) \approx \psi_{1, \text{LRI}}^n := \int_0^{z_n} e^{-i \fl{s} \vep \partial_{xx}} (e^{-i (s-\fl{s}) \vep \partial_{xx}} R_\vep) e^{i \fl{s} \vep \partial_{xx}} \psi(0) \rmd s, 
		\end{equation*}
		where the approximation error is given by
		\begin{equation*}
			\mathcal{E}_\text{LRI} := \psi_1(z_n) - \psi_{1, \text{LRI}}^n = \sum_{k=0}^{n-1} e^{-i z_k \vep \partial_{xx}} (\int_0^\tau e^{-i s \vep \partial_{xx}} R_\vep e^{i (z_k+s) \vep \partial_{xx}} \psi(0) \rmd s - \int_0^\tau (e^{-i s \vep \partial_{xx}} R_\vep) e^{i z_k \vep \partial_{xx}} \psi(0) \rmd s). 
		\end{equation*}
		Taking the Fourier transform of $\mathcal{E}_\text{LRI}$, we obtain
		\begin{align*}
			\widehat{\mathcal{E}_\text{LRI}}(\xi) 
			&= \int_{\frac{\xi_1}{\vep} + \xi_2 = \xi} \sum_{k=0}^{n-1} e^{i z_k \vep (\xi^2 - \xi_2^2)} \int_0^\tau (e^{i s \vep (\xi^2 - \xi_2^2)} - e^{is \frac{\xi_1^2}{\vep}}) \rmd s \widehat{R}(\xi_1) \widehat{\psi}(0, \xi_2) \rmd \xi_1 \notag \\
			&= \int_{\frac{\xi_1}{\vep} + \xi_2 = \xi} 2i \xi_1 \xi_2 \sum_{k=0}^{n-1} e^{i z_k \vep (\xi^2 - \xi_2^2)} \int_0^\tau e^{i s \frac{\xi_1^2}{\vep}} \int_0^s e^{2is_1 \xi_1 \xi_2} \rmd s_1 \rmd s \widehat{R}(\xi_1) \widehat{\psi}(0, \xi_2) \rmd \xi_2,  
		\end{align*}
        where the local truncation error of the LRI is given by $2 i \xi_1 \xi_2 \widehat{R}(\xi_1) \widehat{\psi}(0, \xi_2) = O(1)$.
        
        Assuming again $|\tau \vep (\xi^2 - \xi_2^2)| \leq \pi$, and considering the resonant and non-resonant decomposition: $|\xi_1 + 2\vep\xi_2| \leq \delta $ and $|\xi_1 + 2\vep\xi_2| > \delta$ with $\delta = \vep$, we obtain
		\begin{align*}
			&\| \widehat{\mathcal{E}_\text{LRI}}(\xi) \|_{L^1_\xi} \notag \\
			&\lesssim \tau \int_{\R} |\xi_2 \widehat{\psi}(0, \xi_2)| \int_{\R} 1_{|\xi_1 + 2 \vep \xi_2|<\delta} |\xi_1| |\widehat{R}(\xi_1)| \rmd \xi_1\rmd \xi_2 +  \tau \vep \int_{\R} |\xi_2 \widehat{\psi}(0, \xi_2)| \int_{\R} \frac{1_{|\xi_1 + 2 \vep \xi_2|>\delta}}{|\xi_1 + 2 \vep \xi_2|} |\widehat{R}(\xi_1)| \rmd \xi_1 \rmd \xi_2 \notag \\
			&\lesssim \tau \|\widehat{R}\|_{L^\infty} \int_{\R} |\xi_2 \widehat{\psi}(0, \xi_2)| \int_{\R} 1_{|\xi_1 + 2 \vep \xi_2|<\delta} |\xi_1+2\vep\xi_2| \rmd \xi_1 \rmd \xi_2  \notag \\
			&\quad+ \tau \vep \|\widehat{R}\|_{L^\infty} \int_{\R} |\xi_2^2 \widehat{\psi}(0, \xi_2)| \int_{\R} 1_{|\xi_1 + 2 \vep \xi_2|<\delta} \rmd \xi_1 \rmd \xi_2 + \tau \vep \int_{\R} |\xi_2 \widehat{\psi}(0, \xi_2)| \int_{\R} \frac{1_{|\xi_1 + 2 \vep \xi_2|>\delta}}{|\xi_1 + 2 \vep \xi_2|} |\widehat{R}(\xi_1)| \rmd \xi_1 \rmd \xi_2 \notag \\
			&\lesssim \tau \delta^2 + \tau \vep \delta + \tau \vep |\ln \delta| \lesssim \tau \vep |\ln \vep|. 
		\end{align*}
		The above computation indicates that if the step size restriction is satisfied (e.g., when $\tau \ll \vep$), the LRI approximates the first integral in the Duhamel expansion with an error of order $O(\tau \vep |\ln \vep|)$. Hence, at the level of the first-order term, the LRI is better than splitting methods but worse than the exponential integrator. Similar to the exponential integrator, the numerical results suggest that the overall error of the LRI when $\alpha = 1$ is dominated by the approximation errors of the higher-order terms in the Duhamel's expansion; but again, the performance of the LRI is better than splitting methods while not as good as the exponential integrator. As the numerical results also confirm that the exponential integrator consistently outperforms the LRI, we do not present a detailed analysis of the approximations to higher-order terms for the LRI here. %%%\gb{But overall, is the scheme as good as exponential integrator or not?} 
	\end{enumerate}

    We next apply the Lie-Trotter splitting \cref{eq:LT}, Strang splitting \cref{eq:Strang}, and LRI \cref{eq:LRI} to the model problem \cref{eq:model1} with $\alpha = 1$, and compute the convergence plots of these three methods for various $\vep$ from $2^{-10}$ to $2^{-4}$. The numerical results are presented in \cref{fig:comp_method}, where we divide the errors of the LRI by $\vep^{1/2}$. 
    
    As shown in \cref{fig:comp_method}, all three methods exhibit a clear change of behavior as the time step $\tau$ crosses a threshold that depends on $\vep$. For both splitting methods, the transition occurs around $\tau \sim \vep$, while the transition regime does not seem clear for the LRI. When $\tau \lesssim \vep$, both splitting methods recover the desired convergence orders in $\tau$ with an error of order $O(\tau)$ for the Lie-Trotter splitting, and an error of order $O(\tau^2/\vep)$ for the Strang splitting. This observation conforms well with our informal analysis. When $\tau \gtrsim \vep$, the two splitting methods behave much worse: the convergence orders in $\tau$ are deteriorated and the dependence on $\vep$ is much worse. Such bad performance of splitting methods in the regime $\tau \gtrsim \vep$ can also be seen from the numerical solutions in \cref{fig:soln_diff_method} displaying severe spurious oscillations. Hence, in practice, one should avoid such large step sizes when using the time-splitting methods for \cref{eq:mu} or for the fourth moment equation \cref{eq:4thmoment}. 
    
    Although LRI does not perform as well as the exponential integrator \cref{eq:EWI_mu}, it is better than both splitting methods as it is more robust to the step sizes: the errors improve as $\vep \rightarrow 0$ for both large step sizes and small step sizes with a different rate. For small step sizes, the convergence order is observed to be $O(\tau\vep^\frac{1}{2})$. %\gb{OK. We will have to say this upfront somewhere. This is numerical, not theoretical, right?} 
    For large step sizes, the convergence rate in $\vep$ increases and the rate in $\tau$ decreases; however, the exact order is not yet clear from the numerical results. 

    Finally, we plot in \cref{fig:comp_method_all} the errors for different $\tau$ of all the four methods when $\vep = 2^{-10}$. We see that the exponential integrator substantially outperforms all three other methods, and the LRI outperforms the time-splitting methods at least for large step sizes. 
	
	\begin{figure}[htbp]
		\centering
		{\includegraphics[width=0.3\textwidth]{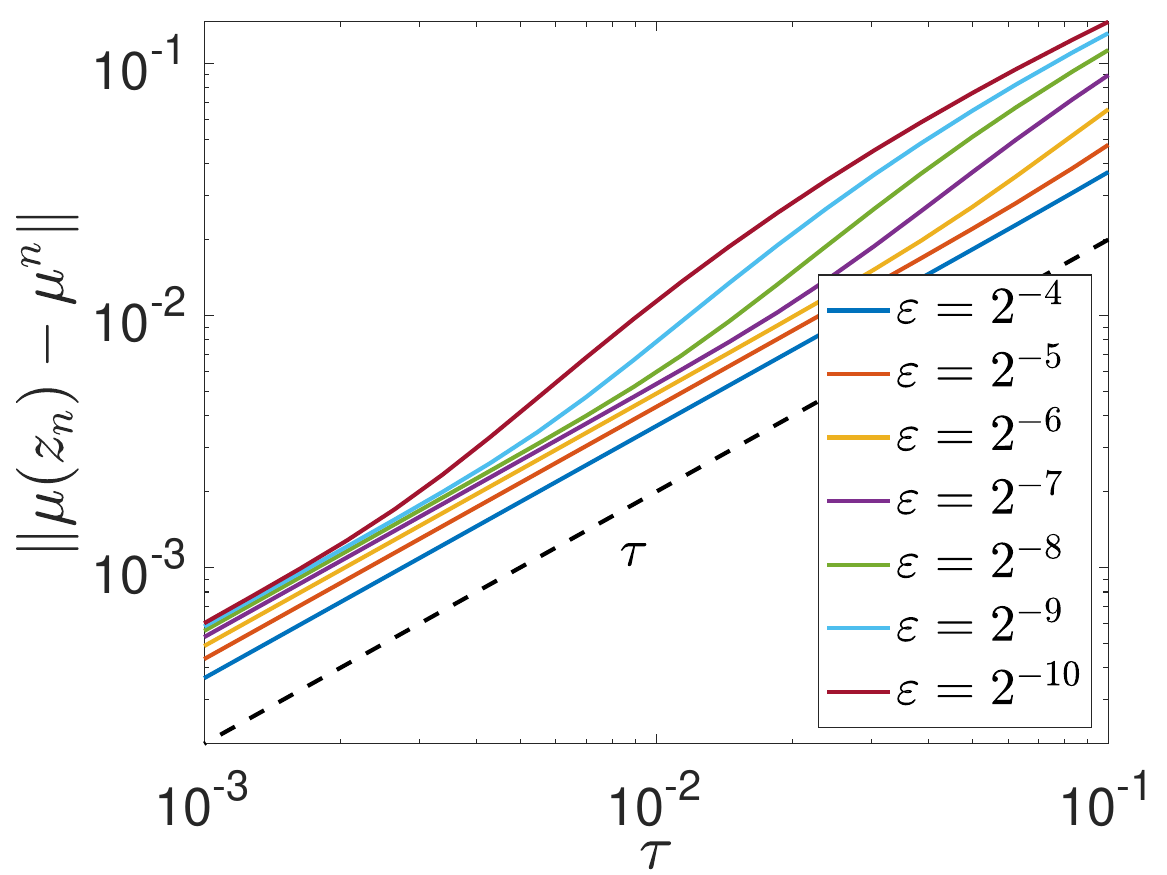}}\hspace{1em}
		{\includegraphics[width=0.3\textwidth]{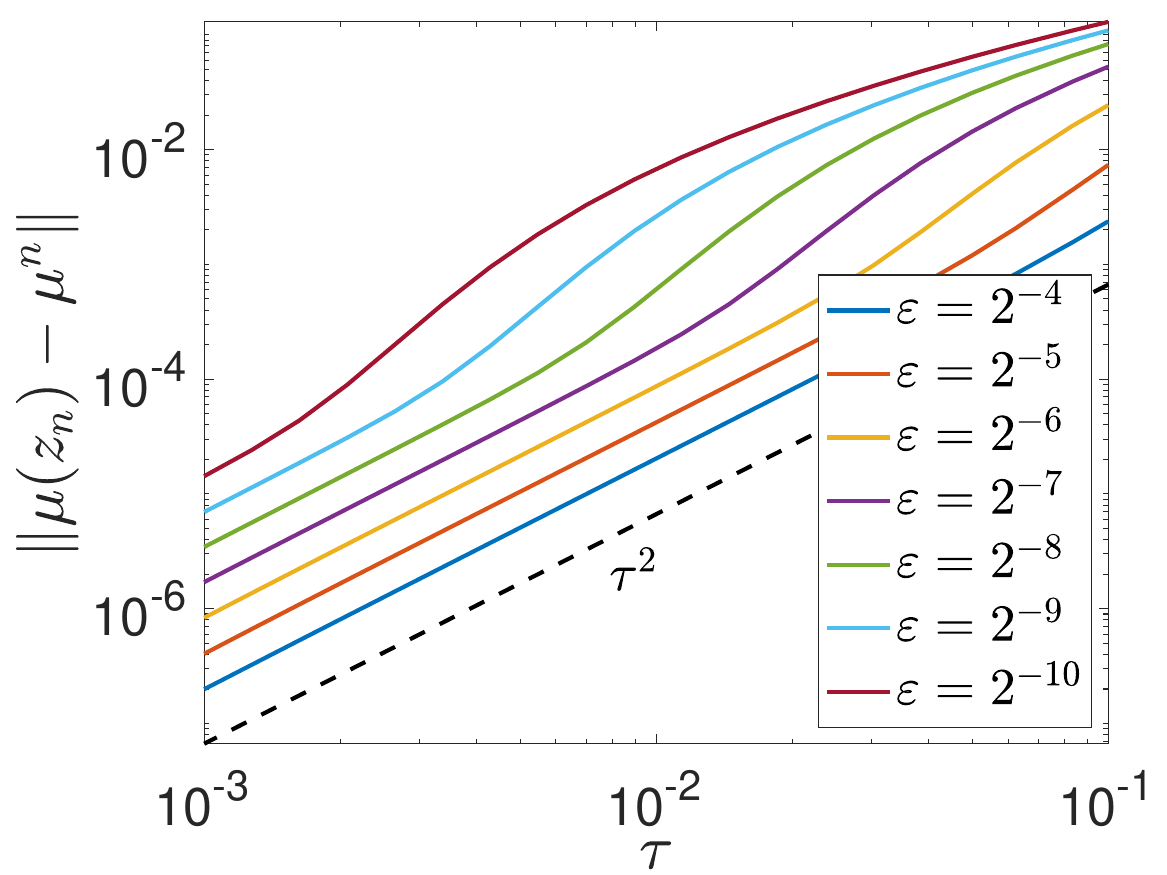}}\hspace{1em}
		{\includegraphics[width=0.3\textwidth]{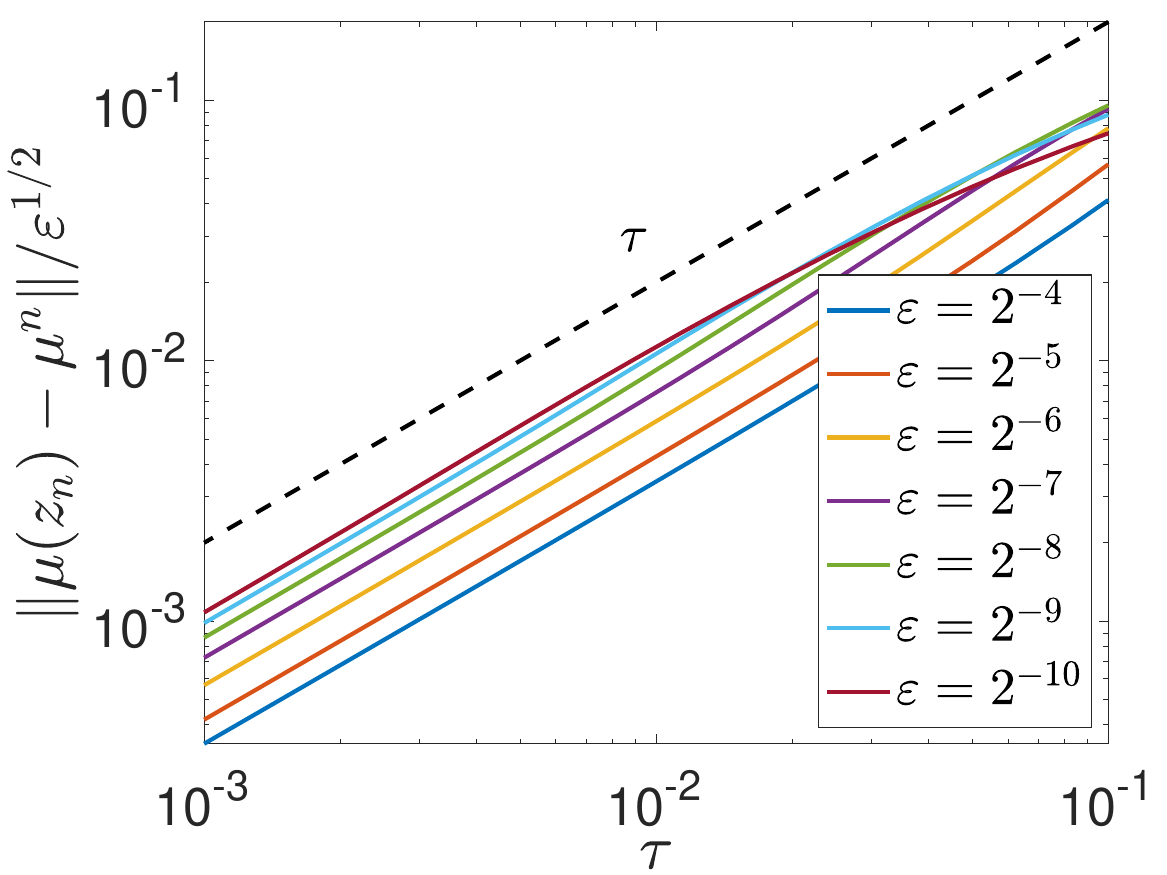}}
		\caption{Errors of the Lie-Trotter splitting \cref{eq:LT} (left), the Strang splitting \cref{eq:Strang} (middle), and error divided by $\vep^{1/2}$ of the LRI \cref{eq:LRI} (right) applied to \cref{eq:model1} with $\alpha = 1$} 
		\label{fig:comp_method}
	\end{figure}
	
	\begin{figure}[htbp]
		\centering
		\includegraphics[width=0.5\textwidth]{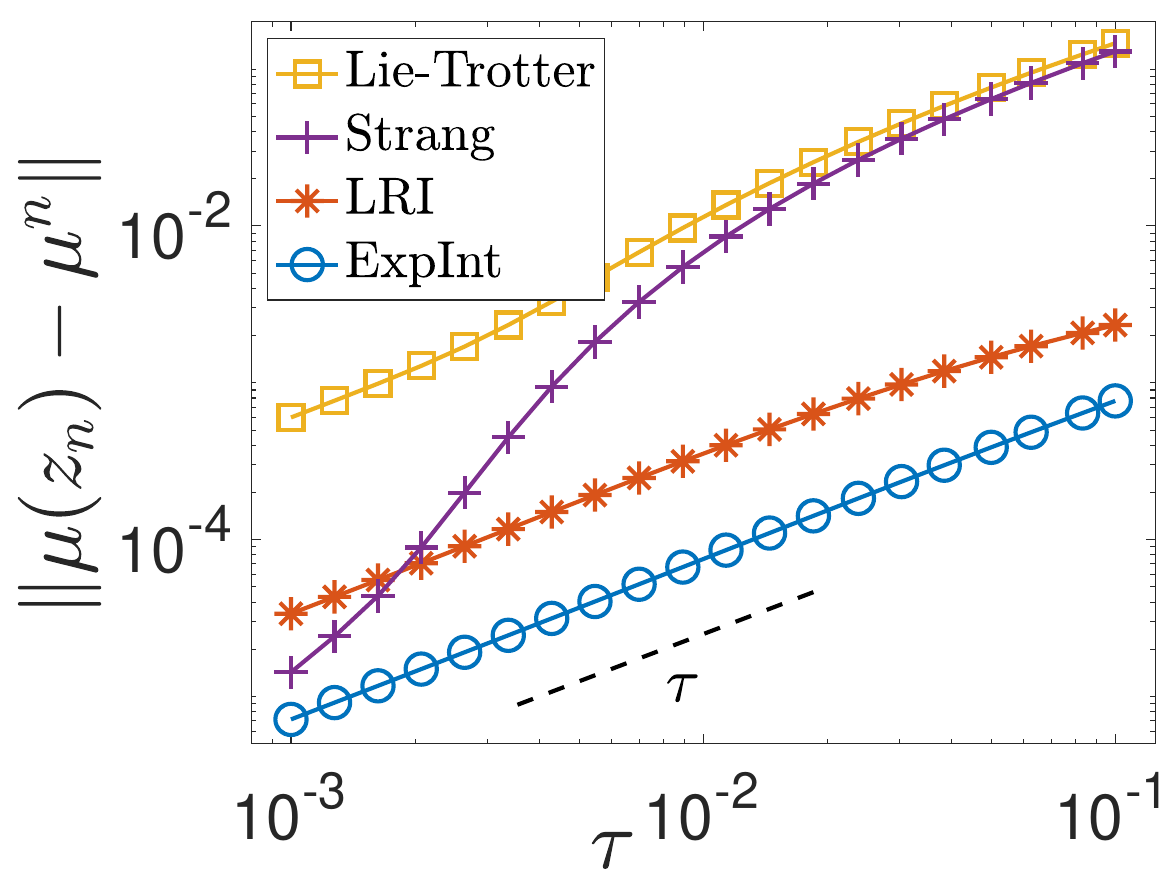}
		\caption{Comparison of the errors of the exponential integrator \cref{eq:EWI_mu}, the LRI \cref{eq:LRI}, the Lie-Trotter splitting \cref{eq:LT}, and the Strang splitting \cref{eq:Strang} applied to \cref{eq:model1} with $\alpha = 1$ and $\vep = 2^{-10}$} 
		\label{fig:comp_method_all}
	\end{figure}
    
	\section{Conclusion}\label{sec:conclusion}
	In this paper, we studied a linear dispersive equation with a highly concentrated potential, where the strength of the dispersion is $O(\vep^\alpha)$ and the wavelength (extension) of the potential is $O(\vep)$. The equation thus features an interesting competition between the dispersion and the potential induced concentration.
    %, and is of particular interest. 
    We first established the regularity results of the exact solution with explicit dependence on $\vep$. Then we proposed an exponential integrator and rigorously established its error estimate. The error estimate hold without requiring any coupling condition between the step size and the small parameter $\vep$. Extensive numerical results are reported to confirm the excellent performance of the exponential integrator and the sharpness of our estimates.

	\appendix
	\section{Equations for the statistical moments}
	\subsection{The fourth moment of the It\^o-Schr\"odinger equation in the scintillation regime}\label{append:A1}
	The It\^o-Schr\"odinger equation models laser beam propagation through random media in the paraxial white-noise regime, which is obtained from the paraxial approximation to the Helmholtz equation \cite{paraxial1} in the white-noise limit \cite{itoschro}. We consider the scintillation scaling. It is characterized by a wide initial beam, a long propagation distance, and weak medium fluctuations. The It\^o-Schr\"odinger equation in the scintillation scaling \cite{gar2016,bal2025} reads for $u = u(z, x)$ a complex-valued function as
	\begin{equation}\label{eq:ito-schrodinger}
		\rmd u = \frac{i}{2 \vep} \Delta_x u \rmd z - \frac{R(0)}{8} u \rmd z + \frac{i}{2} u \rmd B, \qquad z > 0, 
	\end{equation}
	with initial data $u(0, x) = u_0(\vep^\beta x)$ for some $\beta \geq 1$ and $u_0$ decaying rapidly at infinity, where $0 < \vep \ll 1$, $z \geq 0$ is of $O(1)$, and $B$ is a mean-zero Gaussian process characterized by the covariance function $\mathbb{E}[B(z, x) B(z', x')] = \min\{z, z'\}R(x - x')$ with $R$ a real-valued function. We define the $(p, q)$-moment of the solution $u$ as follows: let $X = (x_1, \cdots, x_p)^T$ and $Y = (y_1, \cdots, y_q)^T$, 
	\begin{equation}\label{eq:moment}
		\mu_{p, q}(z, X, Y) := \mathbb{E}[u(z, x_1) \cdots u(z, x_p) \overline{u(z, y_1)} \cdots \overline{u(z, y_q)}]. 
	\end{equation}
	By It\^o calculus \cite{gar2016,bal2025}, the $(p, q)$-moments satisfy a closed form deterministic Schr\"odinger-type equation. Two moments are of particular interest: the $(1, 1)$-moment (i.e. second moment) $\mu_{1,1}(z, x_1, y_1)$ which satisfies
	\begin{equation}\label{eq:2ndmoment}
		\left\{
		\begin{aligned}
			&\partial_z \mu_{1, 1} = \frac{i}{2\vep} (\Delta_{x_1} - \Delta_{y_1}) \mu_{1, 1} + \frac{1}{4} (R(x_1 - y_1) - R(0)) \mu_{1, 1}, \\
			&\mu_{1, 1}(0, x_1, y_1) = u_0(\vep^\beta x_1)\overline{u_0(\vep^\beta y_1)}, 
		\end{aligned}
		\right.
	\end{equation}
	and the $(2, 2)$-moment (i.e. fourth moment) $\mu_{2, 2} = \mu_{2, 2}(z, X, Y)$ which satisfies 
	\begin{equation}\label{eq:4thmoment}
		\left\{
		\begin{aligned}
			&\partial_z \mu_{2, 2} = \frac{i}{2\vep} (\Delta_{x_1} + \Delta_{x_2} - \Delta_{y_1} - \Delta_{y_2}) \mu_{2, 2} + \frac{1}{4} U(X, Y) \mu_{2, 2}, \\
			&\mu_{2, 2}(0, X, Y) = u_0(\vep^\beta x_1)u_0(\vep^\beta x_2)\overline{u_0(\vep^\beta y_1)u_0(\vep^\beta y_2)}, 
		\end{aligned}
		\right.
	\end{equation}
	where the potential $U$ is given by
	\begin{equation*}
		U(X, Y) = R(x_1 - y_1) + R(x_1 - y_2) + R(x_2 - y_1) + R(x_2 - y_2) - R(x_1 - x_2) - R(y_1 - y_2) - 2 R(0).  
	\end{equation*}
	As shown in \cite{bal2025}, after a change of variable and taking a (partial) Fourier transform, the second moment equation \cref{eq:2ndmoment} reduces to a transport equation and can be solved analytically. However, such techniques do not apply to the fourth moment equation \cref{eq:4thmoment}, and one has to use numerical methods to approximate it. 
	
	Solving \cref{eq:4thmoment} directly is difficult especially when $\vep \ll 1$. However, from \cref{eq:4thmoment}, we can identify two different inner structures: (i) when looking at variables $x_j$ and $y_l$ for some $j, l \in \{1, 2\}$, and dropping terms involving other variables, we obtain, by replacing $\mu_{2, 2}(z, X, Y)$ with $\mu(z, x_j, y_l)$, 
	\begin{equation}\label{eq:4thsub1}
		\partial_z \mu = \frac{i}{2\vep} (\Delta_{x_j} - \Delta_{y_l}) \mu + \frac{1}{4} R(x_j - y_l) \mu, 
	\end{equation}
	and (ii) when looking at $x_1$ and $x_2$ (or $y_1$ and $y_2$), and similarly dropping terms involving other variables, we obtain, with $\mu = \mu(z, x_1, x_2)$ replacing $\mu_{2, 2}$, 
	\begin{equation}\label{eq:4thsub2}
		\partial_z \mu = \frac{i}{2\vep} (\Delta_{x_1} + \Delta_{x_2}) \mu - \frac{1}{4} R(x_1 - x_2) \mu. 
	\end{equation}
	It turns out that \cref{eq:4thsub1,eq:4thsub2} are significantly different: \cref{eq:4thsub1} is (almost) the same as the $(1, 1)$-moment \cref{eq:2ndmoment}, and thus can be reduced to a transport equation and solved analytically, while \cref{eq:4thsub2} is indeed a Schr\"odinger-type equation. Following the procedure detailed in the next subsection, we can reduce \cref{eq:4thsub2} to \cref{eq:schrodinger}--the main object of the current paper. We remark here that the same structure of \cref{eq:4thsub2} also appears in the third-order moment or higher-order moments.
    
	\subsection{Application of \cref{eq:mu}}\label{append:A2}
	In the following, we show that \cref{eq:mu} can be derived from the moment equations of the It\^o-Schr\"odinger equation in the scintillation scaling, and their higher-order dispersive counterparts. 
	
	Let $\nu = \nu(z, x, y)$ be the solution to
	\begin{equation}\label{eq:nu_xy}
		\left\{
		\begin{aligned}
			&\partial_z \nu = i \frac{(-i\partial_{x})^\ka \pm (-i\partial_{y})^\ka}{\eta} \nu + R(x - y) \nu, && z \geq 0, \quad x, y \in \R, \\
			&\nu(0, x, y) = u_1(\eta^\beta x) u_2(\eta^\beta y), && x, y \in \R,  
		\end{aligned}
		\right.
	\end{equation}
	where $0 < \eta \ll 1$ is a small parameter. 
	%	When $\ka = 2$, \cref{eq:nu_xy} is the equation for the $(2, 0)$-moment of the Ito-Schr\"odinger equation in the long distance scaling. 
	First, we rescale the equation in space by the change of variables 
	\begin{equation*}
		\eta^\beta x \rightarrow \tilde x, \quad  \eta^\beta y \rightarrow \tilde y, \quad \nu(x, y) \rightarrow \tilde \nu(\tilde x, \tilde y), 
	\end{equation*}
	and obtain, with $\vep := \eta^\beta$ and $\alpha := \ka - 1/\beta$, 
	\begin{equation*}
		\left\{
		\begin{aligned}
			&\partial_z \tilde \nu = i \vep^\alpha ((- i\partial_{\tilde x})^\ka \pm (-i\partial_{\tilde y})^\ka) \tilde \nu + R(\frac{\tilde x - \tilde y}{\vep}) \tilde \nu, && z \geq 0, \quad \tilde x, \tilde y \in \R, \\
			&\tilde \nu(0, \tilde x, \tilde y) = u_1(\tilde x) u_2(\tilde y). 
		\end{aligned}
		\right. 
	\end{equation*}
	Furthermore, we rotate the frame with the change of variables
	\begin{equation*}
		r_1 = \frac{\tilde x + \tilde y}{2}, \quad r_2 = \tilde x - \tilde y,  \quad \check \nu(z, r_1, r_2) = \tilde \nu(z, \tilde x, \tilde y), 
	\end{equation*}
	and we have
	\begin{equation}
		\left\{
		\begin{aligned}
			&\partial_z \check \nu = i \vep^\alpha D \check \nu + R(\frac{r_2}{\vep}) \check \nu, && z \geq 0, \quad r_1, r_2 \in \R, \\
			&\check  \nu(0, r_1, r_2) = u_1(r_1 + \frac{1}{2} r_2) u_2(r_1 - \frac{1}{2}r_2),
		\end{aligned}
		\right.
	\end{equation}
	where
	\begin{align*}
		D 
		&:= (- i\partial_{\tilde x})^\ka + (-i\partial_{\tilde y})^\ka = (\frac{1}{2} (- i \partial_{r_1}) + (- i \partial_{r_2}))^\ka \pm (\frac{1}{2} (- i \partial_{r_1}) - (- i \partial_{r_2}))^\ka \notag \\
		&= \sum_{j = 0}^\ka \binom{\ka}{j} \frac{1}{2^{\ka - j}}(- i \partial_{r_1})^{\ka - j} (- i \partial_{r_2})^{j} (1 \pm (-1)^{j}).  
	\end{align*}
	Finally, taking the Fourier transform of $\check v$ in $r_1$, with the dual variable denoted by $\lambda$ and the Fourier transform denoted by $U_{\lambda}(z, r_2) := \widehat{\check \nu}(z, \lambda, r_2)$, we obtain
	\begin{equation}\label{eq:transformed}
		\left\{
		\begin{aligned}
			&\partial_z U_{\lambda}(z, r_2) = i \vep^\alpha \widehat{D} U_{\lambda}(z, r_2) + R(\frac{r_2}{\vep}) U_{\lambda}(z, r_2), && z \geq 0, \quad r_2 \in \R, \\
			&U_{\lambda}(0, r_2) = \mathcal{F}_{r_1}(u_1(r_1 + \frac{1}{2} r_2) u_2(r_1 - \frac{1}{2}r_2))(\lambda, r_2),  
		\end{aligned}
		\right.
	\end{equation}
	where
	\begin{equation}\label{eq:D_hat}
		\widehat{D} = \sum_{j = 0}^\ka \binom{\ka}{j} \frac{\lambda^{\ka - j}}{2^{\ka - j}} (- i \partial_{r_2})^{j} (1 \pm (-1)^{j}). 
	\end{equation}
	Then the equation \cref{eq:transformed} is decoupled in $\lambda$ and can be solved independently for each $\lambda \in \R$. When $+$ is taken in \cref{eq:D_hat}, we have
	\begin{equation*}
		\widehat{D} = \sum_{0 \leq 2j \leq \ka} \binom{\ka}{2j} \frac{\lambda^{\ka - 2j}}{2^{\ka - 2j -1}} (- i \partial_{r_2})^{2j} = \mathrm{sgn}(\lambda^{\ka})  \sum_{0 \leq 2j \leq \ka} c_{2j} (- i \partial_{r_2})^{2j}, 
	\end{equation*}
	and when $-$ is taken in \cref{eq:D_hat}, we have
	\begin{equation*}
		\widehat{D} = \sum_{0 \leq 2j+1 \leq \ka} \binom{\ka}{2j+1} \frac{\lambda^{\ka - 2j - 1}}{2^{\ka - 2j - 2}} (- i \partial_{r_2})^{2j+1} = \mathrm{sgn}(\lambda^{\ka - 1}) \sum_{0 \leq 2j+1 \leq \ka} c_{2j+1} (- i \partial_{r_2})^{2j+1},  
	\end{equation*}
	with $c_j > 0$ for all $j \geq 0$. Note that when $+$ is taken, $\widehat{D}$ would include the zeroth-order term $\pm c_0$, which can be dropped by a simple change of variable $U_\lambda \rightarrow e^{\pm i c_0 z} U_\lambda$. Hence, for each fixed $\lambda \in \R$, the equation takes the form \cref{eq:mu}. 

    \section{Other numerical schemes}\label{append:A3}
    For the convenience of the reader, we give explicitly the formulation of other numerical methods applied to \cref{eq:mu}. We denote by $\mu^n$ the numerical approximation to $\mu(z_n)$ with $z_n = n\tau$ and $\tau > 0$ being the step size, and the first step is given by $\mu^0 = \mu_0$. The Lie-Trotter splitting method reads
	\begin{equation}\label{eq:LT}
        \mu^{n+1} = e^{i \tau \vep^\alpha D_\ka} e^{\tau R_\vep} \mu^n, \quad n \geq 0.  
	\end{equation}
	The Strang splitting method reads
	\begin{equation}\label{eq:Strang}
        \begin{aligned}
	 		\mu^{(1)} &= e^{i \frac{\tau}{2} \vep^\alpha D_\ka} \mu^n, \\
	 		\mu^{(2)} &= e^{\tau R_\vep} \mu^{(1)}, \\
	 		\mu^{n+1} &= e^{i \frac{\tau}{2} \vep^\alpha D_\ka} \mu^{(2)}, 
	 	\end{aligned}
	 	\qquad n \geq 0. 
    \end{equation}
    The low regularity integrator reads
    \begin{equation}\label{eq:LRI}
	 	\mu^{n+1} = e^{i \tau \vep^\alpha D_\ka} \mu^n + \tau (\vphi_1(-i\tau \vep^\alpha D_\ka) R_\vep) \mu^n, \qquad n \geq 0,   
	\end{equation}
	where $\vphi_1(z) = (e^{z}-1)/z$ for $z \in \C$, as defined in \cref{eq:EWI_mu}, and, by a change of variable, we also have
	\begin{equation*}
	 	(\vphi_1(-i\tau \vep^\alpha D_\ka) R_\vep)(x) = (\vphi_1(-i\tau \vep^{\alpha-\ka} \tilde D_\ka) R)(\frac{x}{\vep}), \quad \tilde D_\ka = -\sum_{0 \leq 2j < \ka} d_{\ka - 2j} \vep^{2j} (-i\partial_x)^{\ka-2j}.
	\end{equation*} 
    
	\section{Technical estimates}
    \subsection{Proof of \cref{cor:phase}}\label{appen:phase}
    First, we have the following algebraic property of the oscillatory phase $\Phi$ \cref{eq:phase}. 
	\begin{lemma}\label{lem:phase}
		For $\ka \geq 2$, we have
		\begin{equation*}
			\Phi(\xi_1, \xi_2) = \frac{1}{\vep^{\kappa - \alpha}} \sum_{0 \leq 2j < \ka} \vep^{2j} \tilde d_{\ka-2j} Q_{\kappa - 2 j}(\xi_1^2, (\xi_1 + 2 \vep \xi_2)^2) \times \left\{
			\begin{aligned}
				&\xi_1 (\xi_1 + 2 \vep \xi_2), &&\kappa \text{ even}, \\
				&\xi_1, &&\kappa \text{ odd}, \\     
			\end{aligned}
			\right., \quad \tilde d_{\ka-2j}:=\frac{d_{\ka - 2j}}{2^{\ka - 2j -1}}, 
		\end{equation*}
		where $Q_r(x, y)$ is a homogeneous polynomial with positive coefficients which is given by
		\begin{equation}\label{eq:Q}
			Q_r(x, y) = \left\{
			\begin{aligned}
				&\sum_{j=0}^{p-1} \binom{r}{2j+1} x^j y^{p-1-j}, && r = 2p, \\
				&\sum_{j=0}^{p} \binom{r}{2j+1} x^j y^{p-j}, && r = 2p+1. 
			\end{aligned}
			\right.
		\end{equation}
	\end{lemma}
	
	\begin{proof}
		From \cref{eq:phase,eq:P_def}, we have
		\begin{equation}\label{Phi}
			\Phi(\xi_1, \xi_2)  = \frac{1}{\vep^{\ka - \alpha}} (\vep^\ka P(\frac{\xi_1}{\vep} + \xi_2) - \vep^\ka P(\xi_2)) = \frac{1}{\vep^{\ka - \alpha}} \sum_{0 \leq 2j < \ka} \vep^{2j} d_{\ka - 2j} ((\xi_1 + \vep \xi_2)^{\ka - 2j} - (\vep\xi_2)^{\ka - 2j}). 
		\end{equation}
		Consider $(\xi_1 + \vep \xi_2)^{r} - (\vep\xi_2)^{r}$ for each $r = \ka - 2j \in [1, \ka]$. Letting $\xi = \xi_1$ and $\eta = \xi_1 + 2 \vep \xi_2$, we have
		\begin{equation}\label{eq:r}
			(\xi_1 + \vep \xi_2)^r - (\vep\xi_2)^r = \frac{(\eta + \xi)^r - (\eta - \xi)^r}{2^r} = \frac{1}{2^r} \sum_{j=0}^r \binom{r}{j} (\eta^{r - j} \xi^j + (-1)^{j+1} \eta^{r - j} \xi^{j}). 
		\end{equation}
		If $r = 2p$ is even, we have
		\begin{align}\label{even}
			&\sum_{j=0}^r \binom{r}{j} (\eta^{r - j} \xi^j + (-1)^{j+1} \eta^{r - j} \xi^{j}) = 2\sum_{j=0}^{p-1} \binom{r}{2j+1} \eta^{2p - (2j+1)} \xi^{2j+1} \notag \\
			&= 2 \xi \eta \sum_{j=0}^{p-1} \binom{r}{2j+1} (\eta^2)^{p - j - 1} (\xi^2)^{j} = 2 \xi \eta Q_r(\xi^2, \eta^2). 
		\end{align}
		If $r = 2p+1$ is odd, we have
		\begin{align}\label{odd}
			&\sum_{j=0}^r \binom{r}{j} (\eta^{\kappa - j} \xi^j + (-1)^{j+1} \eta^{r - j} \xi^{j}) = 2\sum_{j=0}^{p} \binom{r}{2j+1} \eta^{2p - 2j} \xi^{2j+1} \notag \\
			&= 2 \xi \sum_{j=0}^{p} \binom{r}{2j+1} (\eta^2)^{p - j} (\xi^2)^{j} = 2 \xi Q_r(\xi^2, \eta^2). 
		\end{align}
		Plugging \cref{even,odd} into \cref{eq:r} combined with \cref{Phi} completes the proof. 
	\end{proof}

    With \cref{lem:phase}, we can prove \cref{cor:phase}. Recall that $\varsigma(\ka) = 1$ if $\ka$ is even and $\varsigma(\ka) = 0$ if $\ka$ is odd. 
	% \begin{corollary}\label{cor:phase}
	% 	If $|\xi_1| \geq C_0 \vep$ or $|\xi_1+2\vep\xi_2| \geq C_0 \vep$ with $C_0$ sufficiently large depending on $C_1$ and $\ka$, then
	% 	\begin{equation*}
	% 		|\vep^{\ka - \alpha} \Phi(\xi_1, \xi_2)| \gtrsim |\xi_1(\xi_1 + 2 \vep \xi_2)^{\sigma(\ka)}|(\xi_1^{\ka - 1 - \sigma(\ka)} + (\xi_1 + 2 \vep \xi_2)^{\ka - 1 - \sigma(\ka)}). 
	% 	\end{equation*}
	% \end{corollary}
	
	\begin{proof}[Proof of \cref{cor:phase}]
        Let $C_1 \geq 1$ be a constant depending on $\{\tilde d_{\ka - 2j}\}$ and $\ka$ such that
	    \begin{equation}\label{eq:C1_def}
            |\tilde d_{\ka - 2j}| Q_{\ka - 2j}(\xi_1^2, (\xi_1 + 2 \vep \xi_2)^2) \leq C_1 (\xi_1^{\ka - 1 -\varsigma(\ka) - 2j} +  (\xi_1 + 2 \vep \xi_2)^{\ka - 1 -\varsigma(\ka) - 2j}), \qquad 2 \leq 2j <\ka. 
	    \end{equation}
		By \cref{lem:phase,eq:C1_def}, we have
		\begin{align}\label{eq:lowbound}
			|\vep^{\ka - \alpha} \Phi| 
			&= |\xi_1(\xi_1 + 2 \vep \xi_2)^{\varsigma(\ka)}|\big| \tilde d_{\ka} Q_\ka(\xi_1^2, (\xi_1 + 2 \vep \xi_2)^2) + \sum_{2 \leq 2j < \ka} \vep^{2j} \tilde d_{\ka - 2j} Q_{\ka - 2j}(\xi_1^2, (\xi_1 + 2 \vep \xi_2)^2) \big| \notag \\
			&\gtrsim |\xi_1(\xi_1 + 2 \vep \xi_2)^{\varsigma(\ka)}|\big( \xi_1^{\ka - 1 - \varsigma(\ka)} + (\xi_1 + 2 \vep \xi_2)^{\ka - 1 - \varsigma(\ka)} - \sum_{2 \leq 2j < \ka} \vep^{2j} |\tilde d_{\ka - 2j}| Q_{\ka - 2j}(\xi_1^2, (\xi_1 + 2 \vep \xi_2)^2)\big) \notag \\
			&\geq |\xi_1(\xi_1 + 2 \vep \xi_2)^{\varsigma(\ka)}|\big(\xi_1^{\ka - 1 - \varsigma(\ka)} + (\xi_1 + 2 \vep \xi_2)^{\ka - 1 - \varsigma(\ka)} \notag \\
			&\hspace{10em} - C_1\sum_{2 \leq 2j < \ka} \vep^{2j} (\xi_1^{\ka - 1 -\varsigma(\ka) - 2j} +  (\xi_1 + 2 \vep \xi_2)^{\ka - 1 -\varsigma(\ka) - 2j})\big).  
		\end{align}
		By the symmetry of $\xi_1$ and $\xi_1 + 2 \vep \xi_2$, in the following, we only present the proof assuming $|\xi_1| \geq C_0 \vep$ with $C_0$ large enough. Let $K:=(\ka - 1 - \varsigma(\ka))/2$ be the number of terms in the summation $\sum_{2 \leq 2j <\ka}$. We consider two cases. If we have, in addition, $|\xi_1 + 2 \vep \xi_2| \geq 2 K C_1 \vep$, then \cref{eq:lowbound} implies 
		\begin{align}
			|\vep^{\ka - \alpha} \Phi| 
			&\gtrsim |\xi_1(\xi_1 + 2 \vep \xi_2)^{\varsigma(\ka)}|\big(\frac{1}{2} \xi_1^{\ka - 1 - \varsigma(\ka)} + \frac{1}{2}(\xi_1 + 2 \vep \xi_2)^{\ka - 1 - \varsigma(\ka)} \big) \notag \\
			&\quad + |\xi_1(\xi_1 + 2 \vep \xi_2)^{\varsigma(\ka)}| \sum_{2 \leq 2j < \ka} \xi_1^{\ka - 1 -\varsigma(\ka) - 2j} \big( \frac{1}{2K}\xi_1^{2j} - C_1 \vep^{2j} \big) \notag \\
			&\quad+|\xi_1(\xi_1 + 2 \vep \xi_2)^{\varsigma(\ka)}|  \sum_{2 \leq 2j < \ka} (\xi_1 + 2 \vep \xi_2)^{\ka - 1 -\varsigma(\ka) - 2j} \big(\frac{1}{2K} (\xi_1 + 2 \vep \xi_2)^{2j} -  C_1 \vep^{2j} )\big) \notag \\
			&\geq |\xi_1(\xi_1 + 2 \vep \xi_2)^{\varsigma(\ka)}|\big(\frac{1}{2} \xi_1^{\ka - 1 - \varsigma(\ka)} + \frac{1}{2}(\xi_1 + 2 \vep \xi_2)^{\ka - 1 - \varsigma(\ka)} \big). 
		\end{align}
		If otherwise, $|\xi_1 + 2 \vep \xi_2| < 2 K C_1 \vep$, we get from \cref{eq:lowbound}
		\begin{align}
			|\vep^{\ka - \alpha} \Phi| 
			&\gtrsim |\xi_1(\xi_1 + 2 \vep \xi_2)^{\varsigma(\ka)}|\big(\frac{1}{2} \xi_1^{\ka - 1 - \varsigma(\ka)} + (\xi_1 + 2 \vep \xi_2)^{\ka - 1 - \varsigma(\ka)}-C_1 \sum_{2 \leq 2j < \ka}  \vep^{2j} (\xi_1 + 2 \vep \xi_2)^{\ka - 1 -\varsigma(\ka) - 2j} \big) \notag \\
			&\geq |\xi_1(\xi_1 + 2 \vep \xi_2)^{\varsigma(\ka)}|\big(\frac{1}{4} \xi_1^{\ka - 1 - \varsigma(\ka)} + (\xi_1 + 2 \vep \xi_2)^{\ka - 1 - \varsigma(\ka)}\big) \notag \\
			&\quad + |\xi_1(\xi_1 + 2 \vep \xi_2)^{\varsigma(\ka)}|\sum_{2 \leq 2j < \ka} \big(\frac{1}{4K} \xi_1^{\ka - 1 - \varsigma(\ka)} - C_1 \vep^{2j} (\xi_1 + 2 \vep \xi_2)^{\ka - 1 -\varsigma(\ka) - 2j} \big) \notag \\
			&\geq |\xi_1(\xi_1 + 2 \vep \xi_2)^{\varsigma(\ka)}|\big(\frac{1}{4} \xi_1^{\ka - 1 - \varsigma(\ka)} + (\xi_1 + 2 \vep \xi_2)^{\ka - 1 - \varsigma(\ka)}\big), 
		\end{align}
		where we use in the last inequality that, when $C_0^{\ka-1-\varsigma(\ka)} \geq 4KC_1(2KC_1)^{\ka-1-\varsigma(\ka)-2j}$, 
		\begin{align}
			&\sum_{2 \leq 2j < \ka} \big(\frac{1}{4K} \xi_1^{\ka - 1 - \varsigma(\ka)} - C_1 \vep^{2j} (\xi_1 + 2 \vep \xi_2)^{\ka - 1 -\varsigma(\ka) - 2j} \big) \notag \\
			&\quad \geq \sum_{2 \leq 2j < \ka} \big(\frac{1}{4K} \xi_1^{\ka - 1 - \varsigma(\ka)} - \vep^{\ka - 1 -\varsigma(\ka)} C_1 (2KC_1)^{\ka - 1 -\varsigma(\ka) - 2j} \big) \geq 0. 
		\end{align}
		Then we conclude the proof under the assumption that $|\xi_1| \geq C_0 \vep$ and complete the proof by the symmetric roles of $\xi_1$ and $\xi_1 + 2 \vep \xi_2$ in \cref{eq:lowbound}. 
	\end{proof}

    \subsection{Proof of \cref{phase_diff_j_est,phase_diff_j}}\label{append:phase_more}
        For \cref{phase_diff_j_est}, we have (recalling \cref{lem:phase})
		\begin{equation}\label{eq:phase_diff_j}
			\xi^j - \xi_2^j = \frac{1}{\vep^j} \frac{1}{2^{j-1}} Q_j(\xi_1^2, (\xi_1 + 2\vep \xi_2)^2) \times \left\{
			\begin{aligned}
				&\xi_1(\xi_1+2\vep\xi_2), && j \text{ even}, \\
				&\xi_1, && j \text{ odd}, 
			\end{aligned}
			\right.
		\end{equation}
		where $Q_1(x, y) := 1$. Note that when $|\xi_1+2\vep \xi_2| <\delta$, we have
		\begin{equation}\label{xi1}
			|\xi_1| \leq |\xi_1 + 2\vep \xi_2| + 2 \vep |\xi_2| \leq \delta + 2\vep|\xi_2|. 
		\end{equation}
		Then, when $|\xi_1+2\vep \xi_2| <\delta$, using \cref{eq:phase_diff_j,xi1,eq:Q}, we have for even $j = 2p$, 
		\begin{align}
			\vep^{j}|\xi^j - \xi_2^j| \lesssim |\xi_1| \delta Q_j(|\xi_1|^2, \delta^2) \lesssim |\xi_1| \delta(|\xi_1|^{2p - 2} + |\delta|^{2p - 2}) \lesssim |\xi_1|^{2p} + \delta^{2p} \lesssim \delta^{2p} + \vep^{2p} |\xi_2|^{2p}. 
		\end{align}
		Similarly, for odd $j = 2p+1$, we have
		\begin{equation}
			\vep^{j}|\xi^j - \xi_2^j| \lesssim |\xi_1| Q_j(|\xi_1|^2, \delta^2) \lesssim |\xi_1|(|\xi_1|^{2p} +  \delta^{2p}) \lesssim |\xi_1|^{2p+1} + \delta^{2p+1} \lesssim \delta^{2p+1} + \vep^{2p+1} |\xi_2|^{2p+1}. 
		\end{equation}
        Then we obtain \cref{phase_diff_j_est}. 

        For \cref{phase_diff_j}, recalling \cref{eq:phase_diff_j} and $\varsigma(p) = 1$ if $p$ is even and $\varsigma(p) = 0$ if $p$ is odd, using \cref{cor:phase}, 
		\begin{align}
			\left|\frac{\xi^j - \xi_2^j}{\Phi(\xi_1, \xi_2)}\right| 
			&\lesssim \vep^{\kappa - j - \alpha} \frac{|\xi_1+2\vep\xi_2|^{\varsigma(j)}(|\xi_1|^{j - 1 - \varsigma(j)} + |\xi_1+2\vep\xi_2|^{j - 1 - \varsigma(j)})}{|\xi_1+2\vep\xi_2|^{\varsigma(\kappa)}(|\xi_1|^{\kappa - 1 - \varsigma(\kappa)} + |\xi_1+2\vep\xi_2|^{\kappa - 1 - \varsigma(\kappa)})} \notag \\
			&= \frac{\vep^{\kappa - j - \alpha}}{|\xi_1+2\vep\xi_2|^{\kappa - j}} \frac{|\xi_1+2\vep\xi_2|^{\kappa - j + \varsigma(j) - \varsigma(\kappa)}(|\xi_1|^{j - 1 - \varsigma(j)} + |\xi_1+2\vep\xi_2|^{j - 1 - \varsigma(j)})}{|\xi_1|^{\kappa - 1 - \varsigma(\kappa)} + |\xi_1+2\vep\xi_2|^{\kappa - 1 - \varsigma(\kappa)}} \notag \\
			&\lesssim \frac{\vep^{\kappa - j - \alpha}}{|\xi_1+2\vep\xi_2|^{\kappa - j}} \frac{|\xi_1|^{\kappa - 1 - \varsigma(\kappa)} + |\xi_1+2\vep\xi_2|^{\kappa - 1 - \varsigma(\kappa)}}{|\xi_1|^{\kappa - 1 - \varsigma(\kappa)} + |\xi_1+2\vep\xi_2|^{\kappa - 1 - \varsigma(\kappa)}} = \frac{\vep^{\kappa - j - \alpha}}{|\xi_1+2\vep\xi_2|^{\kappa - j}}. 
		\end{align}
    
    \subsection{Derivation of \cref{eq:duhamel_discrete_compact}}\label{append:B}
	From \cref{eq:duhamel_summation} when $m=1$, we have 
	\begin{align}
		\psi^n_1
		&= \sum_{k_1=0}^{n-1} \int_{z_{k_1}}^{z_{k_1+1}} e^{-i s_1 \vep^\alpha D_\kappa} R_\vep e^{i \fl{s_1} \vep^\alpha D_\kappa} \psi^0 \rmd s_1 \notag \\
		&= \sum_{k_1=0}^{n-1} \int_{z_{k_1}}^{z_{k_1+1}} e^{-i s_1 \vep^\alpha D_\kappa} R_\vep \mathfrak{u}_\rmd(s_1) \rmd s_1 = \int_{0}^{z_n} e^{-i s_1 \vep^\alpha D_\kappa} R_\vep \mathfrak{u}_\rmd(s_1) \rmd s_1.
	\end{align}
	When $ m \geq 2$, for the last two integral in \cref{eq:duhamel_summation}, we have
	\begin{align}
		&\sum_{k_{m-1}=0}^{k_{m-2}-1} \int_{z_{k_{m-1}}}^{z_{k_{m-1}+1}} e^{-i s_{m-1} \vep^\alpha D_\kappa} R_\vep e^{i \fl{s_{m-1}} \vep^\alpha D_\kappa} \sum_{k_{m}=0}^{k_{m-1}-1} \int_{z_{k_m}}^{z_{k_m+1}} e^{-i s_m \vep^\alpha D_\kappa} R_\vep e^{i \fl{s_m} \vep^\alpha D_\kappa}  \psi^0 \rmd s_m \rmd s_{m-1} \notag \\
		&= \sum_{k_{m-1}=0}^{k_{m-2}-1} \int_{z_{k_{m-1}}}^{z_{k_{m-1}+1}} e^{-i s_{m-1} \vep^\alpha D_\kappa} R_\vep e^{i \fl{s_{m-1}} \vep^\alpha D_\kappa} \int_{0}^{z_{k_{m-1}}} e^{-i s_m \vep^\alpha D_\kappa} R_\vep \mathfrak{u}_\rmd(s_m) \rmd s_m \rmd s_{m-1} \notag \\
		&= \sum_{k_{m-1}=0}^{k_{m-2}-1} \int_{z_{k_{m-1}}}^{z_{k_{m-1}+1}} e^{-i s_{m-1} \vep^\alpha D_\kappa} R_\vep e^{i \fl{s_{m-1}} \vep^\alpha D_\kappa} \int_{0}^{\fl{s_{m-1}}} e^{-i s_m \vep^\alpha D_\kappa} R_\vep \mathfrak{u}_\rmd(s_m) \rmd s_m \rmd s_{m-1} \notag \\
		&= \sum_{k_{m-1}=0}^{k_{m-2}-1} \int_{z_{k_{m-1}}}^{z_{k_{m-1}+1}} e^{-i s_{m-1} \vep^\alpha D_\kappa} R_\vep (\mathcal{I}_d \mathfrak{u}_\rmd)(s_{m-1}) \rmd s_{m-1} \notag \\
		&= \int_{0}^{z_{k_{m-2}}} e^{-i s_{m-1} \vep^\alpha D_\kappa} R_\vep (\mathcal{I}_d \mathfrak{u}_\rmd)(s_{m-1}) \rmd s_{m-1}.
	\end{align}
	Then we have from \cref{eq:duhamel_summation}, 
	\begin{align}
		\psi^n_m 
		&= \sum_{k_1=0}^{n-1} \int_{z_{k_1}}^{z_{k_1+1}} e^{-i s_1 \vep^\alpha D_\kappa} R_\vep e^{i \fl{s_1} \vep^\alpha D_\kappa} \cdots \int_{0}^{z_{k_{m-2}}} e^{-i s_{m-1} \vep^\alpha D_\kappa} R_\vep (\mathcal{I}_d \mathfrak{u}_\rmd)(s_{m-1}) \rmd s_{m-1} \cdots \rmd s_1 \notag \\
		&= \sum_{k_1=0}^{n-1} \int_{z_{k_1}}^{z_{k_1+1}} e^{-i s_1 \vep^\alpha D_\kappa} R_\vep (\mathcal{I}^{m-1}_\rmd \mathfrak{u}_\rmd) (s_1) \rmd s_1 \notag \\
		&= \int_{0}^{z_n} e^{-i s_1 \vep^\alpha D_\kappa} R_\vep (\mathcal{I}^{m-1}_\rmd \mathfrak{u}_\rmd) (s_1) \rmd s_1, 
	\end{align}
	which completes the proof. 

    \subsection{An integral estimate used in the proof of \cref{prop:est1_discrete}.}
    % We have the following corollary when replacing $y^\ka$ with $P(y)$ with $P$ given by \cref{eq:P_def}. 
    \begin{lemma}\label{lem:int_aux}
    	If $\widehat {\phi} \in L^1(\R) \cap L^\infty(\R)$, we have 
    	\begin{equation}\label{eq:int_aux_1}
    		\int_{\R} \frac{|\widehat{\phi}(\xi_2)|}{|1 + \frac{a\vep^\ka}{u} P(\xi_2)|^\frac{\ka-1}{\ka}} \rmd \xi_2 \lesssim (a \vep^\ka)^{-\frac{\ka-1}{\ka^2}} |u|^\frac{\ka-1}{\ka^2}.   
    	\end{equation}
    	Moreover, assuming, in addition, $\int_{|\xi_2| \geq L} |\widehat{\phi}(\xi_2)| \rmd \xi_2 \lesssim \frac{1}{L^{\ka-1}}$ for all $L$ large enough, we have
    	\begin{equation}\label{eq:int_aux_2}
    		\int_{\R} \frac{|\widehat{\phi}(\xi_2)|}{|1 + \frac{a\vep^\ka}{u} P(\xi_2)|^\frac{\ka-1}{\ka}} \rmd \xi_2 \lesssim 1.   
    	\end{equation}
    \end{lemma}
    
    \begin{proof}
    	Note that when $\frac{a\vep^\ka}{u} P(\xi_2) \geq 0$, the integral in this regime can be simply bounded by $\| \widehat{\phi} \|_{L^1}$. Hence, it suffices to consider
    	\begin{equation}
    		\int_{\R} 1_{P(\xi_2)/u \leq 0}\frac{|\widehat{\phi}(\xi_2)|}{|1 + \frac{a\vep^\ka}{u} P(\xi_2)|^\frac{\ka-1}{\ka}} \rmd \xi_2 \leq \int_{\R} \frac{|\widehat{\phi}(\xi_2)|}{|1 - b |P(\xi_2)||^\frac{\ka-1}{\ka}} \rmd \xi_2, \quad b:=\frac{a\vep^\ka}{|u|} = \frac{\tau \vep^\alpha}{|u|} \leq \tau_0. 
    	\end{equation}
    	There exists some $C_P>0$ sufficiently large depending on $\{d_{\ka - 2j}\}$ such that $P(\xi_2) \sim \xi_2^\ka$ and $P'(\xi_2) \sim \xi_2^{\ka-1}$  when $|\xi_2| \geq C_P$. Choosing  $\tau_0$ sufficiently small such that $\tau_0|P(\xi_2)| \leq 1/2$ for all $|\xi_2| \leq C_P$,  we have
    	\begin{equation}
    		\int_{\R} 1_{|\xi_2| \leq C_P} \frac{|\widehat{\phi}(\xi_2)|}{|1 - b |P(\xi_2)||^\frac{\ka-1}{\ka}} \rmd \xi_2 \lesssim \| \widehat{\phi} \|_{L^1}.
    	\end{equation}
    	Thus the estimate reduces to 
    	\begin{equation}
    		I:=\int_{\R} 1_{|\xi_2| \geq C_P} \frac{|\widehat{\phi}(\xi_2)|}{|1 - b |P(\xi_2)||^\frac{\ka-1}{\ka}} \rmd \xi_2. 
    	\end{equation}
    	We further decompose $I$ as
    	\begin{equation}
    		I = \int_{|\xi_2| \geq C_P} (1_{b |P(\xi_2)| > \frac{3}{2}} + 1_{b |P(\xi_2)| < \frac{1}{2}} + 1_{\frac{1}{2} \leq b |P(\xi_2)| \leq \frac{3}{2}}) \frac{|\widehat{\phi}(\xi_2)|}{|1 - b |P(\xi_2)||^\frac{\ka-1}{\ka}} \rmd \xi_2 =: I_{1} + I_{2} + I_{3},  
    	\end{equation}
    	where $I_1$ and $I_2$ can be simply bounded by $\| \widehat \phi \|_{L^1}$. In the following, we estimate $I_3$, which covers the singular regime. We define a set $\Omega:=\{\xi_2\in\R: |\xi_2| \geq C_P, \,  \frac{1}{2} \leq b |P(\xi_2)| \leq \frac{3}{2}\}$. 
    	Since $P(\xi_2) \sim \xi_2^\ka$ when $\xi_2 \geq C_P$, for any $\xi_2 \in \Omega$, we have $|\xi_2| \sim b^{- \frac{1}{\ka}}$, and $|\xi_2| \geq 2C_P$ by letting $\tau_0$ be small enough. Let $\xi_\ast \in \Omega$ be the solution to $b |P(\xi_\ast)| = 1$ and thus $|\xi_\ast| \sim b^{- \frac{1}{\ka}}$, and $|\xi_\ast| \geq 2C_P$. We then have (choosing the $\xi_\ast$ that has the same sign as $\xi_2$ if $\ka$ is even), 
    	\begin{align}\label{I3_est}
    		I_{3} 
    		&= b^{-\frac{\ka-1}{\ka}} \int_{\Omega} \frac{|\widehat{\phi}(\xi_2)|}{|P(\xi_\ast) - P(\xi_2)|^\frac{\ka-1}{\ka}} \rmd \xi_2 
    		=b^{-\frac{\ka-1}{\ka}} \int_{\Omega} \frac{|\widehat{\phi}(\xi_2)|}{|P'(\xi)|^\frac{\ka-1}{\ka}|\xi_\ast - \xi_2|^\frac{\ka-1}{\ka}} \rmd \xi_2  \notag \\
    		&\lesssim b^{-\frac{\ka-1}{\ka}} \int_{\Omega} \frac{|\widehat{\phi}(\xi_2)|}{b^{-(\frac{\ka-1}{\ka})^2}|\xi_\ast - \xi_2|^\frac{\ka-1}{\ka}} \rmd \xi_2 
    		= b^{-\frac{\ka-1}{\ka^2}} \int_{\Omega} \frac{|\widehat{\phi}(\xi_2)|}{|\xi_\ast - \xi_2|^\frac{\ka-1}{\ka}}  \rmd \xi_2 \notag \\
    		&\lesssim b^{-\frac{\ka-1}{\ka^2}} \| \widehat \phi \|_{L^\infty}^\frac{\ka-1}{\ka} (\int_{\Omega} |\widehat{\phi}(\xi_2)| \rmd \xi_2)^\frac{1}{\ka}, 
    	\end{align}
    	where we use the mean-value theorem with the intermediate point $|\xi| \sim b^{-\frac{1}{\ka}}$, and, in the last inequality, 
    	\begin{equation*}
    		\int_{\Omega} \frac{|\widehat{\phi}(\xi_2)|}{|\xi_\ast - \xi_2|^\frac{\ka-1}{\ka}} \rmd \xi_2 = \int_{\Omega} (1_{|\xi_\ast - \xi_2| \leq \delta} + 1_{|\xi_\ast - \xi_2| \geq \delta}) \frac{|\widehat{\phi}(\xi_2)|}{|\xi_\ast - \xi_2|^\frac{\ka-1}{\ka}} \rmd \xi_2 \lesssim \| \widehat{\phi} \|_{L^\infty} \delta^\frac{1}{\ka} + \frac{\| \widehat{\phi} \|_{L^1(\Omega)}}{\delta^\frac{\ka-1}{\ka}}, \quad \delta := \frac{\| \widehat{\phi} \|_{L^1(\Omega)}}{\| \widehat{\phi} \|_{L^\infty}}. 
    	\end{equation*}
    	As a result, if $\widehat{\phi} \in L^1(\R)$, 
    	\begin{equation}\label{eq:est_decay_1}
    		(\int_{\Omega}  |\widehat{\phi}(\xi_2)| \rmd \xi_2)^\frac{1}{\ka} \leq \| \widehat{\phi} \|_{L^1}^\frac{1}{\ka}, 
    	\end{equation}
    	and if $\int_{|\xi_2| \geq L} |\widehat{\phi}(\xi_2)| \rmd \xi_2 \lesssim \frac{1}{L^{\ka-1}}$, 
    	\begin{equation}\label{eq:est_decay_2}
    		(\int_{\Omega} |\widehat{\phi}(\xi_2)| \rmd \xi_2)^\frac{1}{\ka} \lesssim (\int_{|\xi_2| \gtrsim b^{-\frac{1}{\ka}}} |\widehat{\phi}(\xi_2)| \rmd \xi_2)^\frac{1}{\ka} \lesssim b^{\frac{\ka-1}{\ka^2}}. 
    	\end{equation}
    	Plugging \cref{eq:est_decay_1} into \cref{I3_est} concludes the proof of \cref{eq:int_aux_1}, and plugging \cref{eq:est_decay_2} into \cref{I3_est} proves \cref{eq:int_aux_2}. 
    	
    \end{proof}

    \subsection{Proof of \cref{eq:g_low_bound}}\label{append:g_est}
        Recalling that $g(y) = f(y, 0)$. Using \cref{lem:phase}, we have
        \begin{equation*}
            g(y) = \sum_{0 \leq 2j < \ka} \vep^{2j}\tilde d_{\ka - 2j} (\ka-2j) y^{\ka - 2j}. 
        \end{equation*}
        Then we have
        \begin{equation*}
			g(y) - g(x) 
			\geq \ka \tilde d_\ka (y^{\ka} - x^{\ka}) - \sum_{2 \leq 2j < \ka} \vep^{2j} |\tilde d_{\ka - 2j}| (\ka - 2j) (y^{\ka-2j} - x^{\ka - 2j}).
			% &= \ka \tilde d_\ka (y^{\ka} - x^{\ka}) - \sum_{2 \leq 2j < \ka} \vep^{2j} |\tilde d_{\ka - 2j}| (\ka - 2j) (y-x) \sum_{p=0}^{\ka - 2j - 1} y^{\ka-2j-1-p}x^{p} \ \gtrsim \  y^{\ka} - x^{\ka}, 
		\end{equation*}
        Since $y \geq C_0' \vep$ and $C_0'$ is large enough, we have %%% \gb{We have to be careful with the notation as we also want to consider $\vep$ of order $O(1)$}
		\begin{align*}
			&\sum_{2 \leq 2j < \ka} \vep^{2j} |\tilde d_{\ka - 2j}| (\ka - 2j) (y-x) \sum_{p=0}^{\ka - 2j - 1} y^{\ka-2j-1-p}x^{p} \notag \\
			&\leq \frac{1}{C_0'^{2}} \sum_{2 \leq 2j < \ka} |\tilde d_{\ka-2j}| (\ka-2j) (y - x)\sum_{p=0}^{\ka-2j - 1} y^{\ka - 1 - p}x^p \leq \frac{\ka \tilde d_\ka}{2}(y - x)\sum_{p=0}^{\ka- 1} y^{\ka - 1 - p}x^p = \frac{\ka \tilde d_\ka}{2}(y^\ka - x^\ka).
		\end{align*}
        The conclusion then follows immediately. 
	
	% \bibliography{references.bib}
	% \bibliographystyle{siam}

\section*{Acknowledgments}
This work was funded in part by NSF grant DMS-230641 and ONR grant N00014-26-1-2017.

\end{document}